\documentclass [12pt]{amsart}
\usepackage[utf8]{inputenc}
\pdfoutput=1

\usepackage{amsmath,amssymb,amsthm,amsfonts}
\usepackage{mathtools}%

\usepackage[main=english,french]{babel}%
\usepackage{comment}%
\usepackage[bookmarksdepth=4]{hyperref}%
\usepackage{bbm}%
\usepackage{mathrsfs}%

\usepackage{tikz,graphicx,color}
\usepackage{tikz-cd}%
\usepackage{tikz-3dplot}%
\usetikzlibrary{calc}%
\usetikzlibrary{arrows}%
\usetikzlibrary{shapes}%
\usetikzlibrary{patterns}%
\usetikzlibrary{positioning}%
\usetikzlibrary{arrows.meta}
\usetikzlibrary{knots}
\usetikzlibrary{decorations.markings}
\usepackage{epstopdf}%
\pdfpageattr{/Group <</S /Transparency /I true /CS /DeviceRGB>>}

\usepackage[arrow]{xy}%
\usepackage{diagbox}%
\usepackage{subfig}%
\usepackage{arcs}%
\usepackage{xcolor}%
\usepackage{tabu}
\usepackage{booktabs}%

\usepackage[margin=1in]{geometry}

\usepackage{soul}
\usepackage{accents}
\usepackage{enumerate}
\usepackage{letltxmacro}
\usepackage{thmtools,etoolbox}

\usepackage{nameref,hyperref}
\usepackage[capitalize]{cleveref}

\usepackage{tikz-cd}
\usepackage{stmaryrd}

\usepackage{xspace}

\usepackage{stmaryrd}
\usepackage{bm}

\newtheorem{theorem}{Theorem}[section]

\newtheorem{lemma}[theorem]{Lemma}
\newtheorem{proposition}[theorem]{Proposition}
\newtheorem{corollary}[theorem]{Corollary}
\newtheorem{conjecture}[theorem]{Conjecture}
\theoremstyle{definition}
\newtheorem{remark}[theorem]{Remark}
\theoremstyle{definition}
\newtheorem{definition}[theorem]{Definition}

\theoremstyle{definition}
\newtheorem{problem}[theorem]{Problem}
\theoremstyle{definition}
\newtheorem{example}[theorem]{Example}
\theoremstyle{definition}

\def\<{{\langle}}
\def\>{{\rangle}}

\def\la{{\lambda}}

\def\rk{{\operatorname{rk}}}
\def\cork{ \operatorname{cork}}

\def\Racc{R^\circ}
\def\Rich_#1^#2{\Racc_{#1,#2}}
\def\Richcl_#1^#2{R_{#1,#2}}

\def\bw{{\mathbf{w}}}
\def\bv{{\mathbf{v}}}

\def\bx{{\mathbf{x}}}

\def\xrasim{\xrightarrow{\sim}}

\def\w{{\mathbf{w}}}

\def\Spec{\operatorname{Spec}}

\def\HHH{H\!\!H\!\!H}
\def\HHHC{\HHH_\C}

\def\pt{{\rm pt}}

\def\F{{\mathbb{F}}}
\def\Q{{\mathbb{Q}}}

\def\tP{{\tilde P}}

\def\tP{{\tilde P}}

\def\Pio{\Pi^\circ}

\def\DA{{\Dscr_1}}

\def\Boundkn{\BND(k,n)}

\numberwithin{equation}{section}

\def\FLY{HOMFLY\xspace}

\def\Richafftp_#1^#2{{\Rcal_{#1,#2}^{>0}}}

\def\Richaff_#1^#2{\accentset{\circ}{\mathcal{R}}_{#1,#2}}

\def\Dyck{\operatorname{Dyck}}
\def\Dyckxx(#1,#2){\Dyck_{#1,#2}}

\def\ncyc{\operatorname{ncyc}}
\def\ncyc{c}

\def\fkn{{f_{k,n}}}
\def\Pit_#1{\Xcal^\circ_{#1}}

\def\k{\mathbbm{k}}

\def\top{{\operatorname{top}}}

\def\Povar_#1{\accentset{\circ}{\Pi}_{#1}}
\def\Povarcl_#1{\Pi_{#1}}
\def\RPovar_#1{\accentset{\circ}{\Pi}^\R_{#1}}
\def\RPovarcl_#1{\Pi^\R_{#1}}
\def\Povtp_#1{\Pi_{#1}^{>0}}
\def\Povtnn_#1{\Pi_{#1}^{\geq0}}

\def\Boundkn{\operatorname{Bound}(k,n)}

\def\Gr{\operatorname{Gr}}

\def\KLR_#1^#2{R_{#1,#2}(q)}

\crefname{figure}{Figure}{Figures}
\crefname{conjecture}{Conjecture}{Conjectures}

\addtotheorempostheadhook[theorem]{\crefalias{theoremlisti}{theorem}}
\addtotheorempostheadhook[lemma]{\crefalias{theoremlisti}{lemma}}
\addtotheorempostheadhook[proposition]{\crefalias{theoremlisti}{proposition}}
\addtotheorempostheadhook[corollary]{\crefalias{theoremlisti}{corollary}}

\def\Acal{\mathcal{A}}\def\Fcal{\mathcal{F}}\def\Rcal{\mathcal{R}}\def\Ucal{\mathcal{U}}\def\Xcal{\mathcal{X}}\def\Zcal{\mathcal{Z}}

\def\abf{\mathbf{a}}

\def\C{\mathbb{C}}
\def\R{\mathbb{R}}

\def\Z{\mathbb{Z}}
\def\Q{\mathbb{Q}}

\def\<{{\langle}}
\def\>{{\rangle}}

\def\la{{\lambda}}

\def\rk{{\mathrm{rank}}}
\def\cork{ \operatorname{corank}}

\def\Gr{\operatorname{Gr}}

\def\n{{\mathfrak n}}

\def\Z{{\mathbb Z}}
\def\R{{\mathbb R}}
\def\Gr{{\rm Gr}}

\newcounter{todobackgr}[section]

\newcounter{todofigure}[section]

\def\X{\accentset{\circ}{X}}

\makeatletter
\DeclareRobustCommand{\cev}[1]{%
  \mathpalette\do@cev{#1}%
}
\newcommand{\do@cev}[2]{%
  \fix@cev{#1}{+}%
  \reflectbox{$\m@th#1\vec{\reflectbox{$\fix@cev{#1}{-}\m@th#1#2\fix@cev{#1}{+}$}}$}%
  \fix@cev{#1}{-}%
}
\newcommand{\fix@cev}[2]{%
  \ifx#1\displaystyle
    \mkern#23mu
  \else
    \ifx#1\textstyle
      \mkern#23mu
    \else
      \ifx#1\scriptstyle
        \mkern#22mu
      \else
        \mkern#22mu
      \fi
    \fi
  \fi
}

\makeatother

\def\Rich_#1^#2{\vec R^\circ_{#1,#2}}
\def\LRich_#1^#2{\cev R^\circ_{#1,#2}}
\def\RichL_#1^#2{\LRich_{#1}^{#2}}
\def\Rtp_#1^#2{\vec R_{#1,#2}^{>0}}
\def\LRtp_#1^#2{\cev R_{#1,#2}^{>0}}

\def\Rtnn_#1^#2{\vec R_{#1,#2}^{\geq0}}
\def\LRtnn_#1^#2{\cev R_{#1,#2}^{\geq0}}

\def\PR_#1^#2{\accentset{\circ}{\Pi}_{#1,#2}}%
\def\PRtp_#1^#2{\Pi_{#1,#2}^{>0}}%
\def\PRtnn_#1^#2{\Pi_{#1,#2}^{\geq0}}%
\def\PRcl_#1^#2{\Pi_{#1,#2}}%
\def\PRR_#1^#2{\accentset{\circ}{\Pi}_{#1,#2}^\R}%
\def\PRRcl_#1^#2{\Pi_{#1,#2}^\R}%

\def\bw{{\mathbf{w}}}
\def\bv{{\mathbf{v}}}

\def\Fcal{\mathcal{F}}

\def\bx{{\bm{x}}}

\def\hjmap{\kappa}
\def\hjmp_#1{\hjmap_{#1}}

\def\Uom_#1{U^{\diamond,-}_{#1}}

\def\xrasim{\xrightarrow{\sim}}

\def\Richaff_#1^#2{\accentset{\circ}{\mathcal{R}}_{#1}^{#2}}

\def\Povar_#1{\accentset{\circ}{\Pi}_{#1}}
\def\Povarcl_#1{\Pi_{#1}}
\def\RPovar_#1{\accentset{\circ}{\Pi}^\R_{#1}}
\def\RPovarcl_#1{\Pi^\R_{#1}}
\def\Povtp_#1^#2{\Pi_{#1,#2}^{>0}}
\def\Povtnn_#1{\Pi_{#1}^{\geq0}}

\def\Boundkn{\operatorname{Bound}(k,\n)}

\def\Star_#1{\operatorname{Star}_{#1}}
\def\Startnn_#1{\operatorname{Star}^{\geq0}_{#1}}

\def\pB{\partial B}

\def\Link{\operatorname{Lk}}
\def\Lkx_#1{\Link_{#1}}
\def\Lkxx_#1^#2{\accentset{\circ}{\Link}_{#1}^{#2}}

\def\Lktxx_#1^#2{\Link^{>0}_{#1,#2}}
\def\Starxx_#1^#2{\operatorname{Star}_{#1,#2}}
\def\Startxx_#1^#2{\operatorname{Star}^{\geq0}_{#1,#2}}

\def\sctnn_#1{\sc^{\geq0}_{#1}}

\def\sctp_#1^#2{\sc^{>0}_{#1,#2}}

\def\eps{\varepsilon}
\def\Seps_#1{S_{#1}}

\def\Lktpe_#1^#2{\Link^{>0}_{#1,#2}}
\def\Lktnne_#1{\Link^{\geq0}_{#1}}

\def\Lktp_#1^#2{\Link^{>0}_{#1,#2}}
\def\Lktnn_#1{\Link^{\geq0}_{#1}}

\def\sc{Z}

\def\sco_#1^#2{\accentset{\circ}{\sc}_{#1,#2}}

\def\sccl_#1^#2{\sc_{#1}^{#2}}

\def\Y{\mathcal{Y}}
\def\Yo_#1{\accentset{\circ}{\Y}_{#1}}
\def\Ycl_#1{\Y_{#1}}

\def\Ytp_#1{\Y_{#1}^{>0}}

\def\strg(#1){\normg{#1}}

\def\normg#1{\|#1\|}

\def\Spec{\operatorname{Spec}}
\def\int{{\operatorname{init}}}

\def\DOM^#1_#2{\Delta^{#1 \omega_i}_{#2 \omega_i}}
\def\DOMr^#1_#2{\Delta^{#1 \omega_r}_{#2 \omega_r}}
\def\DOMir^#1_#2{\Delta^{#1 \omega_{i_r}}_{#2 \omega_{i_r}}}

\def\lw{\line{w}}

\def\sv{s^{\mathbf{v}}}

\def\RLtp_#1^#2{\cev R_{#1,#2}^{>0}}
\def\Rsf_#1^#2{\Rich_{#1}^{#2}(K)}
\def\LRsf_#1^#2{\LRich_{#1}^{#2}(K)}

\def\pre{{\,\operatorname{pre}}}
\def\tpre_#1^#2{\vec\twistop^\pre_{#1,#2}}
\def\Ltpre_#1^#2{\cev\twistop^\pre_{#1,#2}}
\def\tpreL_#1^#2{\Ltpre_{#1}^{#2}}
\def\twistop{\tau}
\def\twist_#1^#2{\vec\twistop_{#1,#2}}
\def\twistL_#1^#2{\cev\twistop_{#1,#2}}

\def\Sn{S_\n}

\def\sv_#1{s^{\bv}_{#1}}

\makeatletter
\newcommand{\xMapsto}[2][]{\ext@arrow 0599{\Mapstofill@}{#1}{#2}}
\def\Mapstofill@{\arrowfill@{\Mapstochar\Relbar}\Relbar\Rightarrow}
\makeatother

\makeatletter
\@namedef{subjclassname@2020}{%
  \textup{2020} Mathematics Subject Classification}
\makeatother

\def\Qkn{Q_{k,\n}}

\def\Skn{S_\n^{(k)}}

\def\perm{\operatorname{perm}}
\def\Lperm{L^{\perm}}
\def\tor{\operatorname{tor}}
\def\Ltor{L^{\tor}}
\def\grid{\operatorname{grid}}
\def\Lgrid{L^{\grid}}
\def\Lgridt{\tilde L^{\grid}}
\def\rich{\operatorname{Rich}}
\def\Lrich{L^{\rich}}
\def\plab{\operatorname{plab}}
\def\Lplab{L^{\plab}}
\def\cox{\operatorname{Cox}}
\def\Lcox{L^{\cox}}

\def\linksection#1{\subsubsection{#1}}
\def\fp{\pi}

\def\pig{\pi_G}

\def\arg{\operatorname{arg}}

\def\QG{Q_G}
\def\Qice{\widetilde{Q}}
\def\QGice{\Qice_G}

\def\BQ{B(Q)}
\def\BQice{\widetilde{B}(\Qice)}
\def\Vice{\widetilde{V}}

\def\AQG{\Acal(\QG)}
\def\AQGice{\Acal(\QGice)}
\def\RQ(#1){R(Q;#1)}
\def\RQX#1#2{R(#1;#2)}
\def\RQice(#1){R(\Qice;#1)}
\def\RQG(#1){R(\QG;#1)}
\def\F{\mathbb{F}}

\def\Pio{\Pi^\circ}
\def\HOMP{P}
\def\unkn{%
\begin{tikzpicture}[baseline=(ZUZU.base),scale=0.2]\coordinate(ZUZU) at (0,-0.6);
\draw[line width=0.7pt] (0,0) circle (1cm);
  \end{tikzpicture}
}

\def\top{{\operatorname{top}}}
\def\Ptop_#1(#2){P^\top(#1;#2)}

\def\fbar{\bar f}
\def\k{k}
\def\taukn{\tau_{k,\n}}
\def\T{\mathbb{T}}%

\def\Cproj{\overline{C}}

\def\fro{\operatorname{fro}}
\def\mut{\operatorname{mut}}
\def\Vfro{V_{\fro}}
\def\Vmut{V_{\mut}}
\def\Fcal{{\mathcal{F}}}

\def\Qicefr[#1]{\Qice[#1]}
\def\Bice{\widetilde{B}}
\def\bice{\tilde{b}}
\def\BiceX(#1){\widetilde{B}(#1)}

\def\BX(#1){B(#1)}
\def\X{\Xcal}
\def\XQice{\X(\Qice)}
\def\XX(#1){\X(#1)}
\def\XBQice{\X(\BQice)}
\def\AX(#1){\Acal(#1)}
\def\AQice{\Acal(\Qice)}
\def\ABQice{\Acal(\BQice)}

\def\chiQ{\chi_Q}

\def\Gm{\F^\ast}
\def\Ga{\F}

\def\Aut{\operatorname{Aut}}
\def\QGice{\Qice_G}
\def\QX(#1){Q_{#1}}
\def\disjun{\sqcup}

\def\pmn{\pm I}
\def\DRW{(\pmn)^m} %

\def\Gbw{G(\bw)}
\def\GX(#1){G(#1)}

\def\rev{\operatorname{rev}}

\def\figref#1(#2){Figure~\hyperref[#1]{\ref*{#1}(#2)}}

\def\conn{\operatorname{c}}
\def\ncyc{\conn}

\def\n{N}
\def\qq{q^{\frac12}}
\def\qqi{q^{-\frac12}}

\def\disk{\mathbf{D}}

\def\Lrichp{L}%

\def\projtx#1{\overline{#1}}

\def\Louise{locally acyclic\xspace}
\def\Louiseleaf{recurrent leaf\xspace} %

\def\degtop{\deg^{\top}}

\def\pmtI{\pm\tI}
\def\linkurl#1{\textup{\texttt{\href{http://katlas.org/wiki/#1}{#1}}}}
\def\comma{} %
\def\bmref#1{\hyperlink{#1}{\normalfont(B\ref*{#1})}}

\def\Btilde{$\tilde{\text{B}}$}

\def\abmref#1{\hyperref[#1]{\normalfont(\Btilde\ref*{#1})}}

\def\hopflink{
\scalebox{0.3}{%
\begin{tikzpicture}[baseline=(ZUZU.base)]\coordinate(ZUZU) at (0,-0.8);
\draw[line width=3.5pt, white,
    decoration={markings,mark=at position 0 with {\arrow[scale=1,black]{stealth'}}},
    decoration={markings,mark=at position 1 with {\arrow[scale=1,black]{stealth'}}},
    postaction={decorate}] (-0.99,-0.15)--(-1,-1)--(1,-1)--(0.99,0.15);
\draw[line width=3.5pt, white,
    decoration={markings,mark=at position 0.01 with {\arrow[scale=1,black]{stealth'}}},
    decoration={markings,mark=at position 0.99 with {\arrow[scale=1,black]{stealth'}}},
    postaction={decorate}] (1.99,-0.15)--(2,-1)--(0,-1)--(0.01,0.15);
\begin{knot}[
    clip width=6,
    flip crossing={1},
    ]
    \strand [line width=3.5pt] (0,0) circle (1.0cm);
    \strand [line width=3.5pt] (1,0) circle (1.0cm);
\end{knot}
\end{tikzpicture} 
}
}

\def\pikn{\pi_{k,\n}}
\def\fkn{f_{k,\n}}

\def\pA{^{(1)}}
\def\pB{^{(2)}}
\def\pC{^{(3)}}

\def\red{-}
\def\Qred{Q^{\red}}
\def\QredG{Q^{\red}_G}
\def\BQredG{B(\QredG)}

\def\E{H}%
\def\Tc{{T,c}}

\def\mysq{\begin{tikzpicture}
\node[draw=black,line width=1pt, rectangle,scale=0.5,fill=white] (A) at (0,0) {};
\end{tikzpicture}}

\def\HHH{H\!\!H\!\!H}
\def\HHHC{\HHH_\C}
\def\PGL{\operatorname{PGL}}

\def\mirr#1{#1^\ast}
\begin{document}

\title{Plabic links, quivers, and skein relations}

\author{Pavel Galashin}
\address{Department of Mathematics, University of California, Los Angeles, 520 Portola Plaza,
Los Angeles, CA 90025, USA}
\email{\href{mailto:galashin@math.ucla.edu}{galashin@math.ucla.edu}}

\author{Thomas Lam}
\address{Department of Mathematics, University of Michigan, 2074 East Hall, 530 Church Street, Ann Arbor, MI 48109-1043, USA}
\email{\href{mailto:tfylam@umich.edu}{tfylam@umich.edu}}
\thanks{P.G.\ was supported by an Alfred P. Sloan Research Fellowship and by the National Science Foundation under Grants No.~DMS-1954121 and No.~DMS-2046915. T.L.\ was supported by Grants No.~DMS-1464693 and No.~DMS-1953852 from the National Science Foundation.}



\date{\today}

\begin{abstract}
We study relations between cluster algebra invariants and link invariants.

First, we show that several constructions of \emph{positroid links} (permutation links, Richardson links, grid diagram links, plabic graph links)  give rise to isotopic links. For a subclass of permutations arising from concave curves, we also provide isotopies with the corresponding Coxeter links.

Second, we associate a \emph{point count polynomial} to an arbitrary \Louise quiver. We conjecture an equality between the top $a$-degree coefficient of the HOMFLY polynomial of a plabic graph link and the point count polynomial of its planar dual quiver. We prove this conjecture for leaf recurrent plabic graphs, which includes reduced plabic graphs and plabic fences as special cases.
\end{abstract}

\numberwithin{equation}{section}

\maketitle

\section{Introduction}
In recent years, intriguing connections between knot theory and the theory of cluster algebras have been developed.  Shende--Treumann--Williams--Zaslow~\cite{STWZ} found cluster structures on certain moduli spaces of sheaves microsupported on a Legendrian link.  Fomin--Pylyavskyy--Shushtin--Thurston~\cite{FPST} studied the relation between quiver mutation and morsifications of algebraic links.  Casals--Gorsky--Gorsky--Simental~\cite{CGGS} and Mellit~\cite{Mellit_Cell} studied braid varieties associated to positive braid words; cluster structures on these spaces are the subject of ongoing works \cite{CGGLSS,GLSBS1,GLSBS2}. Other connections can be found e.g. in~\cite{HiIn,Muller_skein,LeeSch,BMS,CW}.

The goal of this work is to compare knot invariants and cluster algebra invariants.  The starting point is our earlier work~\cite{GL_qtcat} on \emph{open positroid varieties} $\Povar_\fp$ \cite{KLS}, which are certain distinguished subvarieties of the Grassmannian equipped with a cluster structure \cite{GL_cluster} (see also~\cite{Scott,MSLA,Lec,SSBW}).  In \cite{GL_qtcat}, we associated a \emph{positroid link} $L_\fp$ to each open positroid variety $\Povar_\fp$, and constructed an isomorphism between the cohomology of $\Povar_\fp$ and the top $a$-degree part of the Khovanov--Rozansky triply-graded link homology~\cite{KR1,KR2,KhoSoe} of $L_\fp$.   In subsequent work~\cite{GL_cat_combin}, we further developed this connection by defining \emph{positroid Catalan numbers} as certain Euler characteristics of open positroid varieties.  This invariant was studied for a subclass of permutations arising from concave curves, using a Dyck path recursion that did not require explicit mention of the cluster structure on $\Povar_\fp$, or of the positroid link $L_\fp$.

The initial motivation of this work is to explain our positroid recursion in the broader setting of \emph{plabic (i.e., planar bicolored) graphs}, making an explicit connection to recursions for quivers and for links.  Plabic graphs were introduced by Postnikov~\cite{Pos} who also speculated on the relation with cluster algebras. 
 To an arbitrary plabic graph $G$ one can associate a \emph{plabic graph link} $\Lplab_G$ introduced in~\cite{FPST,STWZ}.
 On the other hand, the set of equivalence classes of \emph{reduced} plabic graphs is in bijection with open positroid varieties, to which we associated positroid links in~\cite{GL_qtcat}.

In the first part of this work, we establish (\cref{thm:isotopy}) that the various links associated to $\Povar_\fp$ are isotopic, including the positroid link $L_\fp$ and the corresponding plabic graph link $\Lplab_G$.  In the second part of this work, we study the relation between the point count function of a quiver $Q$ and the HOMFLY polynomial $\HOMP(L;a,z)$ of a link.  We introduce a class of \emph{simple} plabic graphs. Our main conjecture (\cref{conj:main}) states that for a simple plabic graph $G$, the point count of the planar dual quiver $Q_G$ is equal to the top $a$-degree part of the HOMFLY polynomial of $\Lplab_G$.  We show (\cref{thm:FLY=pcnt_leaf_rec}) that this conjecture holds for the class of \emph{leaf recurrent} plabic graphs, that includes reduced plabic graphs and plabic fences. 
We interpret and generalize the locally acyclic recursion for positroids~\cite{MSLA} in terms of the \FLY skein relation.

\subsection*{Acknowledgments}
The second author thanks David Speyer for the previous collaborations \cite{LS,LS2}, which have influenced this work.  The authors thank Eugene Gorsky for explanations related to \cite{CGGS}.  %
 
\section{Main results}
Let $\fp\in\Sn$ be a permutation. For simplicity, we assume that $\fp$ has no fixed points, i.e., $\fp(i)\neq i$ for all $i\in[\n]:=\{1,2,\dots,\n\}$. 

\begin{figure}
\setlength{\tabcolsep}{5pt}
\begin{tabular}{cccc}
\includegraphics[width=0.13\textwidth]{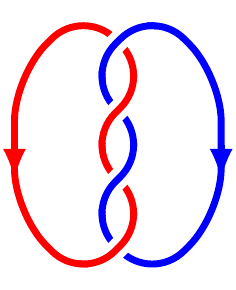} &
\includegraphics[width=0.27\textwidth]{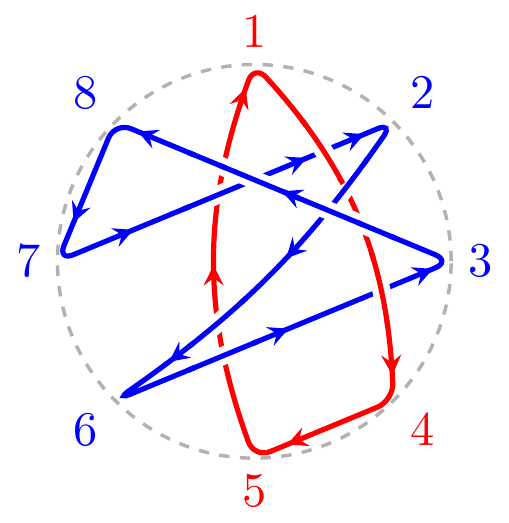} &
\includegraphics[width=0.24\textwidth]{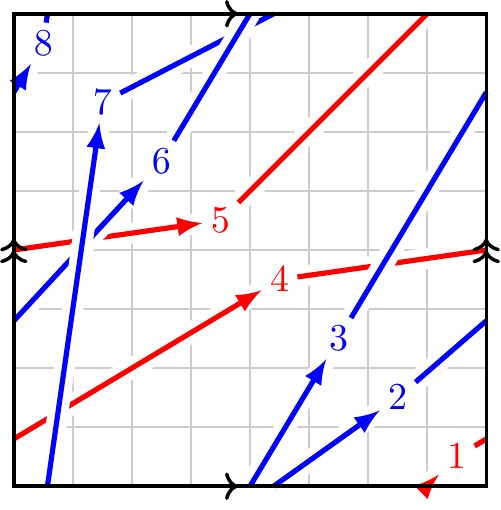}&
\hspace{0.05in}\includegraphics[width=0.25\textwidth]{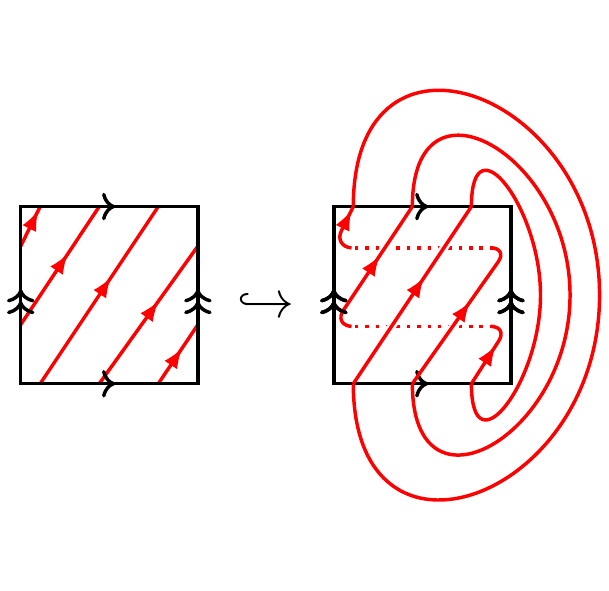}\\
  (a) \linkurl{L4a1} & (b) $\Lperm_\pi$ & (c) $\Ltor_\pi$ & (d) \hspace{0.05in}$\T\hookrightarrow\R^3$
\end{tabular}
  \caption{\label{fig:L4a1} The links \linkurl{L4a1}$\cong \Lperm_\pi\cong\Ltor_\pi$ for $\pi$ given by~\eqref{eq:intro:fp}.}
\end{figure}

In~\cite{GL_qtcat}, we described a way to associate a \emph{positroid link} $L_\fp$ to such $\fp$. The number of components of $L_\fp$ is given by the number $\ncyc(\fp)$ of cycles of $\fp$. Our running example will be
\begin{equation}\label{eq:intro:fp}
 \fp=\left(\!\!\text{
\begin{tabular}{cccccccc}
\textcolor{red}{$1$} & \textcolor{blue}{$2$} & \textcolor{blue}{$3$} & \textcolor{red}{$4$} & \textcolor{red}{$5$} & \textcolor{blue}{$6$} & \textcolor{blue}{$7$} & \textcolor{blue}{$8$}\\
\textcolor{red}{$4$} & \textcolor{blue}{$6$} & \textcolor{blue}{$8$} & \textcolor{red}{$5$} & \textcolor{red}{$1$} & \textcolor{blue}{$3$} & \textcolor{blue}{$2$} & \textcolor{blue}{$7$}
\end{tabular}
}\!\!\right),
\end{equation}
written in two-line notation, with different colors representing different cycles of $\fp$. We have $\n=8$ and $\ncyc(\fp)=2$. The associated positroid link $L_\fp$ is shown in \figref{fig:L4a1}(a); it appears under the name \linkurl{L4a1} in the Thistlethwaite Link Table~\cite{KAT}. 

\subsection{Positroid links} \label{sec:intro:positroid_links}

Our first result relates several different representations of the link $L_\fp$, confirming the conjectures of~\cite{GL_qtcat,GL_cat_combin}. We start by giving a brief overview of these descriptions; see \cref{sec:positr-comb,sec:positr-link-isot} for full details.

\linksection{Permutation link} \label{sec:intro:perm_links}
Draw $\n$ points labeled $1,2,\dots,\n$ on the circle in clockwise order. Draw a line segment connecting $i$ to $\fp(i)$, for each $i\in[\n]:=\{1,2,\dots,\n\}$. If two line segments $[i, \fp(i)]$, $[j,\fp(j)]$ cross for some $i,j\in[\n]$ such that $\fp(i)<\fp(j)$, then the segment $[j,\fp(j)]$ is drawn above $[i, \fp(i)]$; see \figref{fig:L4a1}(b). 
 The union of these segments gives rise to a link diagram. We refer to the resulting link as the \emph{permutation link} of $\fp$, denoted $\Lperm_\fp$. Orienting each line segment from $i$ to $\fp(i)$ induces an orientation on $\Lperm_\fp$.

\linksection{Toric permutation link} Let $\T:=\R^2/\Z^2$ be the torus. Consider its fundamental domain $[0,1]^2$ which we view as a square subdivided into $\n^2$ boxes of size $\frac1\n\times \frac1\n$. We denote by $b_{i,j}$, $1\leq i,j\leq \n$, the box whose upper right corner has coordinates $\frac1\n(i,j)$. For each $i\in[\n]$, place a dot $d_i$ in box $b_{i,\n+1-i}$. Draw an arrow (in $\T$) in the northeast direction from $d_i$ to $d_{\fp(i)}$ for each $i\in[\n]$. When two such arrows cross, the arrow with the higher slope is drawn above the arrow with the lower slope. We obtain a planar diagram of an oriented link $\Ltor_\fp$ drawn on the surface of $\T$; see \figref{fig:L4a1}(c). Embedding $\T$ inside $\R^3$ in a standard way (\figref{fig:L4a1}(d)), we may view $\Ltor_\fp$ as a link in $\R^3$. 

\begin{remark}
Viewing $\Ltor_\fp$ as drawn on the surface of $\T$ (resp., in a solid torus $S^1\times [0,1]$) contains more information---it gives rise to a certain elliptic Hall algebra element~\cite{ScVa,BuSc} (resp., to a symmetric 
function~\cite{Turaev}) with deep connections to Khovanov--Rozansky homology. We aim to pursue this direction in an upcoming paper.
\end{remark}

\begin{figure}
  \includegraphics[width=1.0\textwidth]{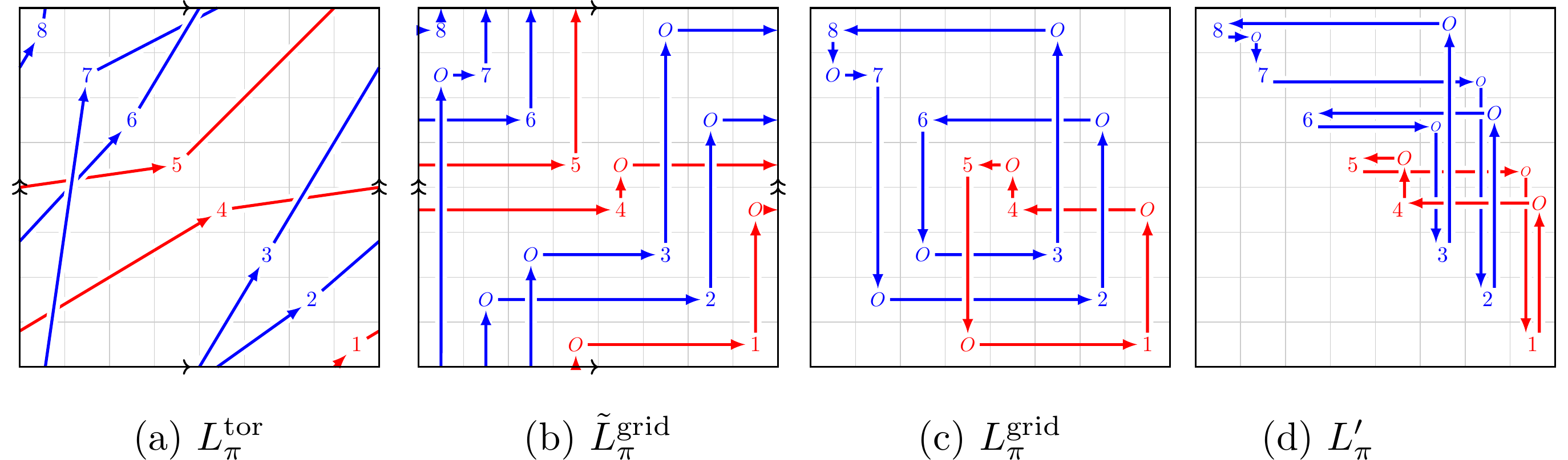}
 \caption{\label{fig:grid_iso} Toric grid links; see \cref{rmk:grid} and \cref{sec:perm-to-toric}.}
\end{figure}

\begin{remark}\label{rmk:grid}
We may also view $\Ltor_\fp$ as coming from a \emph{grid diagram} (see, e.g.,~\cite{OSS}). The grid diagram of $\fp$ is obtained by placing an $X$ in box $b_{i,\n+1-i}$ and an $O$ in box $b_{i,\n+1-\fp(i)}$ for each $i\in[\n]$. We then connect the $X$'s and the $O$'s by horizontal and vertical segments, always drawing the vertical segments above the horizontal ones. More precisely, to a given grid diagram, we can associate two links. First, a \emph{toric grid link} $\Lgridt_\pi$, drawn on the surface of $\T$, is obtained by drawing a horizontal arrow from $O$ to $X$ pointing right in each row, and a vertical arrow from $X$ to $O$ pointing up in each column; see \figref{fig:grid_iso}(b). (In \cref{fig:grid_iso}, for each $i\in[\n]$, we place an $i$ instead of an $X$ in box $b_{i,\n+1-i}$.) Second, a \emph{planar grid link} $\Lgrid_\pi$, whose planar diagram is drawn inside $[0,1]^2$, is obtained by drawing a horizontal arrow from $O$ to $X$ in each row, and a vertical arrow from $X$ to $O$ in each column, but in this case the arrows are not allowed to cross the boundaries of $[0,1]^2$; see \figref{fig:grid_iso}(c). As explained in~\cite[Section~3.2]{OSS}, these two links are isotopic to each other. It is easy to see that the toric grid link $\Lgridt_\pi$ is also isotopic to $\Ltor_\fp$.
\end{remark}

\begin{figure}
  \includegraphics[width=1.0\textwidth]{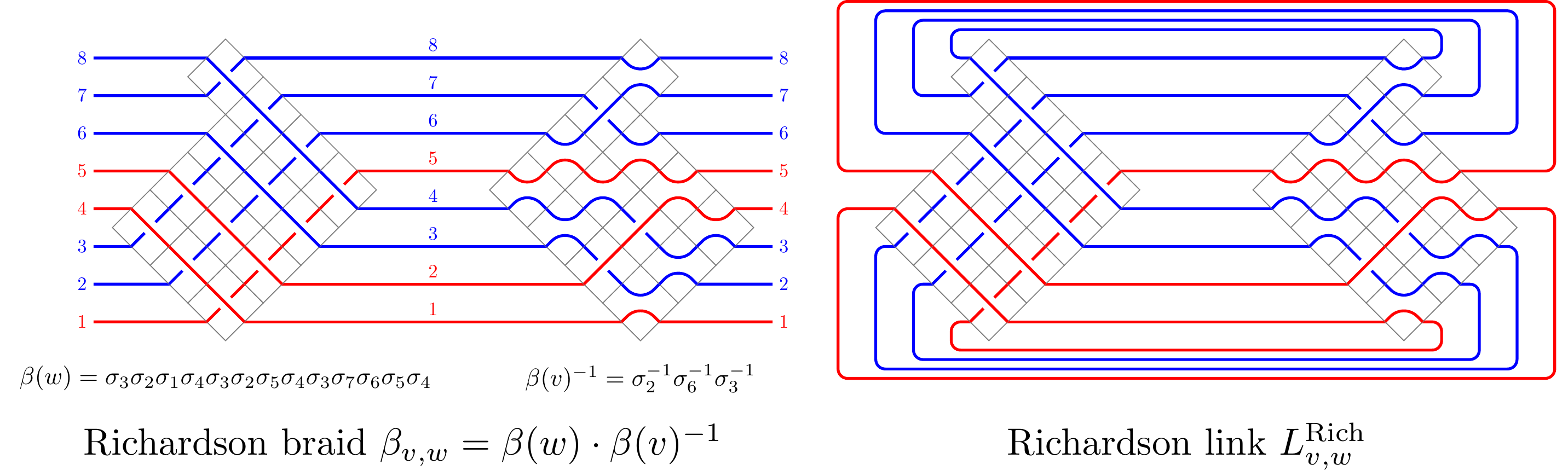}
  \caption{\label{fig:Lrich} A Richardson braid (left) and a Richardson link (right) defined in \cref{sec:intro:Rich_link}.}
\end{figure}

\linksection{Richardson link}\label{sec:intro:Rich_link} A permutation $w\in\Sn$ is called \emph{$k$-Grassmannian} if $w(1)<\cdots<w(\n-k)$ and $w(\n-k+1)<\cdots<w(\n)$. We denote by $\Skn$ the set of $k$-Grassmannian permutations. The permutation $\fp$ can be uniquely decomposed as the product $\fp=wv^{-1}$ for some $1\leq k\leq \n-1$ and some pair $(v,w)\in \Sn\times\Skn$ such that $v\leq w$ in the Bruhat order on~$\Sn$. For the permutation $\fp$ given by~\eqref{eq:intro:fp}, we get that $k=4$ and\footnote{Our convention for multiplying permutations is right-to-left, i.e., $\fp(j)=w(v^{-1}(j))$ for all $j\in[\n]$.}
\begin{equation*}%
v=\left(\!\!\text{
\begin{tabular}{cccccccc}
\textcolor{red}{$1$} & \textcolor{red}{$2$} & \textcolor{blue}{$3$} & \textcolor{blue}{$4$} & \textcolor{red}{$5$} & \textcolor{blue}{$6$} & \textcolor{blue}{$7$} & \textcolor{blue}{$8$}\\
\textcolor{red}{$1$} & \textcolor{red}{$4$} & \textcolor{blue}{$2$} & \textcolor{blue}{$3$} & \textcolor{red}{$5$} & \textcolor{blue}{$7$} & \textcolor{blue}{$6$} & \textcolor{blue}{$8$}
\end{tabular}
}\!\!\right)
,\quad
w=\left(\!\!\text{
\begin{tabular}{cccccccc}
\textcolor{red}{$1$} & \textcolor{red}{$2$} & \textcolor{blue}{$3$} & \textcolor{blue}{$4$} & \textcolor{red}{$5$} & \textcolor{blue}{$6$} & \textcolor{blue}{$7$} & \textcolor{blue}{$8$}\\
\textcolor{red}{$4$} & \textcolor{red}{$5$} & \textcolor{blue}{$6$} & \textcolor{blue}{$8$} & \textcolor{red}{$1$} & \textcolor{blue}{$2$} & \textcolor{blue}{$3$} & \textcolor{blue}{$7$}
\end{tabular}
}\!\!\right)
.
\end{equation*}
Here, the colors represent the cycles of $\fp=wv^{-1}$. 
 Let $\beta(v),\beta(w)$ be the positive braid lifts of $v,w$ to the braid group of $\Sn$. Consider the \emph{Richardson braid} $\beta_{v,w}:=\beta(w)\cdot \beta(v)^{-1}$ shown in \figref{fig:Lrich}(left). Taking the \emph{braid closure} of $\beta_{v,w}$, we obtain the \emph{Richardson link} $\Lrich_{v,w}=\Lrich_\fp$ shown in \figref{fig:Lrich}(right). We orient the strands of $\beta_{v,w}$ from right to left; cf. \figref{fig:Lrich-to-L'}(left).
\begin{remark}\label{rmk:Rich_Young}
It is well known that $k$-Grassmannian permutations are in bijection with Young diagrams that fit inside a $k\times (\n-k)$ rectangle. Thus, in \figref{fig:Lrich}(left), we arrange the crossings in $\beta(w)$ into a (rotated) Young diagram. The condition that $v\leq w$ in the Bruhat order implies that a reduced word for $v$ is a subword of a reduced word for $w$, and thus we represent $\beta(v)$ in \figref{fig:Lrich}(left) by replacing some of the crossings in $\beta(w)$ with ``elbows.''
\end{remark}

 The above recipe was used in~\cite{GL_qtcat} more generally for pairs $v\leq w$ where $w$ is not necessarily $k$-Grassmannian. 

\begin{figure}
\includegraphics[width=1.0\textwidth]{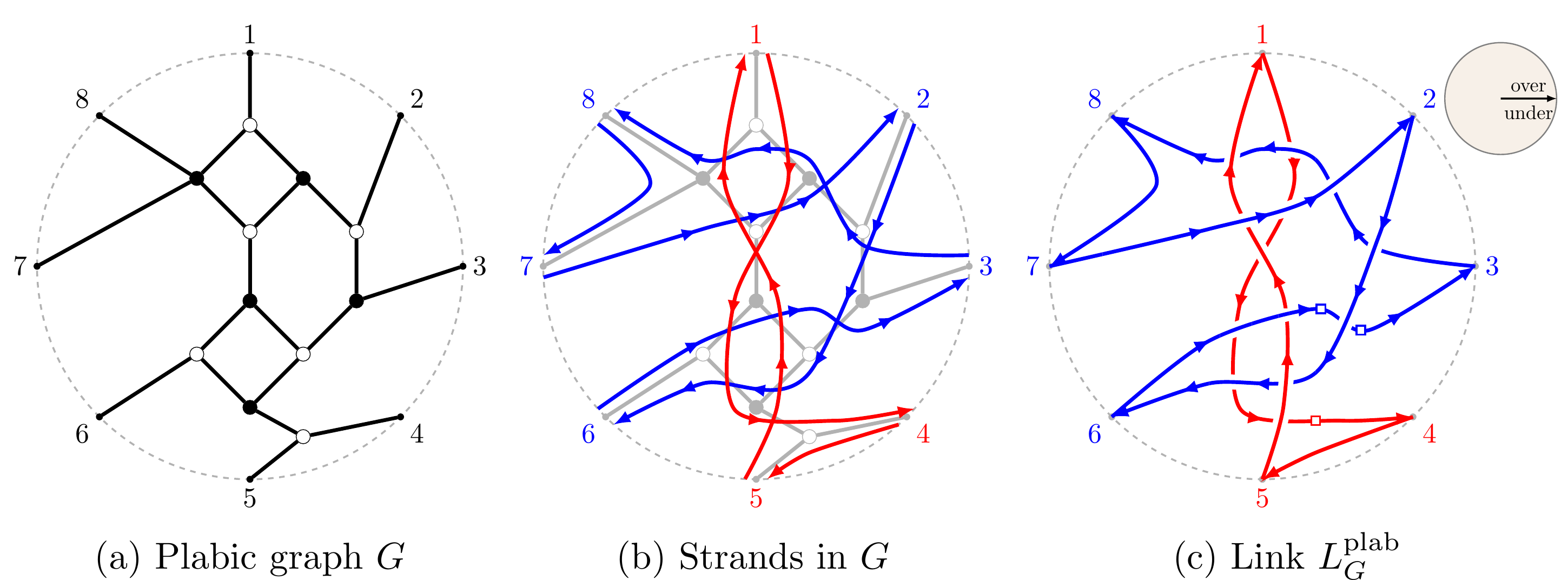}
  \caption{\label{fig:Lplab} Converting a plabic graph $G$ into the plabic graph link $\Lplab_G$.}
\end{figure}

\linksection{Plabic link} \label{sec:intro:plabic_link}
A \emph{plabic graph} is a non-empty planar graph $G$ embedded in a disk with vertices colored black and white; see~\cite{Pos}.  
 A \emph{strand} in $G$ is a path that makes a sharp right turn at each black vertex and a sharp left turn at each white vertex. We assume that the graph $G$ has $\n$ boundary vertices, all of which have degree $1$ and are labeled $1,2,\dots,\n$ in clockwise order. An example of a plabic graph $G$ is shown in \figref{fig:Lplab}(a), and the strands in $G$ are shown in \figref{fig:Lplab}(b).

The following description of the \emph{plabic graph link} $\Lplab_G$, or \emph{plabic link} for short, can be deduced from~\cite{ACampo,STWZ,FPST} by breaking the symmetry following~\cite{Hirasawa}. (See \cref{sec:plabic_links_properties} for a more invariant description.) The planar diagram of $\Lplab_G$ will consist of the union of the strands of $G$. To specify the overcrossings, let $p$ be an intersection point of two strands $S_1,S_2$ in $G$.  The tangent vector to $S_1$ (resp. $S_2$) at $p$ can be considered a vector in the complex plane, and we let $\arg(S_1,p) \in[0,2\pi)$ (resp. $\arg(S_2,p)\in[0,2\pi)$) denote the argument of this vector.
We assume that $0<\arg(S_1,p)\neq\arg(S_2,p)<2\pi$. Then $S_1$ is drawn above $S_2$ if $\arg(S_2,p)>\arg(S_1,p)$, otherwise $S_1$ is drawn below $S_2$. 

\begin{figure}
\includegraphics[width=0.7\textwidth]{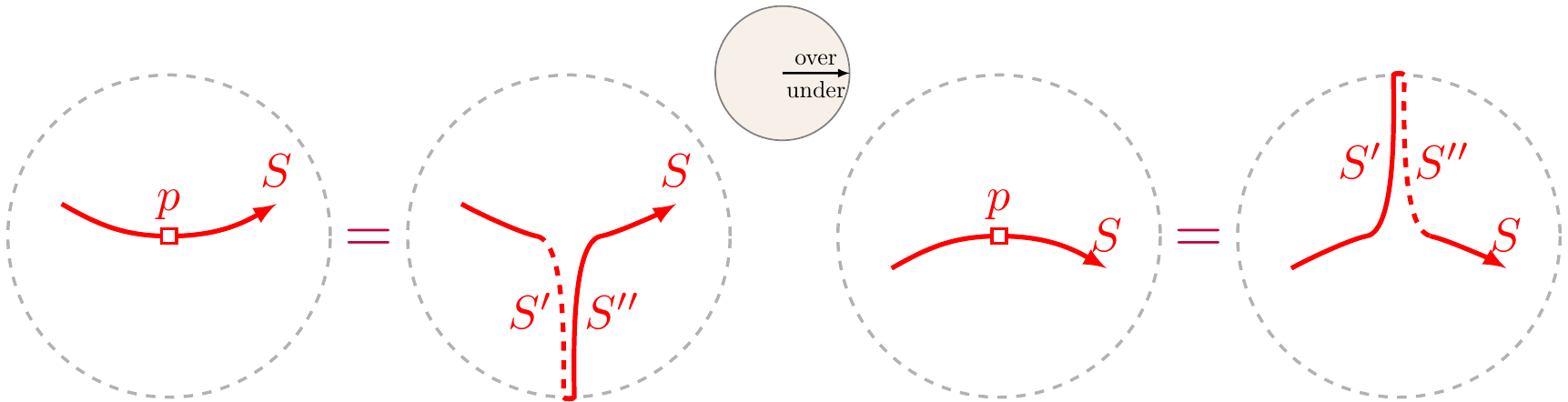}
  \caption{\label{fig:over_under} For every point $p$ on a strand $S$ satisfying $\arg(S,p)=0$, we insert a strand segment going to the boundary and back; see \cref{sec:intro:plabic_link}.}
\end{figure}

A technical extra step is required to finalize the construction of $\Lplab_G$; see \cref{fig:over_under}. Consider a point $p$ on a strand $S$ such that $\arg(S,p)=0$, i.e., such that $S$ is directed to the right at $p$. We add an extra segment $S'$ to $S$ traveling from $p$ to the boundary of the disk followed by a segment $S''$ traveling from the boundary back to $p$. In the neighborhood of $p$, $S$ looks like a plot of a function. If this function is ``concave up'' at $p$ (i.e., has positive second derivative), then $S'$ (resp., $S''$) is drawn below (resp., above) all other strands. If $S$ is ``concave down'' at $p$ then $S'$ (resp., $S''$) is drawn above (resp., below) all other strands. 

Finally, each boundary vertex of $G$ is an endpoint of exactly two strands (one outgoing and one incoming). We join these endpoints together and obtain a planar diagram of a link denoted $\Lplab_G$. See \figref{fig:Lplab}(c), where the points $p$ such that $\arg(S,p)=0$ are marked by the symbol $\mysq$.

 The strands in $G$ that start and end at the boundary vertices give rise to the \emph{strand permutation} $\pig$ of $G$. In this subsection, we will assume that $G$ is \emph{reduced}, i.e., that $G$ has the minimal possible number of faces among all graphs with a given strand permutation. Importantly, in \cref{sec:intro:quivers-point-count}, we will drop this assumption. Postnikov~\cite{Pos} showed that for any $\pi\in \Sn$, there exists a reduced plabic graph $G$ with strand permutation $\pig=\pi$. For example, for $\pi$ given by~\eqref{eq:intro:fp}, one can take $G$ to be the graph in \figref{fig:Lplab}(a). Up to isotopy, the resulting link $\Lplab_G$ depends only on $\pi$ and is denoted $\Lplab_\pi$.

\begin{figure}
\includegraphics[width=1.0\textwidth]{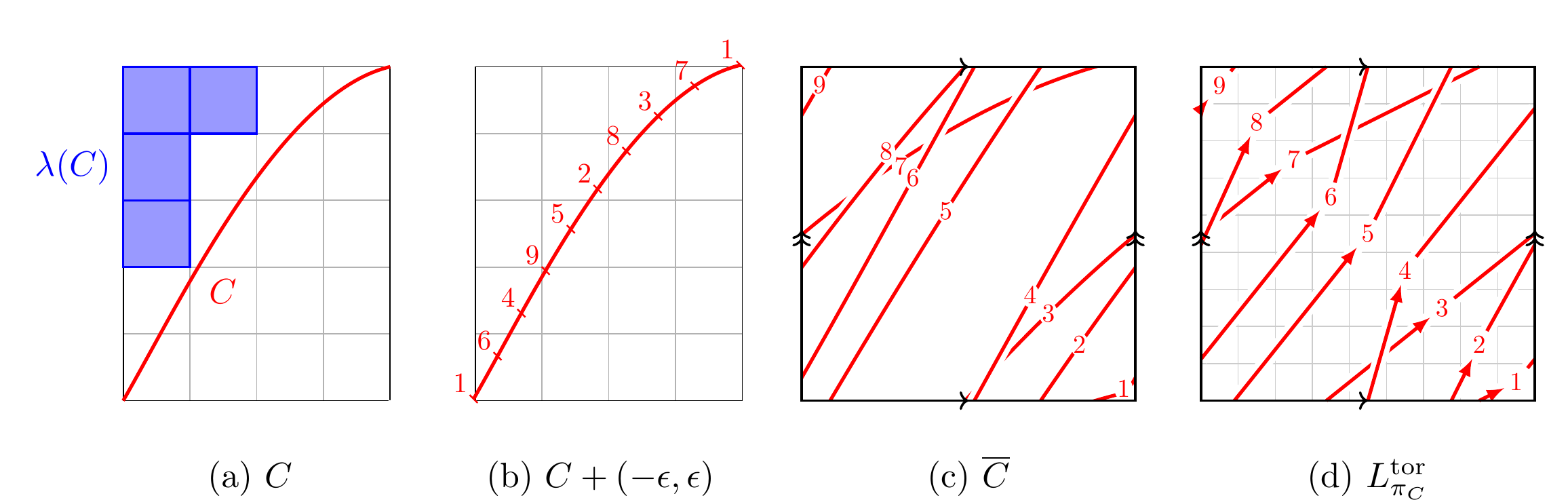}
  \caption{\label{fig:concave} Converting a concave curve $C$ into the link $\Ltor_{\pi_C}$ for some $\pi_C\in\Sn$; see \cref{ex:intro:Cox}.}
\end{figure}

\linksection{Coxeter link} \label{sec:intro:cox_link}
In~\cite[Section~6]{GL_cat_combin}, we studied \emph{concave permutations}, defined as follows. Choose a generic concave curve $C$ inside a $k\times (\n-k)$ rectangle connecting $(0,0)$ to $(k,\n-k)$ as in \figref{fig:concave}(a). We will shift it by the vector $(-\eps,\eps)$ for some small $\eps>0$ as in \figref{fig:concave}(b). Under the natural projection $\R^2\to\T=\R^2/\Z^2$, the curve $C+(-\eps,\eps)$ projects to a curve $\Cproj$ in the unit square $[0,1]^2$. For each self-intersection of $\Cproj$, we draw the segment of the higher slope above the segment of the lower slope; see \figref{fig:concave}(c). The resulting link diagram is isotopic to $\Ltor_\pi$ for a permutation $\pi=\pi_C\in\Sn$ which can be explicitly read off from $\Cproj$ as follows. Label the intersection points of $\Cproj$ with the $y=1-x$ diagonal of $[0,1]^2$ by $d_1,d_2,\dots,d_\n$, proceeding in the northwest direction. Thus, $d_1=(1-\eps,\eps)$ is the projection of the starting and ending points of $C+(-\eps,\eps)$. In \figref{fig:concave}(c), we represent each point $d_i$ by~$i$. Since $C$ was generic, we assume that the points $d_1,d_2,\dots, d_\n$ are pairwise distinct. The permutation $\pi_C$ is defined so that each segment of $C$ connects (in the northeast direction) $d_i$ to $d_{\pi(i)}$ for some $i\in[\n]$. See \figref{fig:concave}(d). We refer to permutations that can be obtained in this way as \emph{concave permutations}. They form a subclass of the class of \emph{repetition-free permutations} introduced in~\cite{GL_cat_combin}.

Let $\la_C=(\la_1,\la_2,\dots,\la_k)$ be the Young diagram inside the $k\times (\n-k)$ rectangle consisting of all unit boxes strictly above $C$, shown in \figref{fig:concave}(a). 
 Consider a sequence $\abf(C)=(a_2,\dots,a_k)$ given by $a_i:=\la_{i-1}-\la_i$ for $2\leq i\leq k$. The \emph{Coxeter link} $\Lcox_C$ is the closure of the braid
\begin{equation}\label{eq:beta_abf_dfn}
  \beta(\abf(C)):=\ell_2^{a_2}\cdots \ell_k^{a_k}\cdot\sigma_1\cdots \sigma_{k-1},
\end{equation}
where $\sigma_1,\dots,\sigma_{k-1}$ are the standard braid group generators of the braid group of $S_k$, and $\ell_i=\sigma_{i-1}\cdots \sigma_1\cdot \sigma_1\cdots \sigma_{i-1}$ are the Jucys--Murphy elements. 

\begin{figure}
\includegraphics[width=0.85\textwidth]{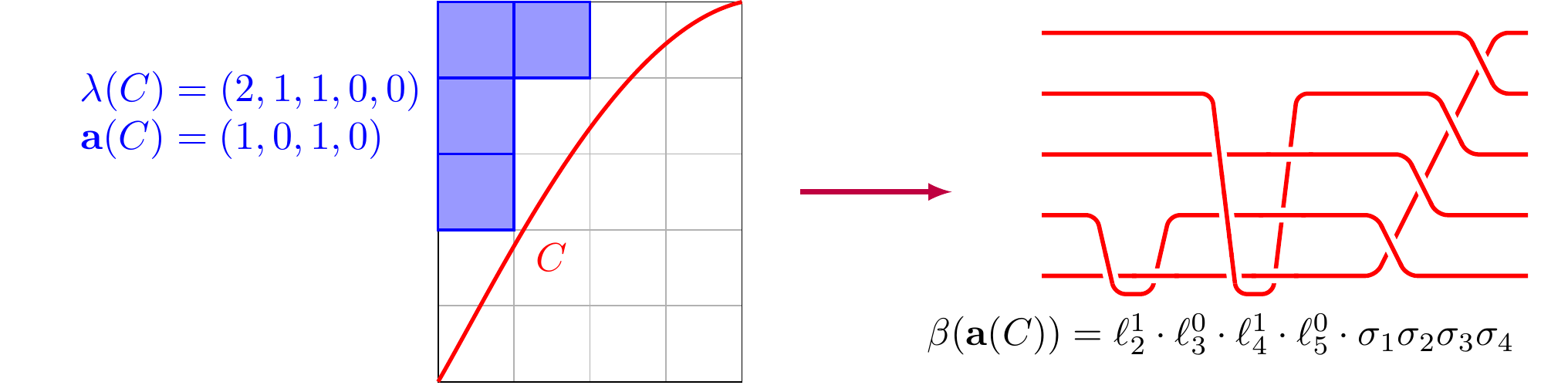}
  \caption{\label{fig:concave_2} Converting a concave curve $C$ into the Coxeter braid $\beta(\abf(C))$.}
\end{figure}
\begin{example}\label{ex:intro:Cox}
Let $C$ be the curve shown in \cref{fig:concave,fig:concave_2}. We have $k=5$ and $\n=9$. The permutation $\pi_C=\left(\!\!\text{
\begin{tabular}{ccccccccc}
$1$ & $2$ & $3$ & $4$ & $5$ & $6$ & $7$ & $8$ & $9$\\
$6$ & $8$ & $7$ & $9$ & $2$ & $4$ & $1$ & $3$ & $5$
\end{tabular}
}\!\!\right)$ can be read off from the labels\footnote{The labels in \figref{fig:concave}(b) are chosen in such a way that when we project $C$ to the unit square, they become ordered along the diagonal $y=1-x$; see \figref{fig:concave}(c).} in \figref{fig:concave}(b): in cycle notation, $\pi_C$ is given by $\pi_C=(1\,6\,4\,9\,5\,2\,8\,3\,7\,1)$. We have $\la_C=(2,1,1,0,0)$, $\abf(C)=(1,0,1,0)$,  and $\beta(\abf(C))=\ell_2\ell_4\sigma_1\sigma_2\sigma_3\sigma_4$; see also~\cite[Figure~13]{GL_cat_combin}.
\end{example}
 Coxeter links and their Khovanov--Rozansky homology have previously been studied in relation to flag Hilbert schemes and generalized shuffle conjectures~\cite{ObRo,GN,GNR,BHMPS}; see also~\cite[Section~7.2]{GL_cat_combin}.

The following is our first main result.
\begin{theorem}\label{thm:isotopy}
  Let $\fp\in\Sn$ be a permutation without fixed points. Then the links $\Lperm_\fp$, $\Ltor_\fp$, $\Lrich_\fp$, and $\Lplab_\fp$ are all isotopic. If $\fp=\fp_C$ for some concave curve $C$ then each of these links is also isotopic to $\Lcox_C$.
\end{theorem}
\noindent We refer to any one of the above links as the \emph{positroid link} of $\fp$.

\begin{remark}
Recently, isotopies between the Richardson link $\Lrich_\pi$ and some other closely related links (namely, closures of juggling braids, cyclic rank matrix braids, and Le-diagram braids)  were independently constructed in~\cite{CGGS}. Isotopies between cyclic rank matrix links and plabic graph links have also been observed in~\cite{STWZ}.
\end{remark}

\subsection{Quiver point count and the HOMFLY polynomial}\label{sec:intro:quivers-point-count}
In this subsection, we consider not necessarily reduced plabic graphs $G$. Throughout the paper, we assume that the interior faces of $G$ are simply connected, i.e., that no interior face of $G$ contains another connected component of $G$ inside of it. For simplicity, we also assume that each plabic graph $G$ is \emph{trivalent}, i.e., that each interior vertex of $G$ has degree $3$. (Recall that the boundary vertices of $G$ are always required to have degree $1$.)

Let $G$ be a (trivalent) plabic graph. The planar dual of $G$ is naturally a directed graph denoted $\QG$. Explicitly, place a vertex of $\QG$ inside each \emph{interior face} of $G$ (i.e., a face not adjacent to the boundary of the disk). For every edge $e$ of $G$ whose endpoints are of different color and such that the two faces $F,F'$ adjacent to $e$ are both interior, $\QG$ contains an arrow between $F$ and $F'$. (These two faces $F,F'$ may or may not be equal.) The direction of the arrow is chosen so that the white endpoint of $e$ is on the left as one moves along this arrow. See \cref{fig:simple}.

\begin{figure}
  \includegraphics[width=1.0\textwidth]{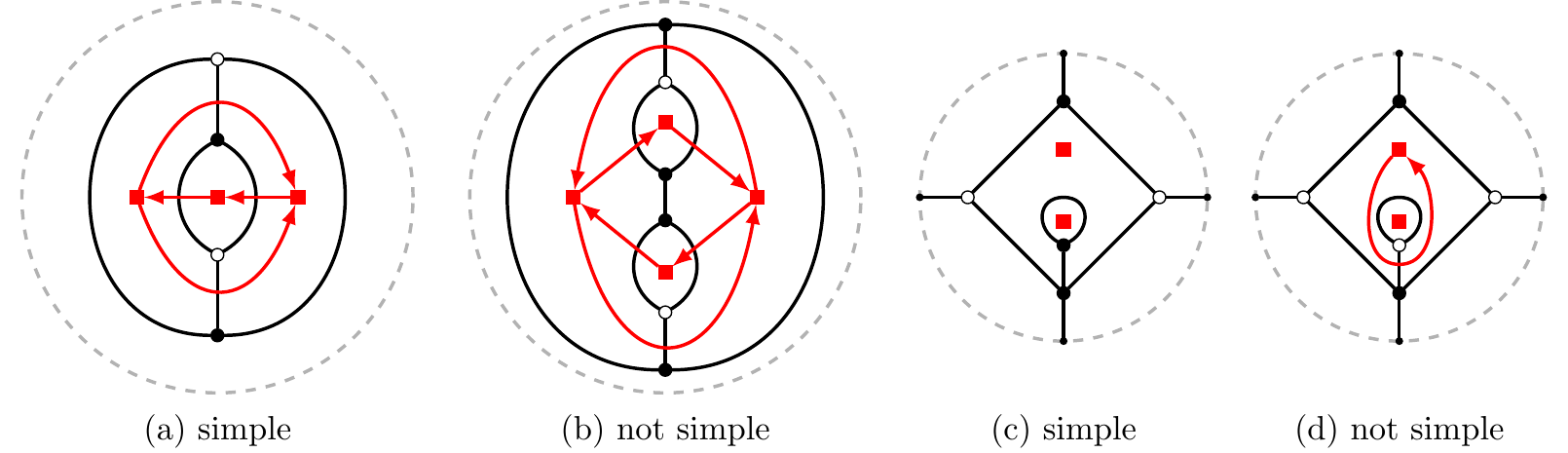}
  \caption{\label{fig:simple} Examples of simple and non-simple plabic graphs $G$ (black) and their planar duals $\QG$ (red).}
\end{figure}

 The following definition is crucial for our analysis.
\begin{definition}
A plabic graph $G$ is called \emph{simple} if the directed graph $\QG$ is a \emph{quiver}, i.e., contains no directed cycles of length $1$ and $2$.
\end{definition}
\noindent For example, the plabic graph $G$ in \figref{fig:simple}(b) is not simple since $\QG$ contains a pair of opposite arrows. The graph $G$ in \figref{fig:simple}(d) is not simple since $\QG$ contains a loop arrow. The graphs in \figref{fig:simple}(a,c) are simple.

To every \emph{\Louise}\footnote{We call a quiver \Louise if it satisfies the ``Louise condition" of~\cite{MSLA,LS}.  This is stronger than the notion of local acyclicity used in \cite{Muller}.} quiver $Q$ (see~\cref{sec:quiver}), we associate a rational function $\RQ(q)$ called the \emph{point count rational function}. It can be computed via an explicit recurrence relation: see~\eqref{eq:RQ_recurrence} and~\eqref{eq:RQ_skein}. The function $\RQ(q)$ has the following interpretation in terms of \emph{cluster algebras}~\cite{FZ}; see~\cref{sec:quiver} for the relevant definitions. 
Let $Q$ be a quiver with $n$ vertices. Consider an ice quiver $\Qice$ with $m$ frozen vertices and mutable part $Q$, and assume that the rows of the $(n+m)\times n$ exchange matrix $\BQice$ of $\Qice$ span $\Z^n$ over $\Z$. Then the cluster algebra $\AQice$ gives rise to a \emph{cluster variety} $\XQice$, and its point count over a finite field $\F_q$ with $q$ elements (for $q$ a prime power) is given by
\begin{equation*}%
  \#\XQice(\F_q)=(q-1)^m\cdot \RQ(q);
\end{equation*}
see \cref{prop:Q_pcnt}.

 A typical example of this phenomenon occurs when $G$ is a reduced plabic graph: then, by~\cite{GL_cluster}, 
  $\AQGice$ is the coordinate ring of the associated \emph{open positroid variety} $\Pio_G$ inside the Grassmannian~\cite{Pos,KLS}. In particular, $(q-1)^{\n-\conn(G)}\RQG(q)$ counts the number of points in $\Pio_G(\F_q)$, where $\n$ is the number of boundary vertices of $G$ and $\conn(G)$ is the number of connected components of $G$.%

Given an (oriented) link $L$, one can define a Laurent polynomial $\HOMP(L)=\HOMP(L;a,z)$ called the \emph{\FLY polynomial}~\cite{HOMFLY,PT} of $L$. It is defined by the skein relation
\begin{equation}\label{eq:HOMFLY_dfn}
  a\HOMP(L_+) - a^{-1} \HOMP(L_-)=z\HOMP(L_0)\quad \text{and}\quad \HOMP(\unkn)=1.
\end{equation}
Here, $\unkn$ denotes the unknot and  $L_+$, $L_-$, $L_0$ are any three links whose planar diagrams locally differ as  follows.
\def\lw{3pt}
\def\lww{10pt}
\def\drcirc{
\draw[line width=0.3pt, dashed] (0,0) circle (1cm);
}
\def\scl{0.4}

\def\DA{30}
\begin{center}
\begin{tabular}{ccc}
\scalebox{\scl}{
\begin{tikzpicture}
\drcirc
\draw[line width=\lw,->,>=latex] (-90+\DA:1)--(90+\DA:1);
\draw[line width=\lww,white] (-90-\DA:1)--(90-\DA:1);
\draw[line width=\lw,->,>=latex] (-90-\DA:1)--(90-\DA:1);
\end{tikzpicture}
} 

&

\scalebox{\scl}{ 
\begin{tikzpicture}
\drcirc
\draw[line width=\lw,->,>=latex] (-90-\DA:1)--(90-\DA:1);
\draw[line width=\lww,white] (-90+\DA:1)--(90+\DA:1);
\draw[line width=\lw,->,>=latex] (-90+\DA:1)--(90+\DA:1);
\end{tikzpicture} 
} 
& 
 
\def\rc{10} 
\scalebox{\scl}{ 
\begin{tikzpicture}
\drcirc
\draw[line width=\lw,->,>=latex,rounded corners=\rc] (-90-\DA:1)--(0,0)--(90+\DA:1);
\draw[line width=\lw,->,>=latex,rounded corners=\rc] (-90+\DA:1)--(0,0)--(90-\DA:1);
\end{tikzpicture} 
} 
\\ 
$L_+$  &  $L_-$  &  $L_0$

\end{tabular}
\end{center}

\begin{figure}
  \includegraphics[width=1.0\textwidth]{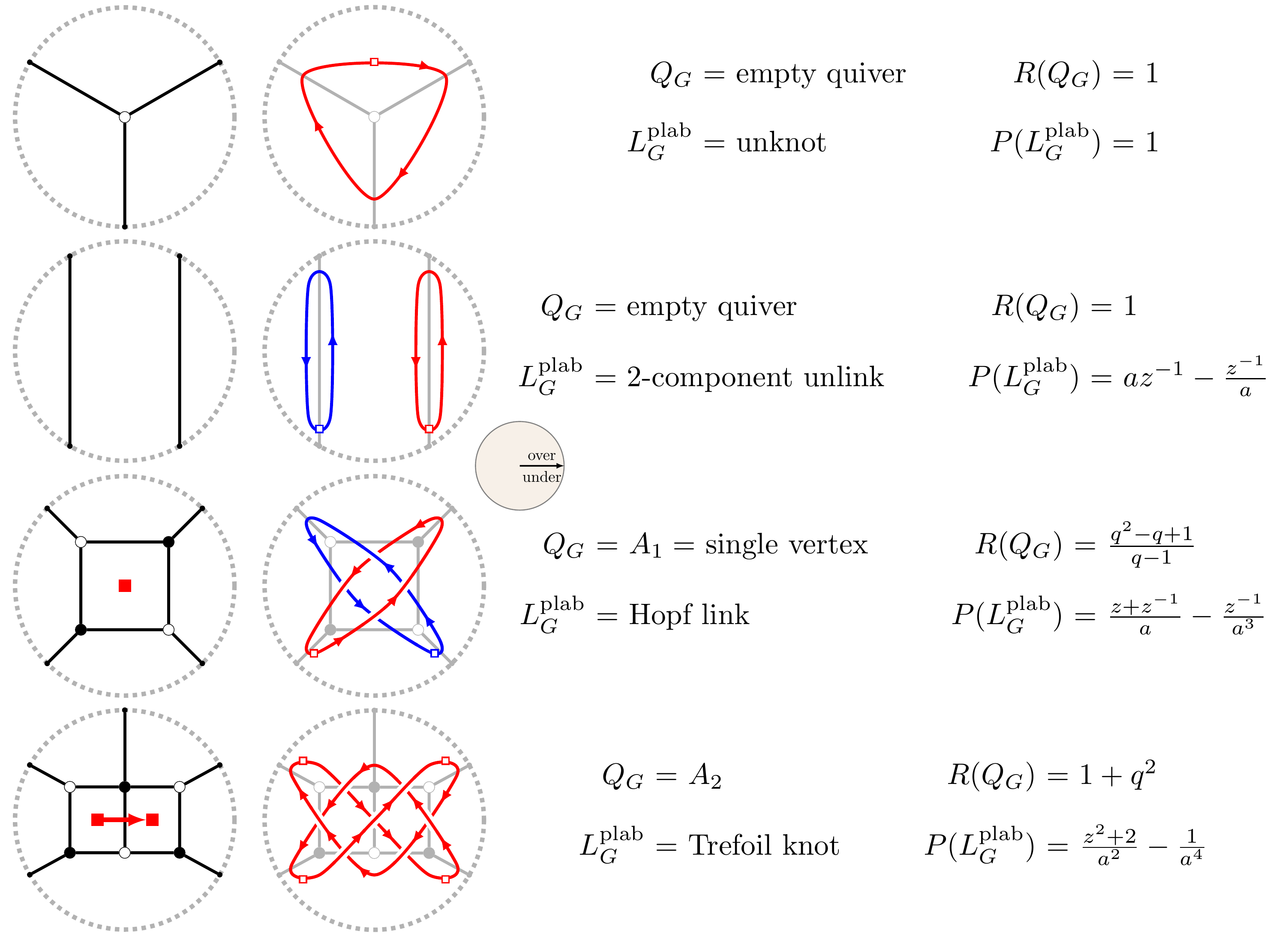}
  \caption{\label{fig:intro_ex} Small plabic graph quivers $\QG$, links $\Lplab_G$, point counts $\RQG(q)$, and \FLY polynomials $\HOMP(\Lplab;a,z)$ illustrating \cref{conj:main}.}
\end{figure}

For a (not necessarily reduced) plabic graph $G$ we let $\Lplab_G$ be the corresponding plabic graph link, defined as in \cref{sec:intro:positroid_links}. Following~\cite{GL_qtcat}, we let  $\Ptop_L(q)$ be obtained from the top $a$-degree term of $\HOMP(L;a,z)$ by substituting $a:=q^{-\frac12}$ and $z:=q^{\frac12}-q^{-\frac12}$. The following conjecture generalizes~\cite[Theorem~1.11]{GL_qtcat} (see also~\cite{STZ,STWZ}).
\begin{conjecture} \label{conj:main}
Let $G$ be a simple plabic graph with $\conn(G)$ connected components.  Then we have
\begin{equation}\label{eq:FLY=pcnt}
  \RQG(q)= (q-1)^{\conn(G)-1}\Ptop_{\Lplab_G}(q).
\end{equation}
\end{conjecture}
\noindent For examples, see \cref{fig:intro_ex}, \cref{ex:Dynkin}, and \cref{sec:conjectures-examples}.

When the plabic graph $G$ is reduced, \cref{conj:main} becomes~\cite[Theorem~1.11]{GL_qtcat}. We present a different proof in \cref{sec:sub_skein-relation} by giving a plabic graph interpretation of the \FLY skein relation~\eqref{eq:HOMFLY_dfn}. More generally, in \cref{sec:leaf-recurr-plab}, we introduce a class of \emph{leaf recurrent plabic graphs} and prove \cref{conj:main} for them. 

\begin{theorem}\label{thm:FLY=pcnt_leaf_rec}
  \Cref{conj:main} holds for leaf recurrent plabic graphs.
\end{theorem}
 The fact that reduced plabic graphs are leaf recurrent was shown in~\cite[Remark~4.7]{MSLA}. Another natural subclass of leaf recurrent plabic graphs are the \emph{plabic fences} of~\cite[Section~12]{FPST}.
 A \emph{plabic fence} is a plabic graph obtained by drawing $k$ horizontal strands and inserting an arbitrary number of black-white and white-black bridges between them; see \cref{fig:plabic_fence} for an example and \cref{sec:plabic-fences} for a precise definition.
\begin{figure}
  \includegraphics[width=0.5\textwidth]{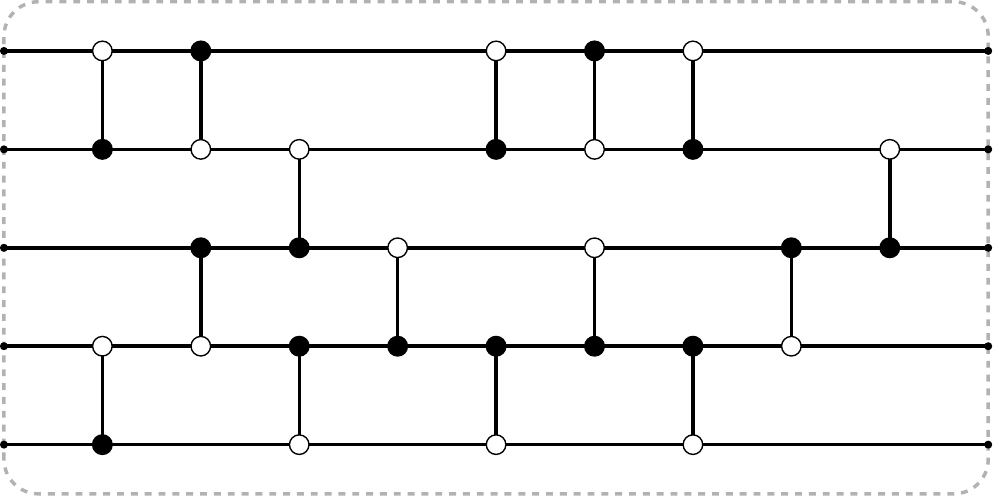}
  \caption{\label{fig:plabic_fence} A plabic fence.}
\end{figure}

\begin{proposition}\label{prop:reduced_and_fence_are_leaf_rec}
Reduced plabic graphs and plabic fences are leaf recurrent. In particular, \cref{conj:main} holds for these classes of plabic graphs.
\end{proposition}
\noindent This result gives a new proof of~\cite[Theorem~1.11]{GL_qtcat}.

A primary difficulty in showing \cref{conj:main} in full generality comes from the examples described in \cref{sec:notleaf,sec:notLA}. \Cref{sec:notleaf} explains that the simple plabic graph $G$ in \figref{fig:simple}(a) is not leaf recurrent. The quiver $\QG$ is \Louise (in fact, it is mutation equivalent to an acyclic quiver), and thus the point count polynomial $\RQG(q)$ can be easily checked to coincide with $\Ptop_{\Lplab_G}(q)$. \cref{sec:notLA} describes a simple plabic graph $G$ such that $\QG$ is not locally acyclic, in which case the task of computing $\RQG(q)$ becomes much harder.

Both sides of~\eqref{eq:FLY=pcnt} are rational functions in $q$ which are specializations of rational functions in $(q,t)$ arising from the cohomology of cluster varieties on the one hand and Khovanov--Rozansky link homology on the other hand. We discuss these more general conjectures in \cref{sec:clust-cohom-vs-link-hom}.

\section{Positroid combinatorics}\label{sec:positr-comb}
In this section, we review some background on the combinatorics of plabic graphs and the associated objects.

\subsection{Bounded affine permutations}\label{sec:bound-affine-perm}
A \emph{$(k,\n)$-bounded affine permutation}~\cite{KLS} is a bijection $f:\Z\to\Z$ such that 
\begin{itemize}
\item $f$ is \emph{$\n$-periodic}: $f(i+\n)=f(i)+\n$ for all $i\in \Z$,
\item  $i\leq f(i)\leq i+\n$ for all $i\in\Z$, and %
\item $\sum_{i=1}^\n(f(i)-i)=k\n$.
\end{itemize}
The set of $(k,\n)$-bounded affine permutations is denoted $\Boundkn$. Each $f\in\Boundkn$ gives rise to a unique permutation $\fbar\in\Sn$ defined by $\fbar(i)\equiv f(i)\pmod \n$ for each $i\in[\n]$. Given $f\in\Boundkn$, an integer $i\in\Z$ is called a \emph{loop} (resp., a \emph{coloop}) of $f$ if $f(i)=i$ (resp., $f(i)=i+\n$). Thus, $\fbar$ has no fixed points if and only if $f$ has no loops and no coloops. In this case, $f$ is uniquely determined by $\fbar$, and the integer $k$ is recovered from $\fbar$ as follows:
\begin{equation*}%
  k=\k(\fbar):=\#\{i\in[\n]\mid \fbar(i)<i\}.
\end{equation*}
We define $\taukn\in\Boundkn$  by
\begin{equation*}%
  \taukn(i)=
  \begin{cases}
    i, &\text{if $1\leq i\leq \n-k$,}\\
    i+\n, &\text{if $\n-k+1\leq i\leq \n$,}
  \end{cases}
\end{equation*}
and extend it to a function $\taukn:\Z\to\Z$ by $\n$-periodicity.

Let $\Qkn:=\{(v,w)\in \Sn\times \Skn\mid v\leq w\}$, where $\leq$ denotes the Bruhat order on $\Sn$. Given a permutation $u\in\Sn$, we can extend it to an $\n$-periodic map $\tilde u:\Z\to\Z$ defined by $\tilde u(i)=u(i)$ for $i\in[\n]$. For $(v,w)\in\Qkn$, let $f_{v,w}:=\tilde w\taukn\tilde v^{-1}$. By~\cite[Proposition~3.15]{KLS}, the map $(v,w)\mapsto f_{v,w}$ gives a bijection $\Qkn\xrasim \Boundkn$. Taking the inverse of this map, we obtain the following result, justifying the definition of a Richardson link in \cref{sec:intro:Rich_link}.
\begin{corollary}
If $\fp\in\Sn$ has no fixed points then it can be uniquely factored as a product $\fp=wv^{-1}$, where $(v,w)\in\Qkn$ and $k=\k(\fp)$.
\end{corollary}

The top-dimensional open positroid variety corresponds to the bounded affine permutation $\fkn\in\Boundkn$ sending $i\mapsto i+k$ for all $i\in\Z$. We let $\pikn\in\Sn$ be the permutation sending $i\mapsto i+k$ modulo $\n$ for all $i\in[\n]$.

\subsection{Plabic graphs}\label{sec:plabic-graphs}

\begin{figure}
  \includegraphics[width=0.7\textwidth]{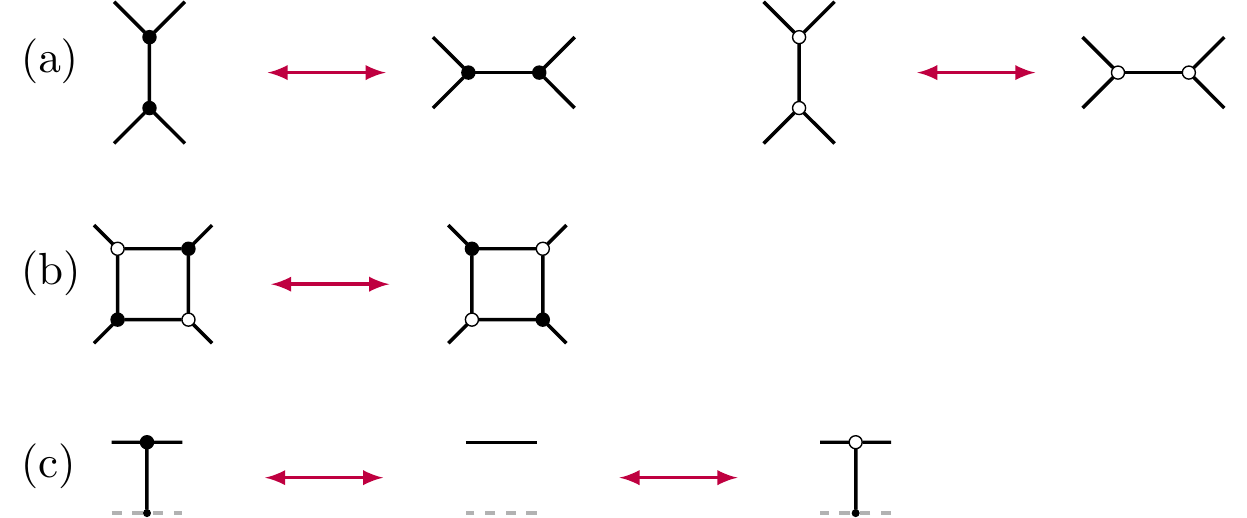}
  \caption{\label{fig:moves} Local moves for plabic graphs: (a) contraction-uncontraction move; (b) square move; (c) tail addition/removal.}
\end{figure}

Let $G$ be a plabic graph. Recall from \cref{sec:intro:plabic_link} that $G$ must be nonempty, that the boundary vertices of $G$ are labeled by $1,2,\dots,\n$ in clockwise order, and that the strand permutation $\pig:[\n]\to[\n]$ sends $i\mapsto j$ whenever the strand starting at $i$ terminates at $j$. We always assume that our plabic graphs $G$ are such that $\pig$ has no fixed points, although most of our constructions can be easily extended to this case.

\begin{definition}[\cite{Pos}]
We say that $G$ is \emph{reduced} if it has the minimal number of faces among all plabic graphs with strand permutation $\pig$.
\end{definition}
Alternatively~\cite[Theorem~13.2]{Pos}, a plabic graph $G$ is reduced if and only if $G$ has no closed strands, no self-intersecting strands, and no pairs of strands forming a \emph{bad double crossing} shown in \figref{fig:double}(right). \emph{Good double crossings} shown in \figref{fig:double}(left) are allowed.

We consider several types of \emph{local moves} on plabic graphs, shown in \cref{fig:moves}. 
 It was shown in~\cite{Pos} that every permutation $\pi\in\Sn$ is the strand permutation of some reduced plabic graph $G$, and that moreover, any two reduced plabic graphs with the same strand permutation are related by a sequence of the moves (a) and (b) in \cref{fig:moves}. Below we review one construction of plabic graphs which will be particularly useful to us.

\subsection{Le-diagrams}
Let $\la=(\la_1\geq\la_2\geq\cdots\geq\la_k\geq0)$ be a partition, which we identify with its Young diagram. We assume that this Young diagram fits inside a $k\times (\n-k)$ rectangle, i.e., satisfies $\la_1\leq \n-k$. We draw Young diagrams using \emph{English} (or \emph{matrix}) \emph{notation}. We use matrix coordinates for the boxes of $\la$, thus, row $i$ contains boxes with coordinates $(i,j)$ for $1\leq j\leq \la_i$, and box $(1,1)$ is the top left box.

A \emph{Le-diagram} $\Gamma$ \emph{of shape $\la$} is a way of placing a dot in some of the boxes of $\la$ so that every box that is below a dot in the same column and to the right of a dot in the same row must also contain a dot.  
 An example of a Le-diagram of shape $\la=(3,3,3,1)$ is shown in \figref{fig:Gamma-to-L'}(left). Le-diagrams whose shape fits inside a $k\times(\n-k)$ rectangle are in bijection with the elements of $\Qkn$ and $\Boundkn$. Given a Le-diagram $\Gamma$ of shape $\la$, one can read off the corresponding pair $(v,w)\in\Qkn$ as follows. First, $w\in\Skn$ is the $k$-Grassmannian permutation corresponding to $\la$. Specifically, let us label the unit steps in the southeast boundary of $\la$ by $1,2,\dots,\n$ in increasing order from the northeast to the southwest corner. Then the labels of the vertical steps form a $k$-element subset of $[\n]$, and this set is precisely $\{w(\n-k+1),\dots,w(\n)\}$. This determines $w$ uniquely. Alternatively, we can label each box $(i,j)$ of $\la$ with the simple transposition $s_{k+j-i}$, and then a reduced word for $w$ is obtained by reading these simple transpositions in the northwest direction. Thus, any reduced word for $w$ ends with $s_k$ which labels the box $(1,1)$. To obtain a reduced word for $v$, we read these simple transpositions but ignore the ones labeled by dots in $\Gamma$. Conversely, given $(v,w)\in\Qkn$, the shape $\la$ of $\Gamma$ is reconstructed from $w$, and the empty boxes of $\Gamma$ correspond to the rightmost subexpression for $v$ inside the (unique up to commutation) reduced word for $w$.
\begin{example}
For the Le-diagram in \figref{fig:Gamma-to-L'}(left), we have 
\begin{equation*}%
w= s_7 s_3 s_4 s_5 s_6 s_2 s_3 s_4 s_5 s_1 s_2 s_3 s_4
\quad\text{and}\quad
v= s_6 s_3 s_2.
\end{equation*}
\end{example}

\begin{figure}
  \includegraphics[width=0.7\textwidth]{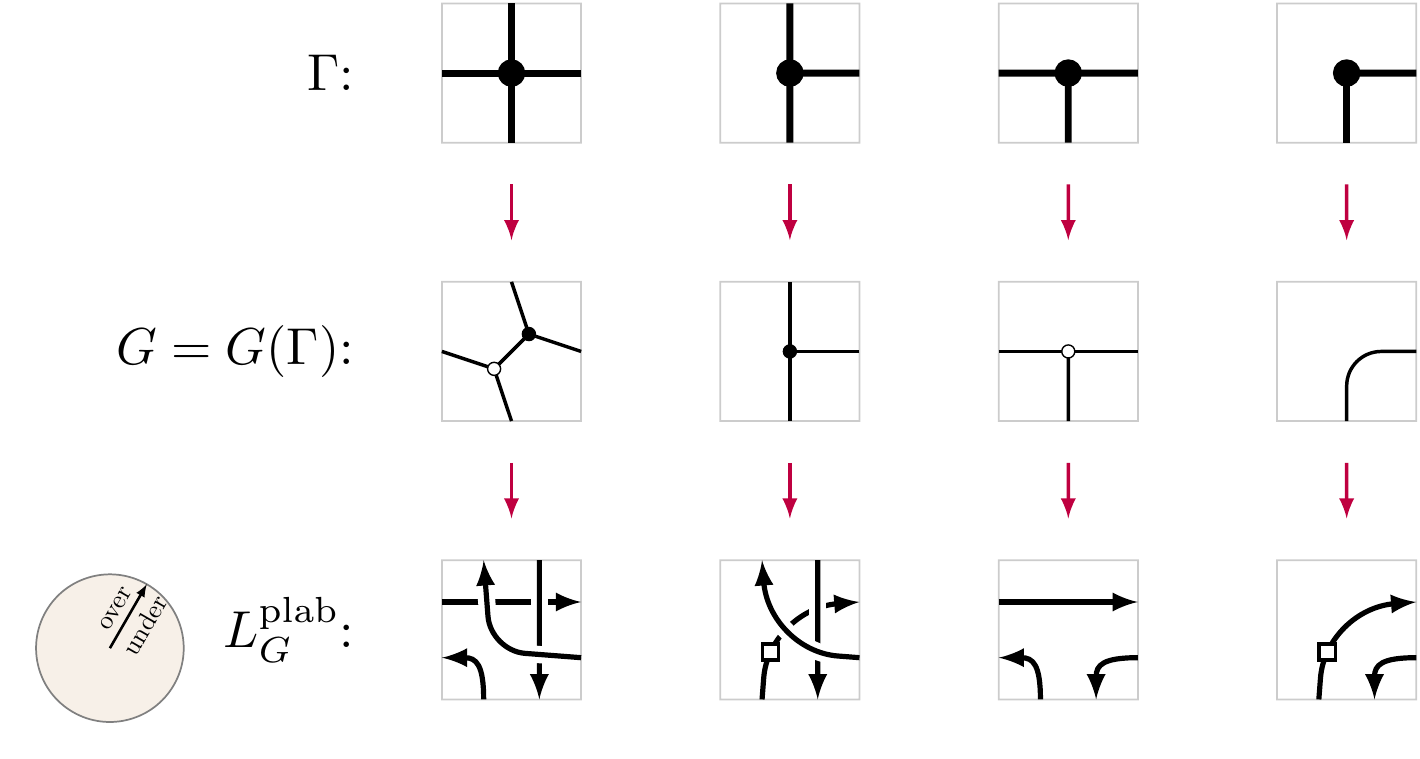}
  \caption{\label{fig:Le_to_G} Converting a Le-diagram (top row) into a plabic graph (middle row) and into a link (bottom row).}
\end{figure}

A Le-diagram $\Gamma$ can be converted into a plabic graph $G(\Gamma)$ using the local rules shown in the first two rows of \cref{fig:Le_to_G}. This plabic graph is always reduced and has strand permutation $\fp=wv^{-1}$, where $(v,w)\in\Qkn$ are recovered from $\Gamma$ using the above procedure.

\subsection{Properties of plabic graph links}\label{sec:plabic_links_properties}
Let $G$ be a reduced plabic graph with strand permutation $\pi=\pig$. Recall the description of the plabic graph link $\Lplab_\pi=\Lplab_G$ from \cref{sec:intro:plabic_link}. It may appear from this description that $\Lplab_G$ depends on the precise way of drawing the strands in $G$ as smooth curves, or that $\Lplab_G$ changes if one rotates the graph $G$ (since the over/under-crossings information depends on the complex arguments of the strands). However, it turns out that the link $\Lplab_G$ is in fact independent of these choices, as explained in~\cite{ACampo,STWZ,FPST}. Let us give a more invariant description of $\Lplab_G$.

Recall that $G$ is embedded in a disk $\disk:=\{x\in\C: |x|\leq 1\}$. Consider the solid torus $\disk\times S^1$, and consider an equivalence relation $\sim$ on it defined as follows: for $(x,y),(x',y')\in \disk\times S^1$, write $(x,y)\sim(x',y')$ if $x,x'\in\partial \disk$ belong to the boundary of the disk and $x=x'$. In other words, we collapse each fiber $x\times S^1$, where $x\in\partial \disk$, to a point. The resulting quotient space $\disk\times S^1/\sim$ is homeomorphic to a $3$-dimensional sphere $S^3$.

An \emph{oriented divide} is a smooth immersion of a union of oriented circles (called \emph{branches}) into $\disk$. 
 Divides are required to satisfy certain genericity conditions, e.g., that any intersections between the branches must be transversal and belong to the interior of $\disk$, and that there have to be no triple intersections; see~\cite[Definitions~2.1 and~8.1]{FPST} for a complete list. 

Given a branch $\gamma$ of an oriented divide $D$, we lift each point of $\gamma$ to $\disk\times S^1$ by letting the $S^1$ coordinate be the complex argument of the tangent vector of $\gamma$ at this point. By the genericity assumptions, we obtain an embedding of a union of oriented circles into $S^3$, i.e., a link, denoted $L_D$.

\begin{figure}
  \includegraphics[width=0.7\textwidth]{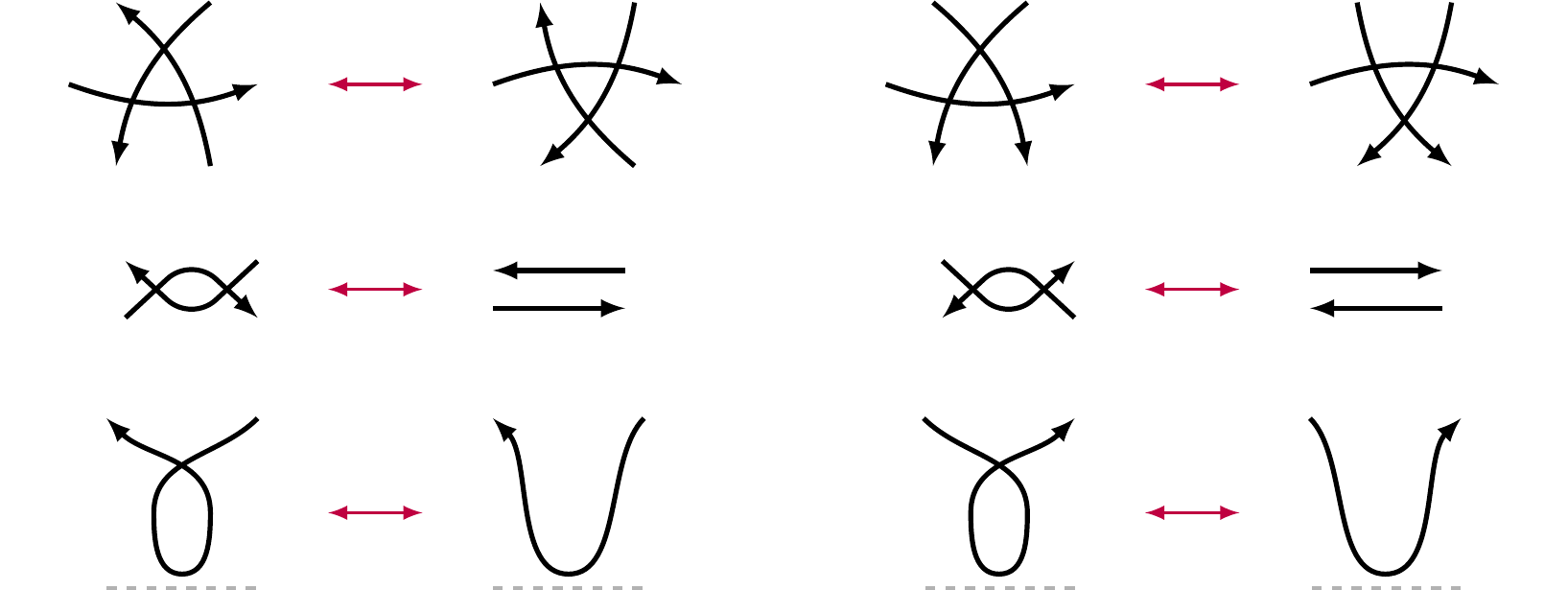}
  \caption{\label{fig:divide_moves} Local moves for divides~\cite[Figures~28--30]{FPST}.}
\end{figure}

The link $L_D$ is invariant under applying the local moves shown in \cref{fig:divide_moves} to $D$; see~\cite[Proposition~8.4]{FPST}.

Take the strands in our reduced plabic graph $G$. Each of the $\n$ boundary vertices of $G$ has one incoming strand and one outgoing strand. Identifying their endpoints, we obtain an oriented divide $D(G)$. It was shown in~\cite{Hirasawa} that the associated link $L_{D(G)}$ is isotopic to the link $\Lplab_G$ described in \cref{sec:intro:plabic_link}. This shows that the link $\Lplab_G$ is indeed invariant under rotation of $G$ and under small perturbations of the branches of $D(G)$, since this is clearly the case for $L_{D(G)}$. 

\begin{remark}\label{rmk:rotation}
In what follows, it will be convenient to us to use rotational invariance of $\Lplab_G$. In our figures, we will fix some angle $\alpha\in[0,2\pi)$, and assume that the link $\Lplab_G$ is obtained by first rotating $G$ by the angle $-\alpha$, then applying the description from \cref{sec:intro:plabic_link}, and then rotating the picture back by $\alpha$. In other words, if two strands $S_1,S_2$ in $G$ intersect at some point $p$, we consider their complex arguments of tangent vectors in a different interval: $\arg(S_1,p),\arg(S_2,p)\in[\alpha,\alpha+2\pi)$. Then we draw $S_1$ above $S_2$ if and only if $\arg(S_2,p)>\arg(S_1,p)$. And then for all points $p$ where some strand $S$ satisfies $\arg(S,p)=\alpha$, we insert line segments going to the boundary and back. In the figures, the direction $\alpha$ is indicated by the ``over/under circle'' \\
\noindent \centerline{\includegraphics[width=0.17\textwidth]{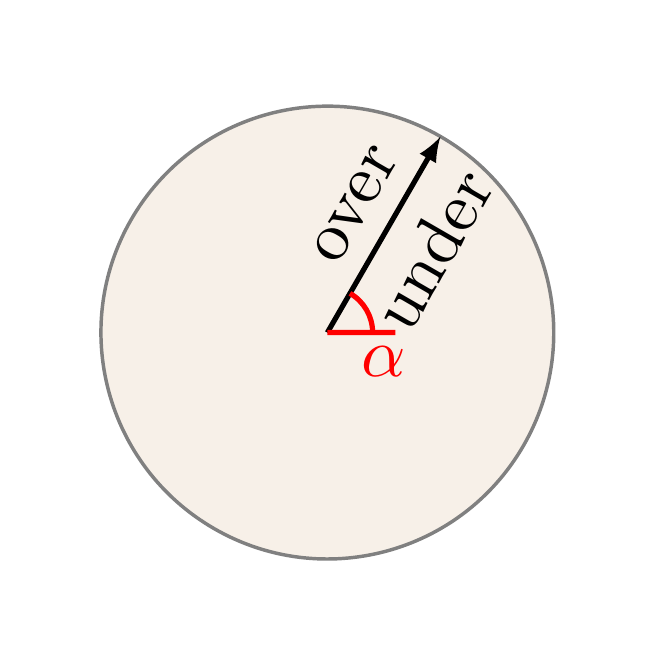}}
\noindent and each point $p$ on a strand $S$ satisfying $\arg(S,p)=\alpha$ is marked by $\mysq$ as in \cref{fig:over_under}.
\end{remark}

\section{Positroid link isotopies}\label{sec:positr-link-isot}
The goal of this section is to prove \cref{thm:isotopy}. We first discuss some properties of plabic graph links, and then we construct a sequence of isotopies, for any permutation $\pi\in\Sn$ without fixed points, of the links
\begin{equation*}%
\Lrich_\pi \cong  \Lplab_\pi \cong \Lperm_\pi \cong \Ltor_\pi,
\end{equation*}
in that order. When $\pi=\pi_C$ for some concave curve $C$, we will also show that
\begin{equation*}%
  \Ltor_\pi\cong\Lcox_\pi,
\end{equation*}
completing the proof.

Throughout this section, we fix a permutation $\pi\in\Sn$ without fixed points, and let $(v,w)\in\Qkn$ be the corresponding pair satisfying $\pi=wv^{-1}$,  $\Gamma$ the corresponding Le-diagram of shape $\la$, and $G$ the corresponding plabic graph, obtained from $\Gamma$ using the recipe in the first two rows of \cref{fig:Le_to_G}.

\subsection{From Richardson links to plabic graph links}\label{sec:from-rich-links}
 Our goal is to construct an isotopy between the Richardson link $\Lrich_{v,w}$ (\figref{fig:Lrich}(right)) and the plabic graph link $\Lplab_G$ (\figref{fig:Lplab}(c)).

Recall from \cref{rmk:Rich_Young} that we are drawing the Richardson braid $\beta_{v,w}=\beta(w)\cdot \beta(v)^{-1}$ in a particular way, so that the crossings of $\beta(w)$ form a shape obtained by rotating $\la$ clockwise by $135^\circ$, and the crossings of $\beta(v)^{-1}$ are drawn in some boxes of a shape obtained from the $135^\circ$-rotation of $\la$ by reflecting it along the vertical axis; see \figref{fig:Lrich}(left). Moreover, the crossings of $\beta(v)^{-1}$ are located precisely in the boxes of $\la$ which do not contain a dot in $\Gamma$. 

\begin{figure}
  \includegraphics[width=1.0\textwidth]{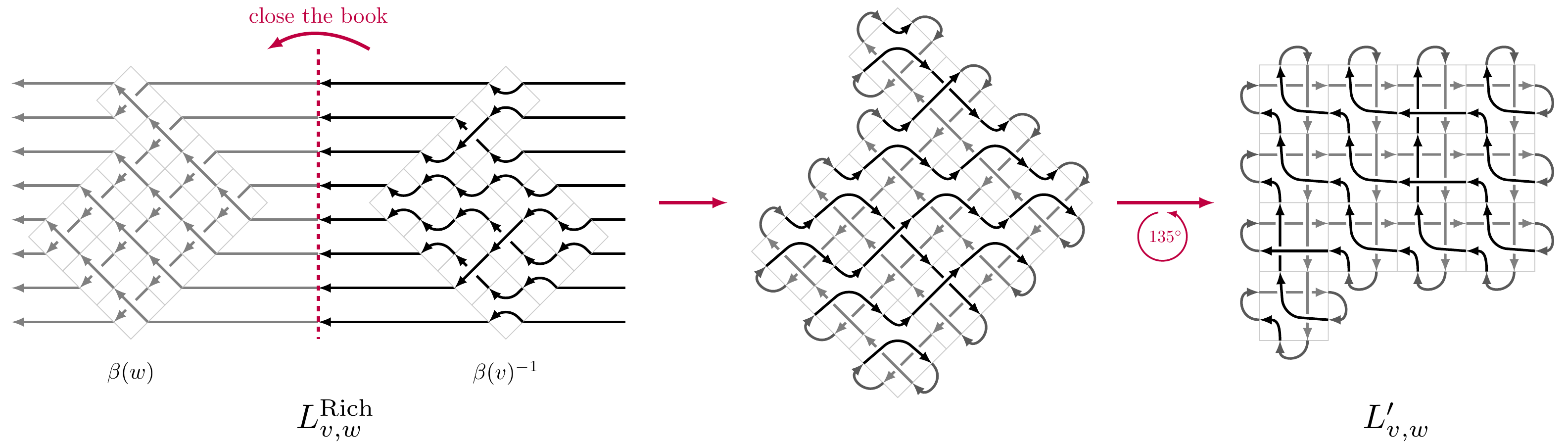}
  \caption{\label{fig:Lrich-to-L'} Converting the link $\Lrich_{v,w}$ into $\Lrichp_{v,w}'$; see \cref{sec:from-rich-links}.}
\end{figure}

Our first goal is to draw the link $\Lrich_{v,w}$ entirely within $\la$. To do that, we take $\beta(v)^{-1}$, reflect it along the vertical axis (flipping the over/under-crossings), and place the resulting braid on top of $\beta(w)$. We rotate the result counterclockwise by $135^\circ$. Every line segment in the boundary of $\la$ contains an endpoint of a strand in $\beta(v)^{-1}$ and an endpoint of a strand in $\beta(w)$. We join these strand endpoints, obtaining a link diagram drawn inside of $\la$. This link $\Lrichp_{v,w}'$, shown in \cref{fig:Lrich-to-L'}, is clearly isotopic to $\Lrichp_{v,w}$. Intuitively, the transformation $\Lrich_{v,w}\to\Lrichp_{v,w}'$ can be described as follows: open a book, and place the left (resp., right) half of \figref{fig:Lrich}(right) onto the left (resp., right) page of the book. Then close the book with your right hand, so that the front cover of the book is at the bottom. The page containing (the reflection of) $\beta(v)^{-1}$ will be on top of the page containing $\beta(w)$. Finally, rotate the closed book $135^\circ$ counterclockwise.

\begin{figure}
  \includegraphics[width=0.9\textwidth]{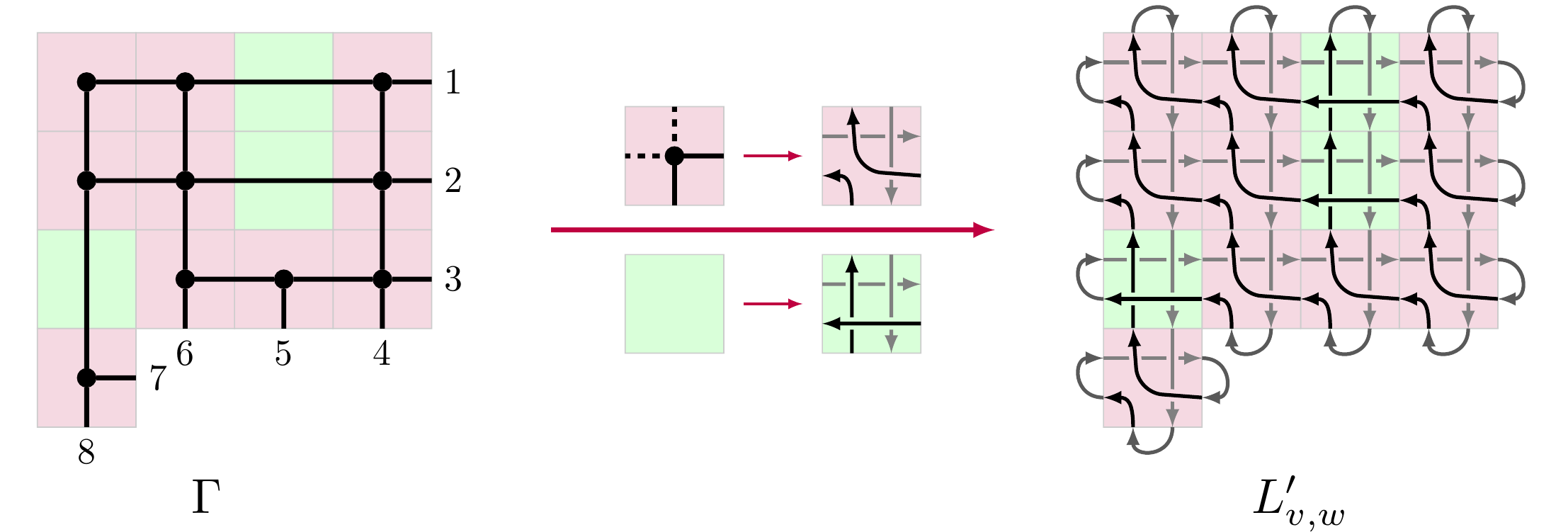}
  \caption{\label{fig:Gamma-to-L'} Converting a Le-diagram $\Gamma$ into a link $L'_{v,w}$.}
\end{figure}

An alternative description of the link $\Lrichp_{v,w}'$ can be obtained directly from the Le-diagram $\Gamma$ by replacing each box containing a dot as shown in \figref{fig:Gamma-to-L'}(middle top) and each box not containing a dot as shown in \figref{fig:Gamma-to-L'}(middle bottom). For example, the Le-diagram in \figref{fig:Gamma-to-L'}(left) gets converted into the link $\Lrichp_{v,w}'$ in \figref{fig:Gamma-to-L'}(right); this is the same link as in \figref{fig:Lrich-to-L'}(right). Since $\pi$ has no fixed points, each row and each column of $\Gamma$ contains at least one dot. Consider some column $j$ of $\Gamma$ and let $b_j$ be the highest box in it which contains a dot. The part of $\Lrichp_{v,w}'$ above $b_j$ contains a strand going vertically to the top boundary and then back, passing between all the horizontal strands that it encounters along the way. We may therefore contract this piece of the strand to be fully contained inside $b_j$, as shown in \figref{fig:pull}(far left). Repeating this procedure for each column $j=1,2,\dots,\la_1$, we obtain a link $\Lrichp_{v,w}''$ isotopic to $\Lrichp_{v,w}'$; see \figref{fig:pull}(middle).

\begin{figure}
  \setlength{\tabcolsep}{0pt}
\begin{tabular}{c|c}
\begin{tikzpicture}[baseline=(Z.base)]
\coordinate(Z) at (0,0);
\node(A) at (0,3.05){
  \includegraphics[width=0.2\textwidth]{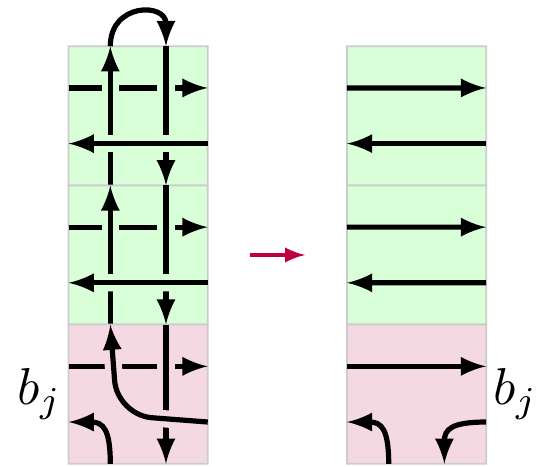}
};
\end{tikzpicture}
&
  \includegraphics[width=0.78\textwidth]{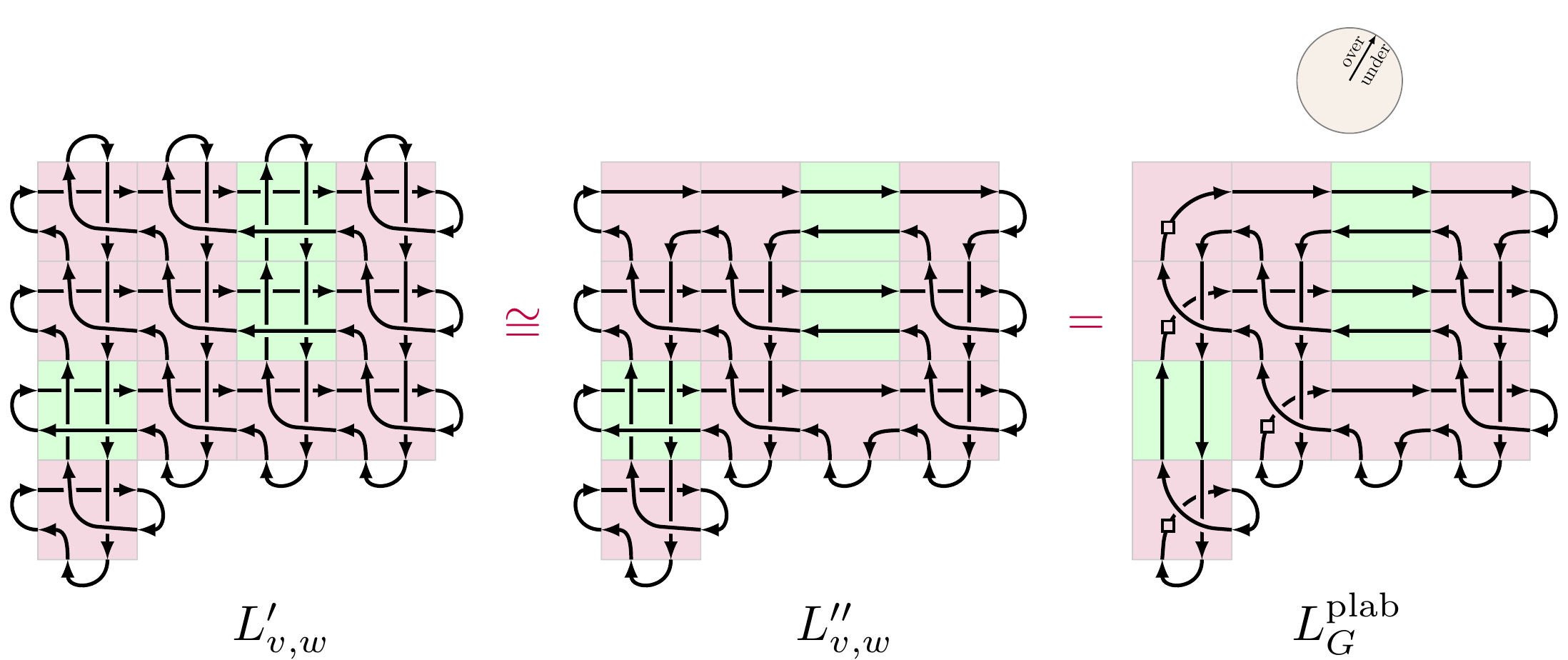}
\end{tabular}
  \caption{\label{fig:pull} An isotopy from $\Lrichp_{v,w}'$ to $\Lplab_G$.}
\end{figure}

Finally, we claim that the link $\Lrichp_{v,w}''$ is isotopic to the plabic graph link $\Lplab_G$. To see this, we choose the angle $\alpha$ from \cref{rmk:rotation} to be $80^\circ$. Then combining the rules from \cref{fig:Le_to_G} for converting the Le-diagram $\Gamma$ into the plabic graph $G$ with the description of $\Lplab_G$ given in \cref{sec:plabic_links_properties}, we see that a planar diagram of $\Lplab_G$ is obtained from $\Gamma$ via the rules shown in the last two rows of \cref{fig:Le_to_G}. Replacing each square (where $\arg(S,p)=\alpha$) by a horizontal segment going to the left boundary above all other strands and back below all other strands, we see that $\Lplab_G= \Lrichp_{v,w}''$; see \figref{fig:pull}(far right). To summarize, we have constructed explicit isotopies
\begin{equation*}%
  \Lrich_\pi=\Lrich_{v,w}\cong\Lrichp_{v,w}'\cong\Lrichp_{v,w}''=\Lplab_G.
\end{equation*}

\begin{figure}
  \includegraphics[width=0.5\textwidth]{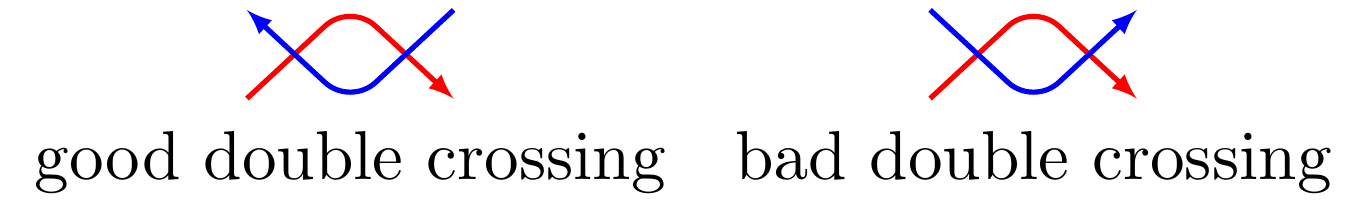}
  \caption{\label{fig:double} Good and bad double crossings of strands in plabic graphs.}
\end{figure}

\subsection{From plabic graph links to permutation links}
Our goal is to find an isotopy $\Lplab_G\cong \Lperm_\pi$. Let $D=D(G)$ be the oriented divide obtained from $G$ as in \cref{sec:plabic_links_properties}. Let $1,2,\dots,\n$ be the boundary vertices of $G$. Since $G$ contains no bad double crossings shown in \figref{fig:double}(right), it is clear that applying the moves from \cref{fig:divide_moves}, one can transform $D$ into an oriented divide $D'$ obtained by drawing a straight arrow $i\to \pi(i)$ for each $i\in[\n]$. The moves from \cref{fig:divide_moves} give rise to link isotopies, so $L_D\cong L_{D'}$. Let us move the boundary vertices $1,2,\dots,\n$ smoothly to the points $1',2',\dots,\n'$ which are located clockwise on the left semicircle of $\partial \disk$, i.e., on the subset of $\partial\disk$ with negative real part. Let $D''$ be the oriented divide obtained by drawing a straight arrow $i'\to \pi(i)'$ for each $i\in[\n]$. It follows that $L_{D'}\cong L_{D''}$.
 Let $i,j\in[\n]$ be such that the corresponding arrows $S_i:=(i'\to \pi(i)')$ and $S_j:=(j'\to \pi(j)')$ of $D''$ cross at some point $p$. Consider the arguments of the tangent vectors $0\leq \arg(S_i,p)\neq\arg(S_j,p)<2\pi$. Then it is easy to check that $\arg(S_i,p)>\arg(S_j,p)$ if and only if $\pi(i)<\pi(j)$. 
 Comparing to the description in \cref{sec:intro:perm_links}, we get that $L_{D''}\cong\Lperm_\pi$.

\subsection{From permutation links to toric permutation links}\label{sec:perm-to-toric}
Our goal is to find an isotopy $\Lperm_\pi\cong\Ltor_\pi$. Recall from \cref{rmk:grid} that $\Ltor_\pi$ is isotopic to the \emph{planar grid link} denoted $\Lgrid_\pi$; see \figref{fig:grid_iso}(a--c). In $\Lgridt_\pi$ and $\Lgrid_\pi$, the vertical arrows are drawn above the horizontal arrows. 

The main diagonal $y=1-x$ of the square $[0,1]^2$ divides $\Lgrid_\pi$ into the \emph{lower-triangular} and the \emph{upper-triangular} parts. In the lower-triangular part, all vertical arrows point down and all horizontal arrows point right, while in the upper-triangular part, all vertical arrows point up and all horizontal arrows point left. 
Let us take the lower-triangular part and reflect it around the main diagonal, placing it below the upper-triangular part; see \figref{fig:grid_iso}(c--d). We obtain a planar diagram of a link which we denote $L_\pi'$. Clearly, $L_\pi'\cong \Lgrid_\pi$. The diagram of $L_\pi'$ contains arrows pointing, up, left, down, and right. It follows that the over/under-crossings rule  in $L_\pi'$ is given by the following total order on the arrow directions: up $>$ left $>$ down $>$ right. In particular, $L_\pi'$ is a divide link for the angle $\alpha=45^\circ$ from \cref{rmk:rotation} (where the main diagonal is treated as part of the boundary of the disk), and we get that $L_\pi'\cong\Lperm_\pi$. 

\subsection{From toric permutation links to Coxeter links}\label{sec:tor_to_Cox}
First, we introduce a more general class of links, associated to arbitrary curves inside the $k\times(\n-k)$ rectangle.

We say that $C_g$ is a \emph{monotone curve} if it is the plot of a strictly monotone increasing function $g(x):[0,\n-k]\to[0,k]$ satisfying $g(0)=0$, $g(\n-k)=k$, and $g(j)\notin\Z$ for all $j=1,2,\dots,\n-k-1$. Thus, the curve $C_g$ passes through the points $(0,0)$ and $(\n-\k,k)$, but through no other lattice points. To any monotone curve $C_g$ we associate a link $L_g$ drawn on the surface of $\T$. The planar diagram of $L_g$ is obtained from the projection $\Cproj_g$ of $C_g$ to $\T$ via the following rule: whenever  two points $(x,g(x))$, $(x',g(x'))$ on $C_g$ (with $x<x'$) project to the same point of $\Cproj_g$, we draw the projection of $(x,g(x))$ above the projection of $(x',g(x'))$. 

The link diagram of $L_g$ is drawn on the torus $\T$, and therefore it is natural to think of $L_g$ itself as a link inside the \emph{thickened torus} $\T\times[0,1]$. Given a point $p=(x,g(x))\in C_g$, we denote by $\projtx p\in \T$ the projection of $p$ to $\T$, and we let $\left(\projtx p,1-\frac{x}{\n-k}\right)$ be the corresponding point in $\T\times[0,1]$. This yields an embedding of $C_g$ into $\T\times[0,1]$. Under this embedding, the points $(0,0)$ and $(\n-k,k)$ map to $((0,0),1)$ and $((0,0),0)$, respectively. Adding a line segment connecting $((0,0),1)$ to $((0,0),0)$, we obtain a representation of the link $L_g$ inside $\T\times[0,1]$. An obvious consequence of this construction is that the link $L_g$ only depends on the set of lattice points below the plot of $g$.
\begin{corollary} \label{cor:monotone_curves}
Let $C_g$, $C_h$ be two monotone curves, and assume that for each $j=1,2,\dots,\n-k-1$, we have $\lfloor g(j)\rfloor=\lfloor h(j)\rfloor$. Then the links $L_g$ and $L_h$ are isotopic. 
\end{corollary}

Suppose now that $\pi=\pi_C$ for a concave curve $C$. Our goal is to find an isotopy $\Ltor_\pi\cong\Lcox_C$. Comparing the above description to the one in \cref{sec:intro:cox_link}, we see that when $g$ is the concave down function satisfying $C=C_g$, we have $\Ltor_\pi\cong L_g$. Our goal is to choose a monotone curve $C_h$ such that $L_g\cong L_h$ by \cref{cor:monotone_curves}, and such that $L_h\cong \Lcox_C$. 

\begin{figure}
\includegraphics[width=1.0\textwidth]{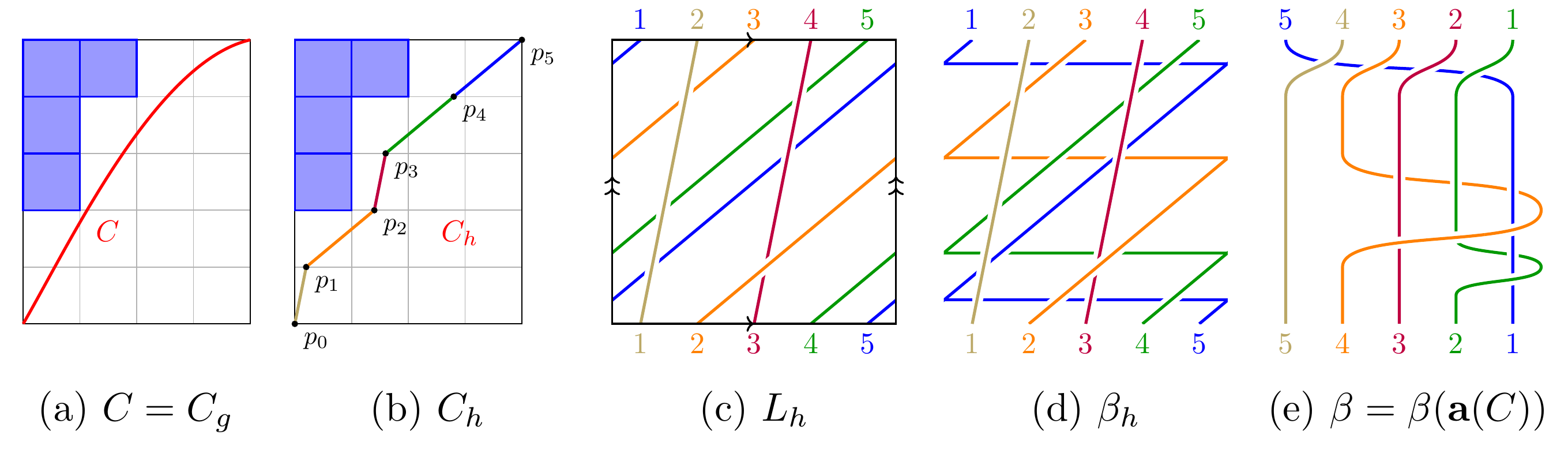}
  \caption{\label{fig:C_to_beta} Converting a concave curve $C$ into the Coxeter link $\beta(\abf(C))$; see \cref{sec:tor_to_Cox}.}
\end{figure}

Let $\la=\la_C=(\la_1\geq \la_2\geq\dots\geq \la_k=0)$ be the Young diagram above $C$. For $r=0,1,\dots,k-1$, let $p_r:= \left(\la_{k-r}+\frac rk,r\right)$. Thus, $p_0=(0,0)$. Let $p_{k}:=(\n-k,k)$. Let $h$ be the piecewise-linear function whose plot passes through the points $p_0,p_1,\dots,p_k$, and let $C_h$ be the corresponding monotone curve; see \figref{fig:C_to_beta}(b). Indeed, by \cref{cor:monotone_curves}, we have $L_g\cong L_h$. It remains to establish the isotopy $L_h\cong \Lcox_C$. 

Let $\beta=\beta(\abf(C))$ be the Coxeter braid associated to $C$, defined in~\eqref{eq:beta_abf_dfn}. It is a braid on $k$ strands. For $i=1,2,\dots,k$, let $S_i(\beta)$ be the strand of $\beta$ whose left endpoint is labeled by $i$. Thus, the right endpoint of $S_i(\beta)$ is labeled by $i+1$ for $i<k$ and by $1$ for $i=k$. On the other hand, the link diagram of $L_h$ in $\T$ intersects the line $y=0$ in exactly $k$ points with coordinates $\frac rk$ for $r=0,1,\dots, k-1$. We would like to view $L_h$ as the closure of a $k$-strand braid $\beta_h$ in the solid torus; see \figref{fig:C_to_beta}(c--d). The braid $\beta_h$ is obtained from $L_h$ via the following procedure: whenever a strand of $L_h$ passes through some point $(1,y)\in[0,1]^2$ and continues from the point $(0,y)\in[0,1]^2$, we connect the points $(0,y)$ and $(1,y)$ by a line segment that is drawn below all strands of $L_h$. The result is a drawing of a braid connecting $k$ points at the bottom of $[0,1]^2$ to $k$ points at the top of $[0,1]^2$. 

For $i=1,2,\dots, k-1$, let $S_i(\beta_h)$ be the piece of $\beta_h$ connecting the point $(\frac{i-1}k,0)$ to the point $(\frac ik,1)$. For $i=k$, let $S_k(\beta_h)$ be the remaining piece of $\beta_h$, connecting $(\frac{k-1}k,0)$ to $(0,1)$. See \figref{fig:C_to_beta}(d--e).

We claim that the braids $\beta_h$ and $\beta$ are isotopic. We compare them strand-by-strand, identifying $S_i(\beta)$ with $S_{k+1-i}(\beta_h)$ for $i=1,2,\dots,k$, in that order. (This correspondence is indicated via colors in \figref{fig:C_to_beta}(d--e).) For convenience, we rotate the drawing of $\beta$ $90^\circ$ counterclockwise. The strand $S_1(\beta)$ moves straight up during the $\ell_2^{a_2}\cdots \ell_k^{a_k}$ part, and then it moves left under all other strands. We apply an isotopy to make $S_k(\beta_h)$ move straight up towards the top boundary of $[0,1]^2$, and then move left under all other strands. Recall that we have a sequence $\abf(C)=(a_2,\dots,a_k)$ given by $a_i=\la_{i-1}-\la_i$ for $i=2,\dots,k$. Thus, the strand $S_2(\beta)$ wraps around $S_1(\beta)$ exactly $a_2$ times. On the other hand, $S_{k-1}(\beta_h)$ is the projection to $\T$ of the line segment of $C_h$ connecting $p_{k-1}$ to $p_k$. This projection intersects the vertical boundary of $\T$ exactly $a_2$ times. Therefore $S_{k-1}(\beta_h)$ wraps around $S_k(\beta_h)$ exactly $a_2$ times. Continuing in this fashion, we see that for $i=3,\dots,k$, the strand $S_i(\beta)$ wraps around the union $S_1(\beta)\cup S_2(\beta)\cup\dots\cup S_{i-1}(\beta)$ exactly $a_i$ times. On the other hand, the strand $S_{k+1-i}(\beta_h)$ wraps around the union $S_k(\beta_h)\cup S_{k-1}(\beta)\cup\dots\cup S_{k+2-i}$ exactly $a_i$ times. This shows that the braids $\beta$ and $\beta_h$ are isotopic, and therefore we get an isotopy $\Lcox_C\cong L_h$ between their closures.

\section{Quivers}
In this section, we introduce the point count polynomial $R(Q;q)$ of a \Louise %
 quiver, and study its basic properties.  We use this polynomial to define the \emph{$Q$-Catalan number} of $Q$, an integer that we conjecture to be nonnegative in \cref{conj:Catalan}.

\subsection{Quiver mutation}\label{sec:quiver}
Recall that a \emph{quiver} is a directed graph without directed cycles of length $1$ and $2$.

Given a quiver $Q$ and a vertex $j\in V(Q)$, one can define another quiver $Q'=\mu_j(Q)$ called the \emph{mutation of $Q$ at $j$}. This operation preserves the set of vertices:  $V(Q')=V(Q)$, and changes the set of arrows as follows:
\begin{itemize}
\item for every length $2$ directed path $i\to j\to k$ in $Q$, add an arrow $i\to k$ to $Q'$;
\item reverse all arrows in $Q$ incident to $j$;
\item remove all directed $2$-cycles in the resulting directed graph, one by one.
\end{itemize}

An \emph{ice quiver} is a quiver $\Qice$ whose vertex set $\Vice=V(\Qice)$ is partitioned into \emph{frozen} and \emph{mutable} vertices: $\Vice=\Vfro\sqcup \Vmut$. We automatically omit all arrows in $\Qice$ both of whose endpoints are frozen. For an ice quiver $\Qice$, its \emph{mutable part} is the induced subquiver of $\Qice$ with vertex set $\Vmut$.  %
 For a set $S\subset \Vmut$, let $\Qicefr[S]$ denote the ice quiver obtained from $\Qice$ by further declaring all vertices in $S$ to be frozen. We write $\Qice-S$ for the ice quiver obtained from $\Qice$ by removing the vertices in $S$.  For a simply-laced Dynkin diagram $D$, we say that $Q$ or $\Qice$ has type $D$, or is a $D$-quiver, if the underlying graph of $Q$ is isomorphic to $D$.

 Let $\Qice$ be an ice quiver with $\Vice=[n+m]$, $\Vmut=[n]$, and $\Vfro=[n+1,n+m]:=\{n+1,n+2,\dots,n+m\}$. We represent it by an $(n+m)\times n$ \emph{exchange matrix} $\BQice=(b_{ij})$, defined by
 \begin{equation*}%
   b_{ij}=\#\{\text{arrows $i\to j$ in $\Qice$}\} - \#\{\text{arrows $j\to i$ in $\Qice$}\}.
 \end{equation*}
The top $n\times n$ submatrix $\BQ$ of $\BQice$ is skew-symmetric, and is called the \emph{principal part} of $\BQice$.  

Let $\Bice$ be an $(n+m)\times n$ matrix. We denote $\cork(\Bice):=n-\rk(\Bice)$. Following~\cite{LS}, we say that an $(n+m)\times n$ matrix $\Bice$ is \emph{really full rank} if the rows of $\Bice$ span $\Z^n$ over $\Z$. We say that $\Qice$ is \emph{really full rank} if $\BQice$ is really full rank, and write $\cork(\Qice):=\cork(\BQice)$.   We say that $\Qice$ is \emph{torsion-free} if the abelian group $\Z^{n+m}/\BQice \Z^n$ is torsion-free, equivalently, if $\Z^n/\BQice^T \Z^{n+m}$ is torsion-free.  Thus, $\Qice$ is really full rank if and only if it is both full rank and torsion-free.

\begin{definition}
Given a quiver $Q$, we say that an ice quiver $\Qice$ with mutable part $Q$ is a \emph{minimal extension} of $Q$ if $\Qice$ has $m=\cork(Q)$ frozen vertices and is really full rank.
\end{definition}

\begin{lemma}\label{lem:ext}
If $Q$ is torsion-free then it admits a minimal extension.
\end{lemma}
\begin{proof}
If $Q$ is torsion-free, we can find vectors $d_1,\ldots,d_m \in \Z^n$ that form a basis of $\Z^n/B(Q)^T \Z^n$.  Define $\Qice$ with $m$ frozen vertices so that the bottom $m$ rows of the exchange matrix $\Bice$ are equal to $d_1,\ldots,d_m$.
\end{proof}

\begin{example}
The exchange matrices of the two quivers $Q,Q'$ in \figref{fig:LA}(c) are given by
$$
\BQ=\begin{pmatrix}
0 & 2 & -2 \\
-2&0&2\\
2&-2&0
\end{pmatrix} \qquad \text{and} \qquad
B(Q')=\begin{pmatrix}
0 & -1& -1 &2 \\
1&0&1 &-1\\
1&-1&0 &-1 \\
-2&1&1&0
\end{pmatrix}.
$$
Thus, $\cork(Q)=1$, $\cork(Q')=0$, and neither quiver is torsion-free.
\end{example}

We say that an ice quiver $\Qice$ is \emph{isolated} if it has no arrows. For example, the quiver in \figref{fig:simple}(c) is isolated.

We call an ice quiver $\Qice$ \emph{acyclic} if it has no directed cycles, and \emph{mutation acyclic} if it is mutation equivalent to an acyclic ice quiver.  For example, any quiver of Dynkin type is acyclic, as is the quiver in \figref{fig:LA}(a).  The quiver in \figref{fig:simple}(a) and the left quiver in \figref{fig:LA}(b) are not acyclic, but they are both mutation acyclic.

An edge $u\to v$ in a quiver $Q$ is called a \emph{separating edge}~\cite{Muller} if it does not belong to a bi-infinite walk in $Q$. Here, a \emph{bi-infinite walk} is a sequence $(w_j)_{j\in\Z}$ of vertices in $Q$ such that for each $j\in\Z$, $Q$ contains an arrow $w_j\to w_{j+1}$.

We define the class of \emph{\Louise quivers}, called ``Louise" in \cite{MSLA,LS}, as follows.
\begin{itemize}
\item Any isolated quiver $Q$ is \Louise.
\item Any quiver that is mutation equivalent to a \Louise quiver is \Louise.
\item Suppose that a quiver $Q$ has a separating edge $u\to v$, and that all three quivers $Q-\{u\}$, $Q-\{v\}$, $Q-\{u,v\}$ are \Louise. Then $Q$ is \Louise.
\end{itemize}
We say that an ice quiver $\Qice$ is \emph{\Louise} if its mutable part $Q$ is \Louise.    See \cref{fig:LA}.
\begin{figure}
  \includegraphics[width=\textwidth]{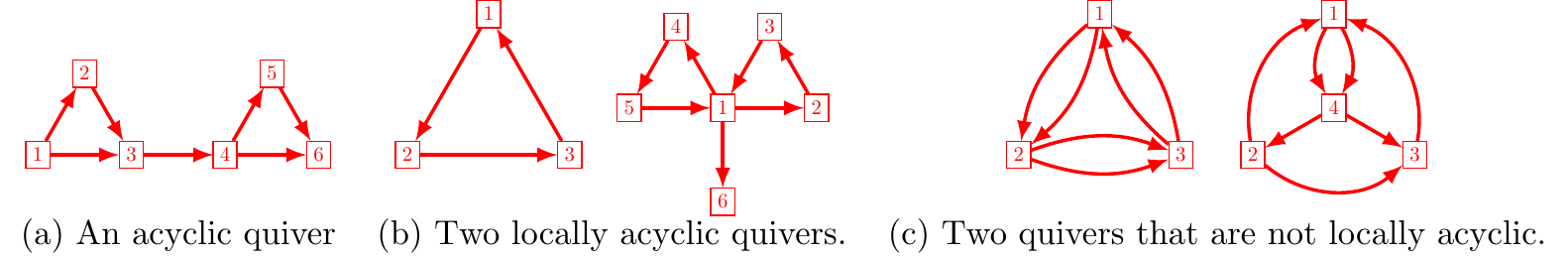}
  \caption{\label{fig:LA} Acyclic, locally acyclic and non-\Louise quivers.}
\end{figure}

\subsection{Cluster algebras}
Let $\Qice$ be an ice quiver with $\Vice=[n+m]$, $\Vmut=[n]$, and $\Vfro=[n+1,n+m]$.  We associate to $\Qice$ the initial seed $(x_1,x_2,\ldots,x_{n+m})$ of \emph{cluster variables}, considered as elements in the function field $\Fcal = \Q(x_1,x_2,\dots,x_{n+m})$.  The variables $x_{n+1},\ldots,x_{n+m}$ are called \emph{frozen variables}.  For a mutation $Q' = \mu_j(Q)$, we associate the seed $(x'_1,x'_2,\ldots,x'_{n+m})$, where $x'_i = x_i$ if $i \neq j$ and one has the new cluster variable
$$
x'_j = \frac{\prod_{r \to j} x_r + \prod_{j \to s} x_s}{x_j} \in \Fcal.
$$
By repeatedly mutating, we generate (possibly infinitely) many seeds and cluster variables. 
 We denote by $\AQice=\ABQice$ the cluster algebra associated to $\Qice$. This is the $\Z$-subalgebra of $\Q(x_1,x_2,\dots,x_{n+m})$ generated by all cluster variables and the inverses of frozen variables. 
 
The \emph{cluster variety} $\XQice=\XBQice$ is defined to be the scheme
\begin{equation*}%
\XQice:=\Spec(\AQice).
\end{equation*}
We say that $\AQice$ (resp., $\XQice$) is a cluster algebra (resp., cluster variety) \emph{of type $Q$}, where $Q$ is the mutable part of $\Qice$.  We say that $\AQice$ is isolated, acyclic, (really) full rank, if $\Qice$ is isolated, mutation acyclic, (really) full rank, respectively.  If $\Qice$ is a \Louise quiver, then $\AQice$ is locally acyclic in the sense of Muller~\cite{Muller}.

\begin{proposition}[{\cite[Theorem 7.7]{Muller}, \cite[Theorem 10.1]{LS}}]\label{prop:Mulsmooth}
Suppose that $\Qice$ is \Louise and really full rank.  Then $\XQice_\C$ is a smooth complex algebraic variety and $\XQice_{\bar \F_q}$ is smooth for any prime power $q$.
\end{proposition}

\begin{proposition}[{\cite[Corollary 5.4]{Muller}}]\label{prop:Mulcover}
Let $i \to j$ be a separating edge in $\Qice$.  Then the open sets $\{x_i \neq 0\}$ and $\{x_j \neq 0\}$ cover $\XQice$.
\end{proposition}

\begin{proposition}[{\cite[Proposition 3.1, Lemma 3.4, Theorem 4.1]{Muller}}]\label{prop:Mulloc}
Suppose that $i \in V(\Qice)$ is a mutable vertex such that $Q- \{i\}$ is \Louise.  Then $\AQice[x_i^{-1}] \simeq \AX(\Qice[i])$.
\end{proposition}

\subsection{Quiver point count}
For a mutable quiver $Q$, we define the function $\RQ(q)$ as follows. Choose a really full rank ice quiver $\Qice$ with mutable part $Q$ and $m$ frozen vertices. Then for $q$ a prime power, we set
\begin{equation}\label{eq:RQ_dfn}
  \RQ(q):=\frac{\#\XQice(\F_q)}{(q-1)^m}.
\end{equation}
For a rational function $R(q)=P(q)/Q(q)$, we let the \emph{degree} $\deg(R)$ be the difference $\deg(P)-\deg(Q)$. 
\begin{proposition}[{\cite[Proposition 5.11 and Theorem 10.5]{LS}}]\label{prop:Q_pcnt}
  Let $Q$ be a quiver with $n$ vertices.
  \begin{enumerate}
  \item The function $\RQ(q)$ defined in~\eqref{eq:RQ_dfn} does not depend on the choice of $\Qice$.
  \item Suppose $Q$ is \Louise. Then $\RQ(q)$ is a rational function in $q$ of degree $n$. %
  \end{enumerate}
\end{proposition}
If $u \to v$ is a separating edge in $Q$ and $Q-\{u\},Q-\{v\},Q-\{u,v\}$ are \Louise, then we have the recurrence
\begin{equation}\label{eq:RQ_recurrence}
  \RQ(q) = (q-1)\RQX{Q-\{u\}}{q}+ (q-1)\RQX{Q-\{v\}}{q}- (q-1)^2\RQX{Q-\{u,v\}}{q},
\end{equation}
which follows from Propositions~\ref{prop:Mulcover} and~\ref{prop:Mulloc}.

\begin{proposition}[{\cite[Proposition 3.9]{LS2}}] \label{prop:acycliccount}
Let $Q$ be an acyclic quiver with $n$ vertices.  Then 
$$
\RQ(q) = \sum_{k \geq 0} a_k (q-1)^{n-2k} q^k,
$$
where $a_k$ is the number of independent sets of size $k$ in the underlying undirected graph of $Q$.
\end{proposition}

When $Q$ itself is really full rank, $\RQ(q)$ is a polynomial in $q$. However, $\RQ(q)$ is in general a genuine rational function whose denominator is a power of $(q-1)$. E.g., for $Q$ a single isolated vertex (we denote this quiver by $A_1$; it has corank $1$), we have
\begin{equation}\label{eq:A1}
  \RQX{A_1}{q}=\frac{q^2-q+1}{q-1}.
\end{equation}
For an ice quiver $\Qice$ with mutable part $Q$ and $m$ frozen vertices, we set
\begin{equation*}%
  \RQice(q):=(q-1)^m\RQ(q).
\end{equation*}
Thus, when $\Qice$ is really full rank and \Louise, $\RQice(q)$ is a polynomial in $q$ of degree $n+m$ satisfying $\#\XQice(\F_q)=\RQice(q)$ for all prime powers $q$.
\begin{definition}\label{def:QCat}
Let $Q$ be a torsion-free quiver, and $\Qice$ be a minimal extension as in \cref{lem:ext}.  If $\RQice(q)$ is a polynomial in $q$, define %
  \begin{equation*}%
    \chiQ:=\RQice(1).
  \end{equation*}
  By \cref{prop:Q_pcnt}, $\chiQ$ does not depend on the choice of $\Qice$.
\end{definition}
We call $\chiQ$ the \emph{Catalan number} of the quiver $Q$, or simply the \emph{$Q$-Catalan number}.

\begin{conjecture}\label{conj:Catalan}
Suppose that $Q$ is a \Louise quiver. %
Then $\chiQ \geq0$. %
\end{conjecture}

We note that $\chiQ$ can be zero.  For example, let $Q$ be an acyclic orientation of the three-cycle, and let $\Qice$ be a minimal extension. Thus, $\Qice$ is a really full rank quiver with $\cork(Q)=1$ frozen vertices.
 Then by \cref{prop:acycliccount}, we have $R(Q;q)=(q-1)^3+3(q-1)q$, and thus $R(\Qice;q) = (q-1)^4+3(q-1)^2q$ and  $\chiQ =0$.  For another example, the quiver $Q$ in \figref{fig:LA}(a) satisfies $\cork(Q) = 0$ and $\chiQ = 0$.  %
  
\begin{example}\label{ex:Dynkin}
In \cref{tab:Dynkin}, which can be computed using \cref{prop:acycliccount} or \cref{cor:RQ_skein}, we take $Q$ to be any orientation of a Dynkin diagram and $\Qice$ to be any minimal extension of $Q$. Note that $R(Q;q) = R(\Qice;q)/(q-1)^m$, where $m=\cork(Q)$ is given in \cref{tab:Dynkin}.

The corresponding \FLY polynomials are computed via the following observation, which amounts to a single application of~\eqref{eq:HOMFLY_dfn}. Let $G$ be a connected simple plabic graph such that $Q:=\QG$ is an orientation of a tree, and let $Q'$ (resp., $Q''$) be obtained from $Q$ by adjoining a path of length $1$ (resp., of length $2$) to some vertex of $Q$. Each of $Q'$ and $Q''$ is the planar dual of some simple plabic graph, and by \cref{lemma:bdry_leaf_skein_FLY}, the \FLY polynomials of the associated links are related by
\begin{equation}\label{eq:HOMP_recurrence_Dynkin}
  \HOMP(L'')=\frac za \HOMP(L')+\frac1{a^2} \HOMP(L).
\end{equation}
For example, we can take $(Q'',Q',Q)$ to be one of $(A_n,A_{n-1},A_{n-2})$, $(D_n,D_{n-1},D_{n-2})$, $(E_6,D_5,A_4)$, $(E_7,E_6,D_5)$, or $(E_8,E_7,E_6)$. Writing $\HOMP(Q)$ in place of $\HOMP(L)$ where $L$ is the link corresponding to $Q$, we find:
\begin{equation*}
  \HOMP(A_0)=1,\quad \HOMP(A_1)=\frac{z + z^{-1}}{a} - \frac{z^{-1}}{a^{3}},\quad \HOMP(A_n)=\frac za \HOMP(A_{n-1})+\frac1{a^2} \HOMP(A_{n-2}) \quad\text{for $n\geq2$}.
\end{equation*}
Here $A_0$ denotes the empty quiver whose associated link is the unknot; the links corresponding to $A_1$ and $A_2$ are the Hopf link and the trefoil knot, respectively; see \cref{fig:intro_ex}. The above recurrence can be solved explicitly: for each $n\geq 1$, we have
\begin{equation*}%
  \HOMP(A_n)=\frac{T_{n+1}(z)}{a^{n-1}}-\frac{T_{n-1}(z)}{a^{n+1}}, \quad\text{where}\quad T_n(z):=\sum_{k=0}^{\lfloor n/2\rfloor} {n-k\choose k}z^{n-2k-1}.
\end{equation*}
Similarly, we find
\begin{align*}
  \HOMP(D_2)&=\left(\frac{z + z^{-1}}{a} - \frac{z^{-1}}{a^{3}} \right)^2, \quad 
  \HOMP(D_3)=\HOMP(A_3)=\frac{z^{3} + 3 z + z^{-1}}{a^{3}} - \frac{z + z^{-1}}{a^{5}},\\
 \HOMP(D_n)&=\frac za \HOMP(D_{n-1})+\frac1{a^2} \HOMP(D_{n-2})\quad\text{for $n\geq4$}.
\end{align*}
Here $D_2$ denotes the quiver consisting of two isolated vertices\footnote{One can check that the values for $D_2$ and $D_3$ are correct by computing the \FLY polynomial directly from the plabic graph links for $D_4$ and $D_5$ and then running the recurrence~\eqref{eq:HOMP_recurrence_Dynkin} backwards.} and the associated link is by definition the connected sum of two Hopf links; cf. \cref{prop:FLYsum} and~\eqref{eq:Hopf}. We also set $D_3:=A_3$. Finally, having computed $\HOMP(A_4)$ and $\HOMP(D_5)$, we find
\begin{align*}
  \HOMP(E_6)&=\frac{z^{6} + 6  z^{4} + 10  z^{2} + 5}{a^{6}} - \frac{z^{4} + 5  z^{2} + 5}{a^{8}} + \frac{1}{a^{10}}; \\
  \HOMP(E_7)&=\frac{z^{7} + 7  z^{5} + 15  z^{3} + 11  z + 2z^{-1}}{a^{7}} - \frac{z^{5} + 6  z^{3} + 9  z + 3z^{-1}}{a^{9}} + \frac{z + z^{-1}}{a^{11}}; \\
  \HOMP(E_8)&=\frac{z^{8} + 8  z^{6} + 21  z^{4} + 21  z^{2} + 7}{a^{8}} - \frac{z^{6} + 7  z^{4} + 14  z^{2} + 8}{a^{10}} + \frac{z^{2} + 2}{a^{12}}.
\end{align*}
Substituting $a:=\qqi$ and $z:=\qq-\qqi$ into the top $a$-degree term in the above formulas, one recovers the point count formulas in \cref{tab:Dynkin}, in agreement with \cref{thm:FLY=pcnt_leaf_rec}.
\end{example}

\begin{table}
\begin{tabular}{|c|c|c|c|} \hline
&&& \\[-12pt]
Dynkin type& $m$ & $R(\Qice;q)$ & $\chiQ$ 

\\
\hline\hline
$A_n$, $n$ even & 0 & $1+q^2+ q^4 + \cdots +q^{n}$ & $n/2$   
\\
\hline
$A_n$, $n$ odd & 1 & $1-q + q^2 - q^3 + q^4 - \cdots +q^{n+1}$ & $1$   
\\
\hline
$D_n$, $n \geq 4$ even & 2 & $(1-q + q^2 - q^3 + q^4 - \cdots +q^{n+2}) + (-q+q^2+q^n-q^{n+1})$ & $1$   
\\
\hline
$D_n$, $n \geq 5$ odd & 1 & $(1+q^2+ q^4 + \cdots +q^{n+1}) - (q+q^n)$ & $(n-1)/2$   
\\
\hline
$E_6$ & 0 & $1+q^2+ q^3+ q^4 +q^{6}$ & $5$   
\\
\hline
$E_7$ & 1 & $1-q +q^2+ q^6 -q^7+q^8$ & $2$   
\\
\hline
$E_8$ & 0 & $1+q^2+ q^3+ q^4 + q^5+q^6+q^8$ & $7$   
\\
\hline
\end{tabular}
\caption{\label{tab:Dynkin} Point counts of Dynkin quivers.}
\end{table}

\begin{remark}
Let $G$ be a reduced plabic graph with $N$ boundary vertices such that $\pi_G = \pi_{k,N}$ is the ``top cell'' permutation defined in \cref{sec:bound-affine-perm}.  For $(k,N) = (2,N),$ $(3,6),$ $(3,7),$ $(3,8)$, we get the quivers $\QG = A_{N-3}, D_4, E_6, E_8$ respectively.  If $\gcd(k,N) = 1$, then it is shown in \cite{GL_qtcat} that $m= 0$ and $\chi_{\QG} = \frac{1}{N}\binom{N}{k}$ is the $(k,N-k)$ rational Catalan number.  Compare with \cref{tab:Dynkin}.
\end{remark}

\subsection{Leaf recurrence}
Let $i\in V(Q)$ be a leaf vertex of a quiver $Q$, connected to some other vertex $j\in V(Q)$. We call $i$ a \emph{\Louiseleaf} if both quivers $Q-\{i\}$ and $Q-\{i,j\}$ are \Louise.  This automatically implies that $Q$ itself is \Louise.

\begin{proposition}\label{prop:Louise_leaf}
Fix a field $\Ga$, and consider cluster varieties over $\Ga$.
Let $\Qice$ be a really full rank ice quiver with $m$ frozen vertices and mutable part $Q$. Let $i\in V(Q)$ be a \Louiseleaf in $Q$ connected to $j\in V(Q)$. Then the subvariety of $\XQice$ given by $x_i=0$ is isomorphic to $\Ga\times \XX(\Qice')$, where $\Qice'$ is a really full rank quiver with $m$ frozen vertices and mutable part $Q':=Q-\{i,j\}$.
\end{proposition}
\begin{proof}%
Our goal is to construct a quiver $\Qice'$ satisfying the above assumptions and an isomorphism
\begin{equation}\label{eq:leaf1}
  \{\bx\in\XX(\Qice)\mid x_i=0\} \cong \F\times \XX(Q').
\end{equation}

  By Proposition~\ref{prop:Mulcover}, the subvariety of $\XX(\Qice)$ given by $x_i=0$ is contained inside the subvariety of $\XQice$ given by $x_j\neq0$. The quiver $Q-\{j\}$ is the disjoint union of $Q'$ and a single isolated vertex. By assumption, $Q'$ is \Louise and thus so is $Q-\{j\}$. By Proposition~\ref{prop:Mulloc}, we have $\AQice[x_j^{-1}]\cong \AX(\Qicefr[j])$. We therefore find
\begin{equation}\label{eq:leaf2}
    \{\bx\in\XX(\Qice)\mid x_i=0\} \cong \{\bx\in\XX(\Qicefr[j])\mid x_i=0\}.
\end{equation}

Let $A= \AX(\Qicefr[j] - \{i\})[x_i^{\pm 1}]$ be obtained from the cluster algebra $ \AX(\Qicefr[j] - \{i\})$ by adjoining $x_i$ and its inverse.  The cluster algebra $\AX(\Qicefr[j])$ is isomorphic 
 to the subalgebra of $A$ generated by $\AX(\Qicefr[j] - \{i\})$, the cluster variable $x_i$ and the mutation $x'_i$, which is of the form $f/x_i$ where $f \in \AX(\Qicefr[j] - \{i\})$.  

Let $\X':=\XX(\Qicefr[j] - \{i\})$.  It follows that 
\begin{equation}\label{eq:leaf2.5}
  \XX(\Qicefr[j])\cong \{\left((x_i,x_i'),\bx'\right)\in\Ga^2\times\X'\mid x_ix_i'=M+x_jM'\},
\end{equation}
where $M,M'$ are monomials in the frozen variables of $\Qicefr[j] - \{i\}$. We therefore see that the right-hand side of~\eqref{eq:leaf2} is given by 
\begin{equation}\label{eq:leaf3}
  \{\bx\in\XX(\Qicefr[j])\mid x_i=0\}
\cong \{\left((x_i,x_i'),\bx'\right)\in\Ga^2\times\X'\mid x_i=0 \text{ and } x_ix_i'=M+x_jM'\}.
\end{equation}
We see that the variable $x_i'$ is free, and thus the right-hand side of~\eqref{eq:leaf3} is isomorphic the direct product of $\Ga$ and 
\begin{equation} \label{eq:yM'/M=1} 
\left\{\bx'\in \X'\ \middle|\  \frac{x_jM'}{M}=-1\right\}.
\end{equation}
  It remains to show that the locus~\eqref{eq:yM'/M=1} is isomorphic to $\XX(\Qice')$ for some quiver $\Qice'$ satisfying the conditions in \cref{prop:Louise_leaf}. By assumption, $\Bice:=\BQice$ is really full rank. Thus, there exist integers $(\alpha_k)_{k\in[n+m]}$ such that 
  \begin{equation*}%
    \sum_{k} \alpha_k \bice_k=e_i,
  \end{equation*}
where $\bice_k$ denotes the $k$-th row of $\Bice$ and $e_i=(0,\dots,0,1,0,\dots,0)$ is the $i$-th standard basis vector in $\Z^n$. 

Since $i$ is a leaf, the row $\bice_i$ has a single nonzero entry in the $j$-th column. Let $\Bice\pA$ be obtained from $\Bice$ by removing the $i$-th row and the $j$-th column.\footnote{When removing rows and columns from matrices, we preserve the labels of the remaining rows and columns. For instance, removing row $1$ from $\Bice$ produces a matrix with rows labeled $2,\dots,n+m$.} Then we have
  \begin{equation*}%
    \sum_{k\neq i} \alpha_k \bice\pA_k=\bar e_i,
  \end{equation*}
where $\bar e_i$ is the $i$-th standard basis vector in $\Z^{n-1}$. Moreover, the matrix $\Bice\pA$ is still really full rank.

Let $\Vfro:=[n+1,n+m]$ be the set of frozen vertices of $\Qice$. Then the set of frozen vertices of $\Qicefr[j]-\{i\}$ is $\Vfro\sqcup\{j\}$.  We have $\gcd(\alpha_k\mid k\in \Vfro\sqcup\{j\})=1$
 since all the rows with a nonzero entry in column $i$ of $\Bice$ belong to $\Vfro\sqcup\{j\}$.  By \cite[Section 5]{LS}, replacing a frozen row by its negative, or adding a frozen row to another frozen row produces an isomorphic cluster variety.  After applying a series of such transformations, we obtain a really full rank $(n+m-1)\times(n-1)$ exchange matrix $\Bice\pB$ and a family of integers $(\alpha\pB_k)_{k\neq i}$ satisfying 
   \begin{equation}\label{eq:alpha_kBice''_k=e_i}
    \sum_{k\neq i} \alpha\pB_k \bice\pB_k=\bar e_i,
  \end{equation}
such that moreover for some $\ell\in\Vfro\sqcup\{j\}$ we have $\alpha\pB_\ell=1$. 

Now, let $\Bice\pC$ be the $(n+m-1)\times(n-2)$ matrix obtained from $\Bice\pB$ by further removing the $i$-th column. Thus, $\AX(\Bice\pC)\cong \AX(\Qicefr[j] - \{i\})$, and by construction, $\sum_{k\neq i} \alpha\pB_k\bice\pC_k=0$ is the zero vector in $\Z^{n-2}$.  Thus, by~\cite[Proposition 5.1]{LS}, the numbers $(\alpha\pB_k)_{k\neq i}$ describe a one-parameter subgroup $\Gm$ of the cluster automorphism torus $T=\Aut(\AX(\Bice\pB))$ (see \cite[Section 5]{LS}), with $z\in\Gm$ acting by $x_k\mapsto z^{\alpha\pB_k}x_k$. Equation~\eqref{eq:alpha_kBice''_k=e_i} implies that $z$ acts on the monomial $\frac{x_jM'}{M}$ by $z^{\pm 1}$. Thus, every $\Gm$-orbit contains a unique point in the locus~\eqref{eq:yM'/M=1}. In other words, the locus~\eqref{eq:yM'/M=1} is isomorphic to the quotient $\X(\Bice\pB)/\Gm$.  Since we assumed that $\alpha\pB_\ell = 1$, the quotient $\X(\Bice\pB)/\Gm$ can be obtained by setting $x_\ell=1$.  Let $\Bice'$ be the $(n+m-2)\times(n-2)$ obtained from $\Bice\pC$ by removing the $\ell$-th row. Let $\Qice'$ be the associated quiver, with $m$ frozen vertices corresponding to the rows $(\Vfro\sqcup\{j\})\setminus\{\ell\}$ and mutable part $Q-\{i,j\}$. Clearly, $\Qice'$ is still really full rank. We find that indeed the locus~\eqref{eq:yM'/M=1} is isomorphic to $\XX(\Qice')$.
\end{proof}

\begin{remark}
\Cref{prop:Louise_leaf} generalizes to the case of $i$ being either a source or a sink in $Q$.
\end{remark}
The following result will be later compared to the \FLY skein relation~\eqref{eq:HOMFLY_dfn}.
\begin{corollary}\label{cor:RQ_skein}
  Let $Q$ be really full rank, and let $i\in V(Q)$ be a \Louiseleaf in $Q$ connected to $j\in V(Q)$. Let $Q':=Q-\{i\}$ and $Q'':=Q-\{i,j\}$. Then
  \begin{equation}\label{eq:RQ_skein}
    \RQ(q)=(q-1)\RQX{Q'}{q} + q\RQX{Q''}{q}.
  \end{equation}
\end{corollary}
\begin{proof}
  Let $\Qice$ be really full rank with mutable quiver $Q$, and $\X:=\XQice$. Let $\Ucal:=\{x_i\neq0\}$ and $\Zcal:=\{x_i=0\}$ be two subvarieties of $\X$. Then $\Ucal$ is a really full rank cluster variety of type $Q'$ and $\Zcal$ is described in \cref{prop:Louise_leaf}. The result follows since
  \begin{equation*}%
    \#\Xcal(\F_q)=\#\Ucal(\F_q) + \#\Zcal(\F_q).    \qedhere
  \end{equation*}
\end{proof}

\section{Quivers, links, and plabic graphs}
In this section, we study some basic relations between the properties of plabic graphs and the associated quivers and links.

\subsection{Plabic graph quivers}\label{sec:plabquiver}
Recall the background on plabic graphs from \cref{sec:intro:plabic_link,sec:plabic-graphs}. We continue to assume that each plabic graph $G$ is trivalent, and that each interior face of $G$ is simply connected.

We say that two plabic graphs are \emph{move equivalent} if they are connected by a sequence of moves in \cref{fig:moves}. For the square move in \figref{fig:moves}(b), we require that the four faces $A,B,C,D$ adjacent to the square face in clockwise order satisfy $A\neq B\neq C\neq D\neq A$. (Having $A=C$ or $B=D$ is allowed.) Cf.~\cite[Restriction~6.3]{FPST}; this restriction is needed in order for natural statements such as \cref{prop:simple_vs_moves} to hold.

\begin{proposition}\label{prop:trivred} 
Let $G$ be a connected plabic graph. Then $G$ can be transformed using only tail removal (\figref{fig:moves}(c)) into:
\begin{itemize}
\item the graph in \figref{fig:trivred}(a) if $G$ has no interior faces;
\item the graph in \figref{fig:trivred}(b) or \figref{fig:trivred}(c) if $G$ has one interior face;
\item a plabic graph with no boundary vertices, if $G$ has two or more interior faces.
\end{itemize}
We call this graph the \emph{tail reduction} of $G$. %
\end{proposition}
\begin{proof}
Let $G$ be a connected plabic graph.  By repeatedly applying tail removal  to $G$ (without changing connectedness or the number of interior faces), we may assume no more tail removals are possible.  If $G$ has no boundary vertices then $G$ must have two or more interior faces.  If $G$ has a boundary vertex $b$ connected to an interior vertex $v$, then $v$ must be incident to a loop edge, for otherwise we may apply a tail removal to $b$.  In this case, $G$ has one interior face.  Finally, if $G$ has no interior vertices, then $G$ consists of a single edge connecting two boundary vertices, and has no interior faces.
\end{proof}

\begin{figure}
  \includegraphics[width=0.6\textwidth]{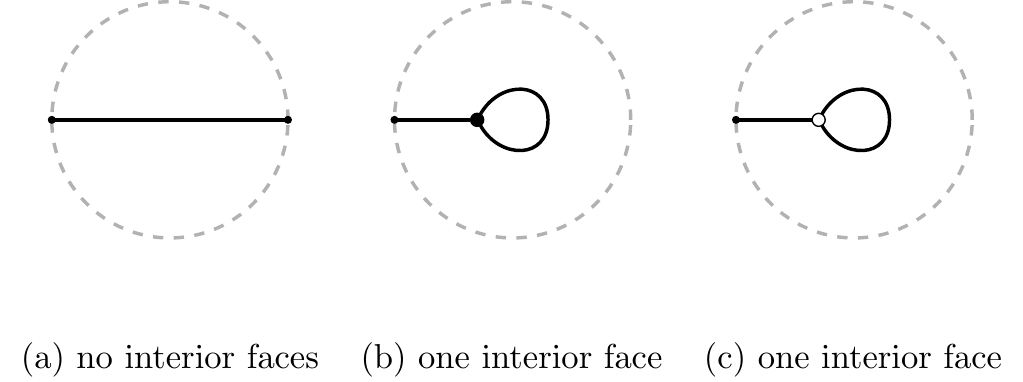}
  \caption{\label{fig:trivred} Tail reductions that have boundary vertices.}
\end{figure}

\begin{proposition}\ \label{prop:simple_vs_moves}
\begin{itemize}
\item Suppose that two simple plabic graphs $G$ and $G'$ are move equivalent. Then $\QG$ and $\QX(G')$ are mutation equivalent.  
\item Any plabic graph $G'$ move equivalent to a simple plabic graph $G$ is also simple.
\end{itemize}
\end{proposition}
\begin{proof}
Direct check.
\end{proof}

\begin{figure}
  \includegraphics[width=0.65\textwidth]{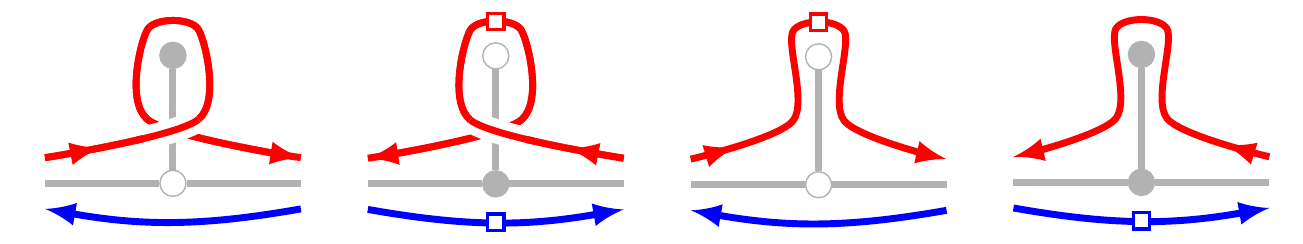}
  \caption{\label{fig:interior_leaves} The behavior of the link $\Lplab_G$ at interior leaves.}
\end{figure}

Recall from \cref{sec:intro:quivers-point-count} that each plabic graph gives rise to a link $\Lplab_G$ and that our main conjecture (\cref{conj:main}) yields a relation between the polynomials $\RQG(q)$ and $P(\Lplab_G;a,z)$.  For a rational function $R(q) = P(q)/Q(q)$, the \emph{leading coefficient} of $R(q)$ is defined as the ratio of the coefficient of $q^{\deg P}$ in $P(q)$ and the coefficient of $q^{\deg Q}$ in $Q(q)$. While \cref{conj:main} applies only to simple plabic graphs (cf. \cref{sec:counterex}), the following more basic statement appears to hold for arbitrary plabic graphs.  In particular, we allow plabic graphs with interior leaves, though we still insist that interior faces are simply connected.  The link $\Lplab_G$ is defined as before, and the behavior of the strands at an interior leaf is shown in \cref{fig:interior_leaves}. 
\begin{conjecture}\label{conj:top_a_deg}
Let $G$ be a plabic graph with $\conn(G)$ connected components and $n$ interior faces. Then the top $a$-degree of $P(\Lplab_G;a,z)$ equals
 \begin{equation}\label{eq:deg_a}
  \degtop_a(P(\Lplab_G))=\conn(G)-n-1.
 \end{equation}
The degree of $\Ptop_L(q)$ (as a rational function in $q$) is given by
 \begin{equation}\label{eq:a_top}
\deg(\Ptop_L(q)) = n,
 \end{equation}
 and the leading coefficient of $\Ptop_L(q)$ is equal to $1$.
\end{conjecture}

Cluster varieties $\XQice$ are irreducible, with dimension equal to the number of vertices in $\Qice$.  Thus, when $R(Q;q)$ is a rational function, it has leading coefficient $1$ and degree equal to the number of vertices of $Q$.  So \cref{conj:main} implies \eqref{eq:a_top} for simple plabic graphs. %

\subsection{Empty quivers}
\begin{proposition}\label{prop:empty}
  Let $G$ be a connected plabic graph. The following are equivalent:
  \begin{enumerate}
  \item $\QG$ is empty.
  \item $G$ is a tree.
  \item\label{item:empty3} The tail reduction of $G$ is a graph with no interior vertices.
  \end{enumerate}
\end{proposition}
\begin{proof}
The mutable quiver $\QG$ is empty if and only if $G$ has no interior faces, which is equivalent to $G$ being a tree, since $G$ is assumed to be connected.  By Proposition~\ref{prop:trivred}, this is also equivalent to~\eqref{item:empty3}.
\end{proof}

\subsection{Disconnected graphs}
For two plabic graphs $G',G''$, we let $G'\disjun G''$ denote the disconnected plabic graph obtained by taking the disjoint union of $G'$ and $G''$. We denote by the same symbol the disjoint union of two quivers and of two links.
\begin{proposition}\label{prop:disc}
Suppose $G=G'\disjun G''$ is a disconnected plabic graph. Then $\QG=\QX(G')\disjun\QX(G'')$ and $\Lplab_G=\Lplab_{G'}\disjun \Lplab_{G''}$.
\end{proposition}
\begin{proof}
Follows directly from the definitions.
\end{proof}

\subsection{Isolated plabic graph quivers}
\def\LHopf{L_{{\rm Hopf}}}
Recall that the \emph{connected sum} of two oriented knots $K, K'$ is defined as follows.  Find planar projections of the two knots that are disjoint.  Then find a (topological) rectangle $R$ in the plane, with oriented boundary consisting of sides $S_1,S_2,S_3,S_4$ in order, such that $S_1$ (resp. $S_3$) is an oriented arc along $K$ (resp. $K'$), and $S_2,S_4$ are disjoint from $K$ and $K'$.  The oriented connected sum knot $K \#K'$ is obtained by deleting $S_1,S_3$ and adding $S_2,S_4$.  To construct the connected sum of two links $L, L'$, we choose components $K \subset L$ and $K' \subset L'$ and perform the above surgery.  We call any link produced in this way the \emph{connected sum} of $L$ and $L'$, denoted $L \#L'$. The following result is well known; see e.g.~\cite[Proposition~16.2]{Lickorish}.

\begin{proposition}\label{prop:FLYsum}
Let $L$ and $L'$ be two oriented links.  Then $\HOMP(L \disjun L') = \left(\frac{a-a^{-1}}{z}\right)\HOMP(L)\HOMP(L')$ and $\HOMP(L \#L') = \HOMP(L) \HOMP(L')$. 
\end{proposition}
\noindent In particular, the $r$-component unlink has HOMFLY polynomial $\left(\frac{a-a^{-1}}{z}\right)^{r-1}$.

Let $\LHopf:=\hopflink$ denote the positively oriented Hopf link.  Then we have 
\begin{equation}\label{eq:Hopf}
\HOMP({\LHopf};a,z) = \frac{z + z^{-1}}{a} - \frac{z^{-1}}{a^{3}}.
\end{equation}

Recall that a quiver $Q$ is \emph{isolated} if it has no arrows.
\begin{proposition}\label{prop:plabic_isolated}
Let $G$ be a simple plabic graph with $n$ interior faces and $\conn(G)$ connected components. Suppose that $\QG$ is isolated. Then $\Lplab_G$ is a disjoint union of $\conn(G)$ links, each of which is a connected sum of some number (possibly zero) of positively oriented Hopf links. 
\end{proposition}
\begin{proof}
By Proposition~\ref{prop:disc}, we may assume that $G$ is connected.  Replace $G$ by its tail reduction (Proposition~\ref{prop:trivred}).
Then the result holds when $n = 0$: the link $\Lplab_G$ is an unknot.  The result also holds when $n =1$: the link $\Lplab_G$ is the Hopf link.  So suppose that $n > 1$.  In particular, we are assuming that $G$ has no boundary vertices.

\begin{figure}
\includegraphics[width=1.0\textwidth]{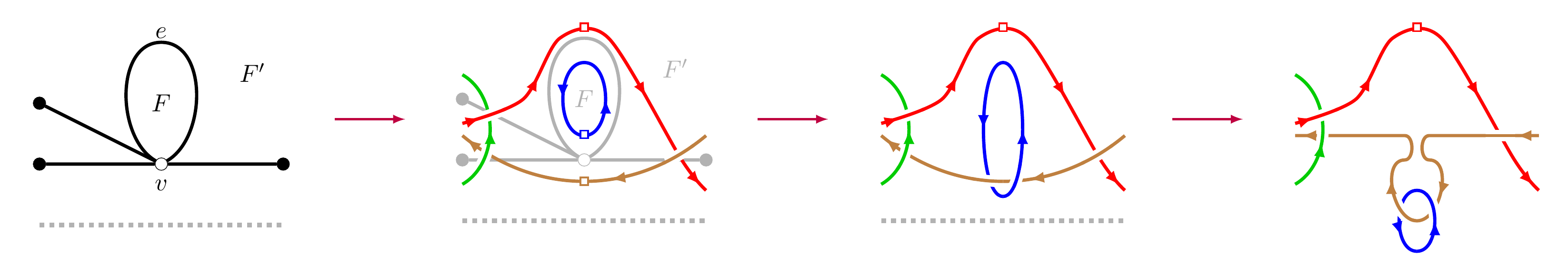}
  \caption{\label{fig:loop_removal} A loop edge in $G$ whose sole vertex is adjacent to a boundary face results in having the Hopf link as a connected summand of $\Lplab_G$; see the proof of \cref{prop:plabic_isolated}.}
\end{figure}

For the remainder of the proof it is convenient to replace $G$ by its \emph{bipartite reduction}, obtained by contracting all interior edges whose endpoints have the same color. In particular, interior vertices are allowed to have degree higher than three.  Suppose first that $G$ contains an interior face $F$ which is bounded by a loop $e$ whose sole endpoint is $v\in V(G)$.  We show that $v$ must be adjacent to a boundary face of $G$. Without loss of generality, we may assume that the edge $e$ is not fully contained inside any other loop edge incident to $v$. Let $F'$ be the face on the other side of $e$ to $F$.  If $F'$ is a boundary face then we are done.  Otherwise, we assume that $F'$ is an interior face.  Let $e'$ be any non-loop edge incident to $v$ that is on the boundary of $F'$.  By our bipartite assumption, $e'$ connects $v$ to a vertex of opposite color.  The assumptions that $G$ is simple and $\QG$ is isolated then imply that $e'$ separates $F'$ from a boundary face.  It follows that $v$ is adjacent to a boundary face of $G$.  Letting $L':=\Lplab_{G-e}$, we see from \cref{fig:loop_removal} that $G$ is a connected sum of $L'$ and $\LHopf$. 
  Since $G-e$ has fewer interior faces, the result follows by induction.%

Suppose now that $G$, still assumed to be bipartite, has no interior faces that are bounded by loop edges. Let $F$ be an interior face of $G$. Since $G$ is simple, the perimeter of $F$ does not contain both sides of any edge. Moreover, since $\QG$ is isolated, every edge adjacent to $F$ is adjacent to a boundary face on the other side. Thus, $G$ consists of a number of interior faces connected in a tree-like manner as in \cref{fig:isolated_tree}.  The interior faces at the leaves of this tree must be bounded by loop edges, contradicting our assumption.
\end{proof}

\begin{figure}
  \includegraphics[width=0.4\textwidth]{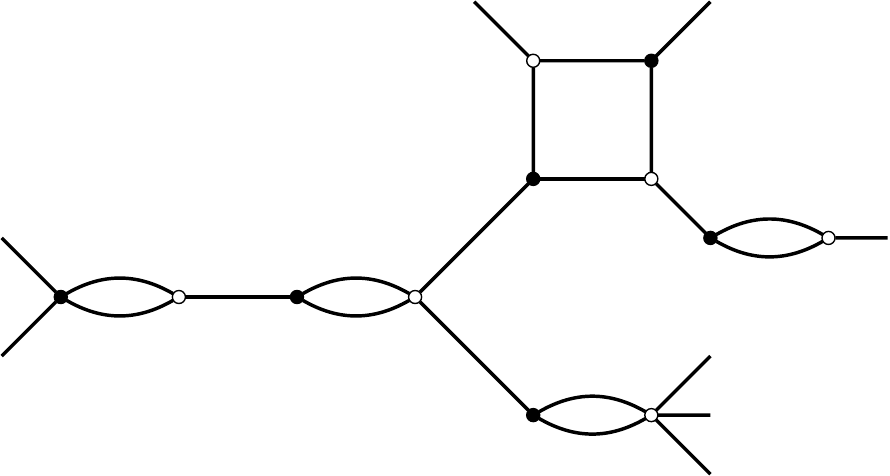}
  \caption{\label{fig:isolated_tree} A ``tree'' of interior faces (without boundary vertices) must have a loop.}
\end{figure}

\subsection{Torsion}
\def\tF{\tilde F}
By definition, for $G$ a (possibly not simple) plabic graph, the quiver $\Qred_G$ is obtained from the directed graph $\QG$ (\cref{sec:intro:quivers-point-count}) by removing all $1$-cycles, and canceling out directed $2$-cycles one by one.  This agrees with the procedure in \cite{FPST}, and the mutation class of $\Qred_G$ is again invariant under the moves in \cref{fig:moves}. Note that $G$ is simple if and only if $\QG=\Qred_G$.

\begin{proposition}\label{prop:torsionfree}
Let $G$ be a plabic graph. Then $\QredG$ is torsion-free.
\end{proposition}
\begin{proof}
We prove the result by induction on the number of vertices of $\QredG$.  In the following, we index rows of the $\BQredG$ by interior faces $F$ of $G$.

We may suppose that $G$ is connected.  If $G$ has fewer than two interior faces, then checking the graphs in \cref{fig:trivred}, we see that the result holds.  So suppose that $G$ has at least two interior faces.  Applying tail removal, we may suppose that $G$ is not reduced.  Then by \cite{Pos}, one of the following holds: (a) $G$ has an interior leaf, or (b) the bipartite reduction (see the proof of \cref{prop:plabic_isolated}) has a loop, or (c) $G$ is move equivalent to a plabic graph with a two parallel edges between vertices of opposite color.

In case (a), removing the leaf vertex does not change $\QredG$.  We may thus assume that all interior leaves have been removed.  In case (b), deleting the loop edge removes a zero row and a zero column from $\BQredG$, so the result holds by induction.   We may thus assume that $G$ is leafless and loopless and we are in situation (c), that is, $G$ has a double edge: a pair $e',e''$ of edges between vertices $u,v$ of opposite colors.  Let $F$ be the interior face between $e'$ and $e''$, and let $F', F''$ be the faces on the other side of $e'$ and $e''$ to $F$; see \cref{fig:leafface}.   Let $G'$ be obtained from $G$ by removing both edges $e',e''$.  Let $B := \BQredG$, $C := B(\Qred_{G'})$, and $n := |V(\QredG)|$.  If $F',F''$ are both boundary faces, then $\Qred_{G'}$ is obtained from $\Qred_G$ by removing an isolated vertex, so the result holds by induction.  

Suppose that one of the faces $F',F''$ is boundary and one, say $F'$, is interior.    Then $\Qred_{G'}$ is obtained from $\Qred_G$ by removing the vertices $F$ and $F'$.  Let $\Z^V \cong \Z^n$ denote the free $\Z$-module with basis vectors $\{e_\E\}_{\E\in V}$ indexed by the vertex set $V = V(\Qred_G) = \{\mbox{interior faces $\E$ of $G$}\}$ and $\Z^{n-2} \cong \Z^{V'} \subset \Z^V$ the submodule spanned by basis elements other than $e_F, e_{F'}$.  The rows of $B$ belong to $\Z^V$ and the rows of $C$ belong to $\Z^{V'} \subset \Z^V$.  By induction, $C$ is torsion-free, so $\Z^{V'}/C^T \Z^{V'}$ has a basis given by some vectors $d_1,d_2,\ldots,d_r \in \Z^{V'}$.   Furthermore, we have
$$
b_F = \pm e_{F'}, \qquad b_{F'} \pm e_F \in \Z^{V'}, \qquad b_\E - c_\E \in  \Z e_{F'},  
$$
where $\E$ is any interior face other than $F,F'$, and $b_\E$ and $c_\E$ denote the rows of $B$ and $C$, respectively.  It is clear that $\Z^V/B^T \Z^V$ has a basis given by $\{d_1,\ldots,d_r\}$.  So $B$ is torsion-free.

Now suppose that $F',F''$ are both interior faces.  If $F' = F''$, then one of $u,v$, say $v$, must be a cut vertex of $G$.  Thus, $G$ contains a non-trivial induced subgraph $H$ containing $v$ and connected to $G-H$ only via $e',e''$.  The subgraph $H$ is contained ``inside" $F'=F''$; we may consider $H$ to be a plabic graph, necessarily not reduced, with a single boundary vertex incident to $v$.  Replacing $G$ by $H$ and repeating the argument, we are reduced the situation where $F' \neq F''$.

With $F' \neq F''$ both interior, $\Qred_{G'}$ is obtained from $\Qred_G$ by removing the vertex $F$ and identifying the vertices $F'$ and $F''$ to give a new vertex $\tF$ of $\Qred_{G'}$.  The rows of $B$ belong to $\Z^V \cong \Z^n$ and the rows of $C$ belong to $\Z^{V'} \cong \Z^{n-2}$, where $V' = V(\Qred_{G'})$.  Define the inclusion $\iota: \Z^{V'} \hookrightarrow \Z^V$ by sending $e_{\tF}$ to $e_{F'}$, and mapping the other basis vectors in the obvious way.  By induction, $C$ is torsion-free, so $\Z^{V'}/C^T\Z^{V'}$ has a basis given by some vectors $d_1,d_2,\ldots,d_r \in \Z^{V'}$.  We have that 
$$
b_F = \pm(e_{F'} - e_{F''}), \qquad b_{F'} +b_{F''} = \iota(c_{\tF}), \qquad b_\E - \iota(c_\E) \in \Z(e_{F'} - e_{F''}),
$$
where $\E$ is an interior face other than $F,F',F''$.  Let $W \subset \Z^{V - \{F\}}$ be the span of the vectors $b_F, b_{F'} +b_{F''}$ and $b_\E, \E \notin \{F,F',F''\}$.  Then $\{\iota(d_1),\ldots,\iota(d_r)\}$ form a basis of $\Z^{V- \{F\}}/W$.
But $b_{F'} \pm e_F \in\Z^{V - \{F\}}$, so $\Z^V/B^T\Z^V$ has basis $\{\iota(d_1),\ldots,\iota(d_r)\}$.  Thus, $B$ is torsion-free.
\end{proof}
Combining \cref{prop:torsionfree} with \cref{lem:ext}, we obtain the following result.
\begin{corollary}\label{cor:ext_simple}
Let $G$ be a simple plabic graph. Then $\QG$ admits a minimal extension.
\end{corollary}

\section{Skein relations}
In this section, we introduce leaf recurrent plabic graphs.  This class of plabic graphs includes the well-studied reduced plabic graphs and plabic fences.  We show that the leaf recurrence for plabic graphs corresponds to the skein relation for the associated links, thus establishing \cref{conj:main} for leaf recurrent plabic graphs.

\subsection{Leaf recurrent plabic graphs}\label{sec:leaf-recurr-plab}
In this section, we assume that all plabic graphs satisfy the assumptions of Section~\ref{sec:plabquiver}.

\begin{definition}
Let $G$ be a plabic graph. An interior face $F$ of $G$ is a \emph{boundary leaf face} if, in the tail reduction of $G$, the face $F$ has the form shown in \cref{fig:leafface}, i.e., $F$ is bounded by two edges $e,e'$ with endpoints $u,v$ of distinct color, with $e$ separating $F$ from a boundary face and $e'$ separating $F$ from an interior face $F'$.
\end{definition}

\begin{figure}
  \includegraphics[width=0.7\textwidth]{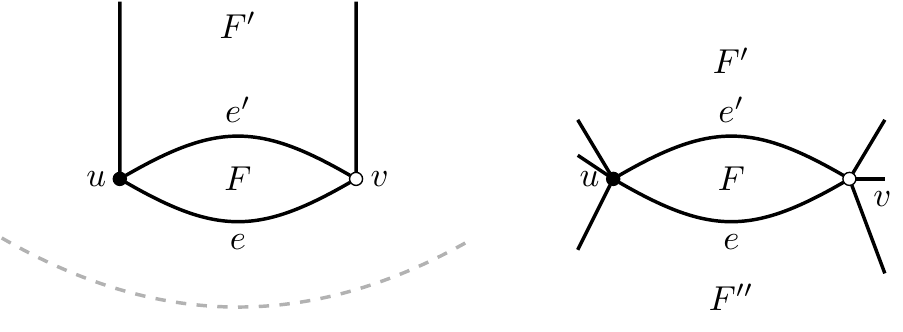}
  \caption{\label{fig:leafface} A boundary leaf face $F$ (left) and an interior double edge (right).}
\end{figure}

\begin{lemma}\label{lemma:bdry_leaf_face}
Let $G$ be a simple plabic graph and $F$ a boundary leaf face. Let $u,v\in V(G)$ be the vertices adjacent to $F$. Then the plabic graphs $G':=G-e$ and $G'':=G-\{u,v\}$ are both simple.
\end{lemma}
\begin{proof}
The directed graph $\QX(G')$ is obtained from $\QG$ by removing the vertex corresponding to $F$.  Similarly, $\QX(G'')$ is obtained from $\QG$ by removing the vertices corresponding to $F$ and $F'$.  Both directed graphs are clearly quivers, and thus $G',G''$ are simple.
\end{proof}

\begin{definition}\label{dfn:leaf_recurrent_plabic_graph}
  The class of \emph{leaf recurrent} plabic graphs is defined as follows.
  \begin{enumerate}[(a)]%
  \item Every leaf recurrent plabic graph is simple.
  \item If $\QG$ is isolated then $G$ is leaf recurrent.
  \item If $G$ is move equivalent to a leaf recurrent plabic graph then $G$ is leaf recurrent.
  \item Suppose $G$ has a boundary leaf face $F$ as in \cref{lemma:bdry_leaf_face}. If the plabic graphs $G'$ and $G''$ are leaf recurrent then $G$ is leaf recurrent.
  \end{enumerate}
\end{definition}

\begin{remark}\label{rmk:leaf_rec=>Louise}
It is immediate that if $G$ is leaf recurrent then $\QG$ is \Louise.
\end{remark}

\begin{theorem}[{\cite[Remark~4.7]{MSLA}}]\label{thm:MSLA}
Any reduced plabic graph $G$ is leaf recurrent.
\end{theorem}

\subsection{The skein relation}\label{sec:sub_skein-relation}

\begin{figure}
  \includegraphics[width=1.0\textwidth]{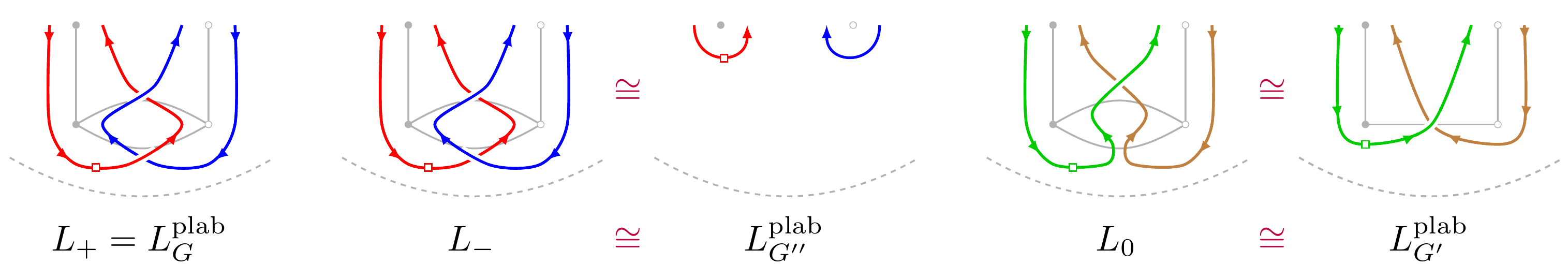}
  \caption{\label{fig:skein} Leaf recurrence for plabic graphs corresponds to the \FLY skein relation; see \cref{lemma:bdry_leaf_skein_FLY}.}
\end{figure}

\begin{lemma}\label{lemma:bdry_leaf_skein_FLY}
Suppose that $G$ is a plabic graph with a boundary leaf face $F$. Let $L_+:=\Lplab_G$, $L_0:=\Lplab_{G'}$, and $L_-:=\Lplab_{G''}$, where $G,G',G''$ are the three graphs in \cref{lemma:bdry_leaf_face}. Then the \FLY polynomials of $(L_+,L_0,L_-)$ are related by~\eqref{eq:HOMFLY_dfn}.
\end{lemma}
\begin{proof}
See \cref{fig:skein}.
\end{proof}

\begin{theorem}
  Suppose $G$ is a leaf recurrent plabic graph. Then Conjectures~\ref{conj:main} and~\ref{conj:top_a_deg} hold for $G$.
\end{theorem}
\begin{proof}
  Suppose $G$ has a boundary leaf face $F$ such that the graphs $G'$, $G''$ from \cref{lemma:bdry_leaf_face} are both leaf recurrent. Denote $Q:=\QG$, $Q':=\QX(G')$, $Q'':=\QX(G'')$, $L_+:=\Lplab_G$, $L_0:=\Lplab_{G'}$, and $L_-:=\Lplab_{G''}$, and note that $G,G',G''$ all have the same number of connected components.

By \cref{lemma:bdry_leaf_skein_FLY}, the \FLY polynomials of $(L_+,L_0,L_-)$ are related by~\eqref{eq:HOMFLY_dfn}. By \cref{rmk:leaf_rec=>Louise}, the quivers $Q,Q',Q''$ are \Louise. By \cref{cor:RQ_skein}, their point counts are related by~\eqref{eq:RQ_skein}. Comparing~\eqref{eq:HOMFLY_dfn} and~\eqref{eq:RQ_skein}, we see that~\eqref{eq:FLY=pcnt} holds for $Q$ if it holds for both $Q'$ and $Q''$.

Now suppose that $\QG$ is isolated with $n$ vertices.  If $G$ is the disjoint union of $G_1$ and $G_2$ then by Proposition~\ref{prop:FLYsum}, we have $\HOMP(\Lplab_{G})= ((a-a^{-1})/z)\HOMP(\Lplab_{G_1})\HOMP(\Lplab_{G_2})$, so $\Ptop_{\Lplab_G}(q) = (q-1)^{-1} \Ptop_{\Lplab_{G_1}}(q)\Ptop_{\Lplab_{G_2}}(q)$.  Since $\RQG(q) = \RQX{\QX(G_1)}{q}\RQX{\QX(G_2)}{q}$, the result follows. We may now assume that $G$ is connected.  By \cref{prop:plabic_isolated}, $\Lplab_G$ must be the connected sum of $n$ Hopf links.  By \cref{prop:FLYsum}, \eqref{eq:Hopf} and~\eqref{eq:A1} we have 
$$\Ptop_{\Lplab_G}(q) = \left(\frac{q^2-q+1}{q-1}\right)^n = \RQG(q).
$$
Thus, the result holds when $\QG$ is isolated.

Since the validity of~\eqref{eq:FLY=pcnt} is unchanged under moves on $G$, it follows that~\eqref{eq:FLY=pcnt} holds for all leaf recurrent plabic graphs $G$.  The equality~\eqref{eq:deg_a} also follows from the same argument.
\end{proof}

\subsection{Plabic fences}\label{sec:plabic-fences}
Following~\cite[Section~12]{FPST}, we consider the following class of plabic graphs. Let $\n\geq 2$ and let $I:=[\n-1]=\{1,2,\dots,\n-1\}$. A \emph{double braid word} is a word $\bw=(i_1,i_2,\dots,i_m)$ in the alphabet
\begin{equation*}%
  \pmn:=\{-1,-2,\dots,-(\n-1)\}\sqcup\{1,2,\dots,\n-1\}.
\end{equation*}
We denote the set of double braid words of length $m$ by $\DRW$. 

To a double braid word $\bw\in\DRW$ we associate a plabic graph $\Gbw$ with $2\n$ boundary vertices. First, draw $\n$ horizontal strands, with both endpoints of each strand on the boundary of the disk. Then for each $j=1,2,\dots,m$, let $h:=|i_j|$ and consider the strands at height $h$ and $h+1$. If $i_j>0$, add a black at $h$, white at $h+1$ bridge between these two strands, and if $i_j<0$, add a white at $h$, black at $h+1$ bridge between these two strands. The bridges are added one by one proceeding from left to right. See \cref{fig:PF2} for an example.

\begin{figure}
  \includegraphics[width=0.7\textwidth]{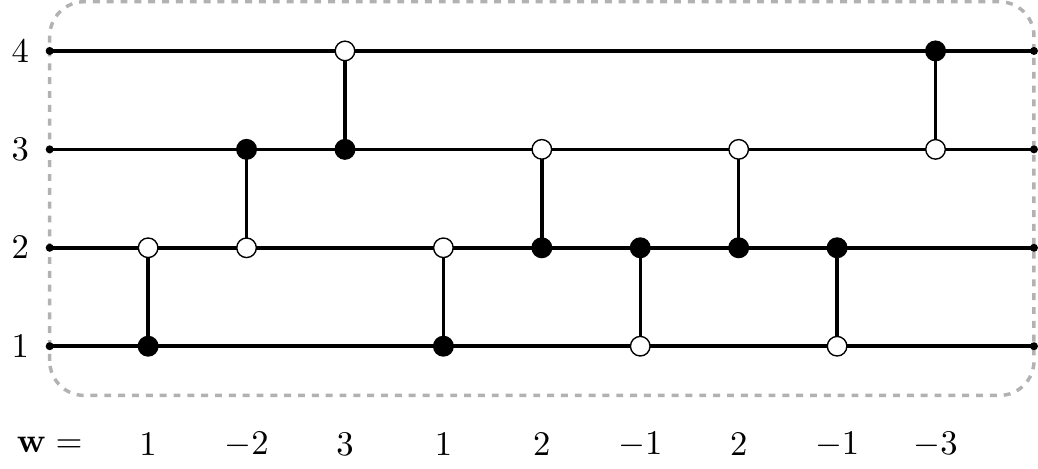}
  \caption{\label{fig:PF2} The plabic fence $\Gbw$ for a double braid word $\w$.}
\end{figure}

\begin{proposition}
For any $\bw\in\DRW$, the plabic graph $\Gbw$ is simple.
\end{proposition}
\begin{proof}
This is clear from construction. %
\end{proof}

Let us say that two braid words $\bw,\bw'$ are \emph{double braid move equivalent} if they are related by a sequence of the following moves:
\begin{enumerate}[(B1)]%
\item\label{bm1} $(\dots,i,j,i,\dots)\leftrightarrow (\dots,j,i,j,\dots)$ \quad if $i,j\in\pmn$ have the same sign and $|i-j|=1$;
\item\label{bm2} $(\dots,i,j,\dots)\leftrightarrow (\dots,j,i,\dots)$ \quad if $i,j\in \pmn$ have the same sign and $|i-j|>1$;
\item\label{bm5} $(\dots,i,j,\dots)\leftrightarrow (\dots,j,i,\dots)$  \quad if $i, j\in\pmn$ have different signs;
\item\label{bm6} $(-i,\dots)\leftrightarrow (i,\dots)$ \quad for $i\in \pmn$;
\item\label{bm7} $(\dots,-i)\leftrightarrow (\dots,i)$ \quad for $i\in \pmn$.
\end{enumerate}

It is clear that if $\bw,\bw'\in\DRW$ are double braid move equivalent then the plabic graphs $\Gbw,\GX(\bw')$ are  move equivalent.

\begin{proposition}
Plabic fences are leaf recurrent. That is, for any double braid word $\bw\in\DRW$, $\Gbw$ is a leaf recurrent plabic graph.
\end{proposition}
\begin{proof}
Proceed by induction on $m$. The base case $m=0$ is clear.

Applying moves~\bmref{bm5}--\bmref{bm6}, we may assume that $\bw$ has no negative indices. We say that a word $\bw=(i_1,i_2,\dots,i_m)\in I^m$ is \emph{reduced} if the associated permutation $\pi(\bw):=s_{i_1}s_{i_2}\cdots s_{i_m}$ has precisely $m$ inversions. If $\bw$ is not reduced then we may apply moves~\bmref{bm1}--\bmref{bm2} to transform $\bw$ into a word of the form $(\bw_1,i,i,\bw_2)$ for some $i\in I$ and some words $\bw_1\in I^{m_1}$, $\bw_2\in I^{m_2}$. Let $-\rev(\bw_1)\in (-I)^{m_1}$ denote the word obtained from $\bw_1$ by negating all indices and reversing their order. One can transform $\bw_1$ into $-\rev(\bw_1)$ by applying moves~\bmref{bm5}--\bmref{bm6}. Thus, applying moves~\bmref{bm5}--\bmref{bm6}, we transform
\begin{equation*}%
  (\bw_1,i,i,\bw_2)\to (-\rev(\bw_1),i,i,\bw_2)\to (i,i,\bw_2,-\rev(\bw_1)) \to (i,i,\bw_2,\bw_1).
\end{equation*}  The double bridge in $\GX(i,i,\bw_2,\bw_1)$ corresponding to $(i,i,\dots)$ gives rise to a boundary leaf face. Applying the induction hypothesis to the words $\bw':=(i,\bw_2,\bw_1)$ and $\bw'':=(\bw_2,\bw_1)$, we see that the plabic graphs $\GX(\bw')$ and $\GX(\bw'')$ are leaf recurrent. Thus, by \cref{dfn:leaf_recurrent_plabic_graph}, $\GX(\bw)$ is also a leaf recurrent plabic graph.

Finally, suppose that the word $\bw=(i_1,i_2,\dots,i_m)\in I^m$ is reduced. Then $\Gbw$ is a reduced plabic graph and therefore we are done by \cref{thm:MSLA}. Alternatively, applying moves~\bmref{bm5}--\bmref{bm6}, we may transform $\bw$ as follows:
\begin{equation*}%
  (i_1,i_2,\dots,i_m) \to (-i_1,i_2,\dots,i_m) \to (i_2,\dots,i_m,-i_1)\to (i_2,\dots,i_m,i_1).
\end{equation*}
We refer to this transformation as a \emph{conjugation move}. Applying conjugation and braid moves repeatedly, we either arrive at a non-reduced word and proceed by induction, or we find that $\pi(\bw)\in\Sn$ is a minimal length representative in its conjugacy class; see~\cite[Theorem~1.1]{GePf}. Recall that conjugacy classes in the symmetric group $\Sn$ correspond to cycle types, and $\pi(\bw)$ has minimal length  if and only if it is a product of cycles of the form $(a,a+1,\dots,b)$ on disjoint intervals $[a,b]\subset[\n]$. The resulting plabic fence $\Gbw$ has no interior faces and therefore is leaf recurrent.
\end{proof}

\section{Invariants of quivers and links}
In this section, we investigate the relation between some other invariants of the quiver $Q_G$ and invariants of the link $\Lplab_G$.

\subsection{Corank} Recall that we set $\cork(Q):=\cork(\BQ)=n-\rk(\BQ)$ for a quiver $Q$ with $n$ vertices.

\begin{proposition}\label{prop:conn=cork}
Let $G$ be a plabic graph. Then
\begin{equation*}%
  \conn(G)+\cork(\QG)=\conn(\Lplab_G).
\end{equation*}
\end{proposition}
\begin{proof}
We prove the result using the same induction as in the proof of \cref{prop:torsionfree}.  The result can be checked directly for the graphs in \cref{fig:trivred}.  We now suppose that $G$ has one of the following: (a) an interior leaf, or (b) a loop edge, or (c) two parallel edges between vertices of opposite color.

In case (a), removing the leaf vertex preserves $\conn(G)$, $\conn(\Lplab_G)$, and $\cork(\QG)$.  In case (b), removing the loop edge reduces $\conn(\Lplab_G)$ by $1$ and $\cork(\QG)$ by $1$.   In case (c), we use the same notation as in the proof of \cref{prop:torsionfree}; see \cref{fig:leafface}.   Let $G'$ be obtained from $G$ by removing both edges $e',e''$.  This operation preserves $\conn(\Lplab_G)$. If $F',F''$ are boundary faces, then $\conn(G') = \conn(G)+1$ and $\cork(\QG) = \cork(\QG)-1$. If one of the faces $F',F''$ is boundary and one is interior, then $\conn(G') = \conn(G)$ and $\cork(Q_{G'})=\cork(\QG)$, as shown in the proof of \cref{prop:torsionfree}.   If $F',F''$ are both interior, we may assume they are distinct.  Thus, $\conn(G') =\conn(G)$ and as shown in the proof of \cref{prop:torsionfree}, we have $\cork(Q_{G'}) = \cork(\QG)$.  Thus, $\conn(\QG)+\cork(\QG)=\conn(\Lplab_G)$ follows from the same equation for $G'$.
\end{proof}

\begin{remark}
For algebraic links, \cref{prop:conn=cork} should be compared to \cite[Proposition 4.7]{FPST}.
\end{remark}

\subsection{Catalan numbers}
Recall that for a quiver $Q$, we defined the $Q$-Catalan number in \cref{def:QCat}.

\def\tP_#1(#2){{\tilde P}(#1;#2)}
\def\chiL{\chi_L}
\begin{definition}
Let $L$ be a link and %
let $\tP_L(q):=(q-1)^{\conn(L)-1} \Ptop_L(q)$.  If $\tP_L(q)$ is a polynomial in $q$, define
  \begin{equation*}%
    \chiL:=\tP_L(1).
  \end{equation*}
We call $\chiL$ the \emph{Catalan number} of the link $L$, or simply the \emph{$L$-Catalan number}.
\end{definition}
\noindent See \cref{tab:small_knots} for some examples.

\begin{proposition}
Suppose that \cref{conj:main} holds for a simple plabic graph $G$.  Then $\chi_{\QG} = \chi_{\Lplab_G}$.
\end{proposition}
\begin{proof}
Let $\Qice$ be a minimal extension of $\QG$ (cf. \cref{cor:ext_simple}), and set $m=\cork(\QG)$.  Then 
\begin{align*}
R(\Qice;q) = (q-1)^m R(Q_G;q)=(q-1)^{m+\conn(G)-1}\Ptop_{\Lplab_G}(q) =(q-1)^{\conn(\Lplab_G)-1}\Ptop_{\Lplab_G}(q) 
\end{align*}
by \cref{conj:main} and \cref{prop:conn=cork}. %
 Thus, $\chiQ = \chi_{\Lplab_G}$.
\end{proof}

\begin{problem}\ 
\begin{itemize}
\item For which links is $\chiL$ defined?
\item For which links is $\chiL$ nonnegative, or positive?
\end{itemize}
\end{problem}
\begin{example}
  Even when $\chiL$ is defined, it may be negative. The smallest knot $K$ for which this happens is listed as \href{https://knotinfo.math.indiana.edu/results.php?searchmode=singleknot&desktopmode=0&mobilemode=0&singleknotprev=&submittype=singleknot&singleknot=12n_199}{$12n_{199}$} in~\cite{knotinfo}. It has \FLY polynomial 
\begin{equation*}
 \HOMP(K;q)= -\frac{1}{a^{4}} + \frac{z^{6} + 6 \, z^{4} + 11 \, z^{2} + 6}{a^{6}} - \frac{z^{4} + 4 \, z^{2} + 4}{a^{8}},
\end{equation*}
and therefore $\tP_K(q)=\Ptop_K(q)=-q^2$, with $\chi_K=-1$.
\end{example}

\subsection{Cluster cohomology versus link homology}\label{sec:clust-cohom-vs-link-hom}
Let $Q$ be a \Louise and torsion-free quiver.  Let $\Qice$ be a minimal extension of $Q$ and set $m = \cork(B(Q))$.  Let $T=T(\Qice)$ be the cluster automorphism torus of $\Qice$ defined in \cite[Section 5]{LS}.  The torus $T$ acts on $\AQice$ sending every cluster variable to a nonzero multiple of itself, and the character lattice of $T$ is naturally identified with $\Z^{n+m}/\BQice\Z^n$.

We may associate to $Q$ the cohomology and the compactly-supported $T$-equivariant cohomology of the cluster variety $\XQice$.
Both cohomologies are equipped with mixed Hodge structures
$$
H^*(\XQice) = \bigoplus_{k,p,q} H^{k,(p,q)}(\XQice; \C) \qquad \text{and} \qquad H^*_{T,c}(\XQice) = \bigoplus_{k,p,q} H^{k,(p,q)}_\Tc(\XQice; \C).
$$

\begin{proposition}
The cohomology $H^*(Q):= H^*(\XQice) $ is of \emph{mixed Tate type}, that is,
$$
H^*(\XQice) = \bigoplus_{k,p} H^{k,(p,p)}(\XQice; \C).
$$
Moreover, it does not depend on the choice of $\Qice$.
\end{proposition}
\begin{proof}
The first statement is part of \cite[Theorem 8.3]{LS}.  Let $\Qice$ and $\Qice'$ be two really full rank ice quivers with mutable part $Q$ and the same number $m$ of frozen variables.  In \cite[Proposition 5.11]{LS}, it is shown that we have an isomorphism of mixed Hodge structures $H^*(\XQice) = H^*(\X(\Qice'))$.  \end{proof}

The results of \cite{LS} are only established in the case of ordinary cohomology.  In the following, we assume that the results there extend to equivariant cohomology.   That is, we have 
$$
H^*_\Tc(Q):=H^*_{T,c}(\XQice) = \bigoplus_{k,p} H^{k,(p,p)}_{T,c}(\XQice; \C)
$$
and it does not depend on the choice of $\Qice$.  We expect to return to equivariant cohomology of cluster varieties in future work.

Let $L$ be an oriented link.  Let $\HHH(L)$ denote the Khovanov--Rozansky triply-graded homology~\cite{KR1,KR2,KhoSoe} of $L$, and $\HHH^0(L)$ the Hochschild degree $0$ part of $\HHH(L)$.  Similarly, let $\HHHC$ denoted the variant of Khovanov--Rozansky homology defined in~\cite[Section~3]{GL_qtcat}, and let $\HHHC^0$ denote the Hochschild degree 0 part.  We assume these homology groups have been normalized so they become link invariants; cf. \cite[Equations~(3.11) and~(3.15)]{GL_qtcat}.  %

\begin{conjecture}\label{conj:cohom}
Let $G$ be a connected simple plabic graph.  Then we have bigraded isomorphisms
$$
H^*(\QG) \cong \HHH^0_\C(\Lplab_G) \qquad \text{and} \qquad
H^*_\Tc(\QG) \cong \HHH^0(\Lplab_G).
$$
\end{conjecture}

We refer the reader to \cite{GL_qtcat} for the precise correspondence between the bigradings.  Here are some comments on Conjecture~\ref{conj:cohom}:

\noindent 
(1) The torus $T$ has dimension $m = \cork(B(Q))$, which by~\cref{prop:conn=cork} equals to $\conn(\Lplab_G)-1$, where $\conn(\Lplab_G)$ is the number of connected components of $\Lplab_G$.

\noindent
(2) The Khovanov--Rozansky homology $\HHH(\Lplab_G)$ has an action of a polynomial ring with $\conn(\Lplab_G)-1$ generators.  This action should correspond to the action of $H^*_T(\pt)$ on $H^*_\Tc(\QG)$.

\noindent
(3) The Euler characteristic of $\HHH(L)$ recovers the \FLY polynomial $P(L)$; see \cite[(3.13)]{GL_qtcat}.  Thus, taking the Euler characteristic of \cref{conj:cohom} produces \cref{conj:main}.

\noindent
(4) When $m = \cork(B(G)) = 0$, the torus $T$ is the trivial group.  In this case $H^*_{T,c}(\XQice) = H^*_c(\XQice)$, so $H^*(\QG)$ and $H^*_\Tc(\QG)$ are related by Poincare duality.  Thus, \cref{conj:cohom} predicts that the ordinary cohomology $H^*(\QG)$ is identified with the degree $0$ part $\HHH^0(\Lplab_G)$ of the Khovanov--Rozansky homology.

\noindent 
(5) When $G$ is a reduced plabic graph with $N$ boundary vertices, \cref{conj:cohom} is proven in \cite{GL_qtcat}.  Let $\Qice_G$ be the ice quiver associated to $G$, with frozen vertices corresponding all but one of the boundary faces; see \cite{GL_cluster}.  Then $\Qice_G$ is a really full rank quiver, and $\X(\Qice_G)$ is isomorphic to an open positroid variety $\Pio_f$ sitting inside a Grassmannian $\Gr(k,n)$.   

In \cite{GL_qtcat}, we considered the action of the $(N-1)$-dimensional torus $T' \subset \PGL_N$ acting on $\Pio_f$.  When $G$ is connected, the positroid ${\mathcal M}$ associated to $\Pio_f$ is connected as a matroid and the action of $T'$ on $\Pio_f$ is faithful, for example by \cite[Lemma 3.1]{ALS}.  It is straightforward to verify that $T'$ acts on $\Pio_f$ by cluster automorphisms.  In this case, $\Qice_G$ has $(N-1)$ frozen vertices and $T(\Qice_G)$ is also $(N-1)$-dimensional.  It follows that we have an inclusion $T' \hookrightarrow T(\Qice_G)$ where $T(\Qice_G)/T'$ is a finite group.  This finite group can be shown to be trivial, for example with the same argument as the proof of \cite[Lemma 4.6]{ALS}.  If $G$ is not connected, the torus $T'$ does not act faithfully on $\Pio_f$.  Instead, we have a decomposition $T' \cong T'' \times T(\Qice_G)$ where $T''$ acts trivially.

Finally, we remark that the number of frozen vertices in $\Qice_G$ may not equal to $\cork(\QG)$; nevertheless, the computation of $H^*_{T,c}(\X(\Qice_G))$ and $H^*_\Tc(Q_G)$ are essentially equivalent.

\noindent 
(6) When $G$ is a plabic fence, the cluster variety is, up to a torus factor, a \emph{double Bott--Samelson variety} in the sense of Shen--Weng~\cite{SW}.  By combining the cluster structure results of \cite{SW} with results on the cohomology of braid varieties \cite{Trinh} (see~\cite[Remark 4.12]{CGGS}), one can deduce a version of \cref{conj:cohom} for plabic fences.  The point count of these cluster varieties is also discussed in \cite[Corollary 6.7]{SW}.

\section{Conjectures and examples}\label{sec:conjectures-examples}
\subsection{Example: a simple plabic graph that is not leaf recurrent}\label{sec:notleaf}
Consider the plabic graph $G$ in \figref{fig:simple}(a). The associated quiver $\QG$ is mutation equivalent to an acyclic triangle quiver, with a single edge between each vertex of the triangle.  In particular, $\QG$ is \Louise.  However, it is easy to check that $G$ is not a leaf recurrent plabic graph. Nevertheless, \cref{conj:main} still holds for $G$.  By \cref{prop:acycliccount}, we have 
\begin{equation*}%
  \RQG(q) = (q-1)^3+3q(q-1) = q^3-1.
\end{equation*}
On the other hand, the plabic graph link $\Lplab_G$ is listed as \linkurl{L9n19} in~\cite{KAT}. Its \FLY polynomial is given by
\begin{equation*}%
  \HOMP(\Lplab_G;a,z)=\frac{z^{3} + 3 \comma z}{a^{3}} + \frac{z + z^{-1}}{a^{7}} - \frac{z^{-1}}{a^{9}}.
\end{equation*}
This agrees with Conjectures~\ref{conj:main} and~\ref{conj:top_a_deg}.

\subsection{Example: a simple plabic graph quiver that is not \Louise}\label{sec:notLA}
Consider the plabic graph $G$ and the associated quiver $\QG$ in \figref{fig:4puncture}(b). This quiver comes from a cluster algebra associated with a $4$-punctured sphere; see~\cite[Section~11.5]{Muller}.  
\begin{figure}
  \includegraphics[width=1.0\textwidth]{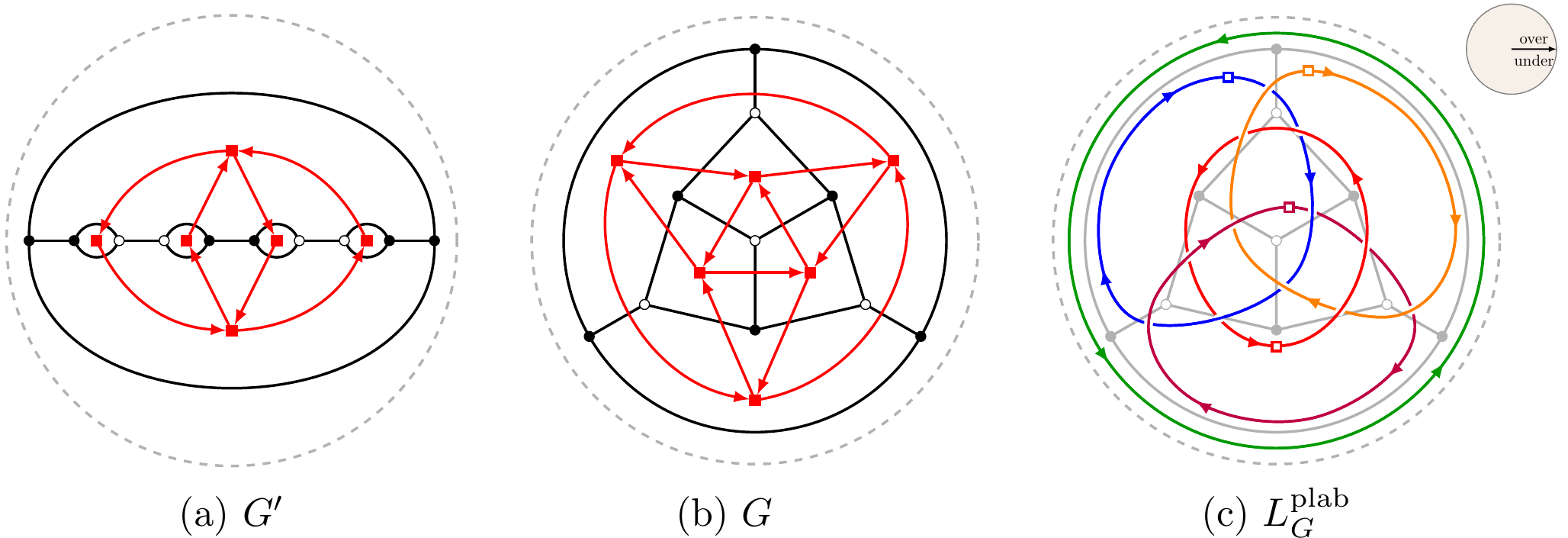}
  \caption{\label{fig:4puncture} Two connected simple plabic graphs $G,G'$ which are move equivalent. The associated quivers $\QG,Q_{G'}$ are not \Louise.} %
\end{figure}
  
Muller defines a cluster variety to be \emph{locally acyclic} if it can be covered by cluster localizations that are acyclic.

\begin{proposition}[{\cite[Section 11.5]{Muller}}]
The quiver $\QG$ is not \Louise. The associated cluster algebra $\AQG$ is not locally acyclic.
\end{proposition}
\begin{proof}
If $\QG$ is \Louise then $\AQG$ would be locally acyclic, so it is enough to establish the second statement.  The proof in~\cite{Muller} that $\AQG$ is not locally acyclic has the following gap: in \cite[Theorem 8.3]{Muller}, it is assumed that it is enough to check the cluster exchange relations to produce a point of a cluster variety.  However, the ideal of relations between cluster variables is in general larger than the ideal generated by the exchange relations.  G.~Muller has supplied us with the following alternative argument. 

The cluster algebra $\AQG$ includes into the tagged skein algebra $S_{q=1}$ at $q=1$ of the 4-punctured sphere.  This is the algebra generated by immersed curves in the surface with taggings at each puncture, modulo certain relations~\cite{Muller_skein}.  The two algebras $\AQG$ and $S_{q=1}$ may or may not be equal.  One can define a point $p_0$ in $\Spec(S_{q=1})$ by giving every arc the value 0, and giving every loop the value $2(-1)^i$, where $i$ is the number of punctures in the interior of the loop.  The point $p_0$ defines a point in $\Spec(\AQG)$, and since every arc takes value $0$ on $p_0$, so does every cluster variable.  Since every cluster variable vanishes on $p_0$, it cannot be contained in any proper cluster localization (see \cite[Definition 3.3]{Muller}). Therefore, $p_0$ cannot be contained in any locally acyclic cover of $\Spec(\AQG)$.  
\end{proof}

We do not know the point count function $\RQ(q)$ of this quiver, or even whether it is a polynomial. On the other hand, the \FLY polynomial of the link $\Lplab_G$ (see \figref{fig:4puncture}(c)) is given by
\begin{align*}%
 \HOMP(\Lplab_G;a,z)=& \frac{z^{6} + 6 \comma z^{4} + 11 \comma z^{2} + 12 + 6z^{-2} + z^{-4}}{a^{6}} - \frac{2 \comma z^{4} + 14 \comma z^{2} + 27 + 19z^{-2} + 4z^{-4}}{a^{8}}  \\
&+ \frac{3 \comma {\left(z^{2} + 6 + 7z^{-2} + 2z^{-4}\right)}}{a^{10}}- \frac{3+ 9z^{-2} + 4z^{-4}}{a^{12}} + \frac{z^{-2} + z^{-4}}{a^{14}}.
\end{align*}
Thus, according to \cref{conj:main}, the prediction for the point count $\RQ(q)$ is
\begin{equation*}%
  \frac{q^{10} - 4 \comma q^{9} + 8 \comma q^{8} - 6 \comma q^{7} - 3 \comma q^{6} + 9 \comma q^{5} - 3 \comma q^{4} - 6 \comma q^{3} + 8 \comma q^{2} - 4 \comma q + 1}{{\left(q - 1\right)}^{4}}.
\end{equation*}

\subsection{Example: conjectures fail when the plabic graph is not simple}\label{sec:counterex}
The plabic graph $G$ in \figref{fig:simple}(d) is not simple as the directed graph $\QG$ has a loop arrow.  If we remove this loop, we obtain the quiver $Q:=\QredG=A_1 \sqcup A_1$ consisting of two isolated vertices.  The quiver point count is 
\begin{equation}\label{eq:RQ_simple1}
  \RQ(q) = \left(\frac{q^2-q+1}{q-1}\right)^2.
\end{equation}
On the other hand, the link $\Lplab_G$ is listed as \linkurl{L10n94} in~\cite{KAT}. Its \FLY polynomial is 
\begin{equation*}%
 \HOMP(\Lplab_G;a,z)=\frac{z^{2} + 2}{a^{2}} - \frac{z^{2} + 2z^{-2} + 3}{a^{6}} + \frac{z^{-2}}{a^{4}} + \frac{1+z^{-2} }{a^{8}}.
\end{equation*}
Thus, \cref{conj:main} predicts that the point count $\RQ(q)$ should be equal to $q^2+1$, which does not match~\eqref{eq:RQ_simple1}.

The plabic graph $G$ in \figref{fig:simple}(b) is not simple as the directed graph $\QG$ has a directed 2-cycle.  If we cancel out the arrows in this $2$-cycle, we obtain a quiver $Q:=\QredG$ that is a directed $4$-cycle, which is mutation equivalent to the quiver of type $D_4$.  Using \cref{prop:acycliccount}, we obtain that the quiver point count is $R(D_4;q) = (q^6 - 2q^5+2q^4-q^3+2q^2-2q+1)/(q-1)^2$; see \cref{tab:Dynkin}.
On the other hand, the \FLY polynomial of $\Lplab_G$ is
\begin{equation*}%
\HOMP(\Lplab_G;a,z)=\frac{z^{4} + 4  z^{2} + 2}{a^{4}} - \frac{z^{2} + 3 + 2z^{-2}}{a^{10}} + \frac{z^{-2}}{a^{8}} + \frac{1+z^{-2}}{a^{12}}.
\end{equation*}
Thus, \cref{conj:main} predicts that the point count $\RQ(q)$ should be given by $q^{4}+1$, which does not agree with $R(D_4;q)$.

\subsection{Connected sum and disjoint union}
\begin{figure}
  \includegraphics[width=0.5\textwidth]{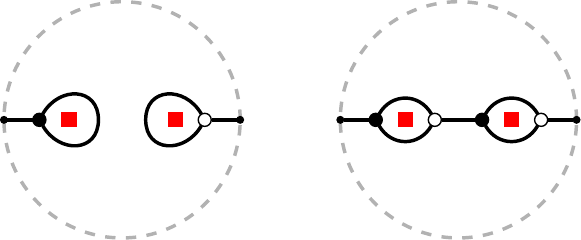}
  \caption{\label{fig:Hopfsum} Two plabic graphs $G$ (left) and $G'$ (right) with the same quiver $\QG = \QX(G')$.}
\end{figure}
Consider the graphs $G$ and $G'$ in \cref{fig:Hopfsum}. The quiver $Q=\QG=\QX(G')$ satisfies 
\begin{equation*}%
  \RQ(q)=\left(\frac{q^2-q+1}{q-1}\right)^2.
\end{equation*}
However, the links $\Lplab_G$ and $\Lplab_{G'}$ are not isotopic: $\Lplab_G$ is a disjoint union of two Hopf links while $\Lplab_{G'}$ is a connected sum of two Hopf links. The \FLY polynomials are given by 
\begin{align*}
P(\Lplab_G;a,z)&=\left(\frac{a-a^{-1}}{z}\right)\left(\frac{z + z^{-1}}{a} - \frac{z^{-1}}{a^3}\right)^2;\\
P(\Lplab_{G'};a,z)&=\left(\frac{z + z^{-1}}{a} - \frac{z^{-1}}{a^3}\right)^2.
\end{align*}
We see that the \FLY polynomials of these two links are different, however, their top $a$-degree terms satisfy 
\begin{equation*}%
  (q-1)\Ptop_{\Lplab_G}(q)=\Ptop_{\Lplab_{G'}}(q)=\RQ(q),
\end{equation*}
 in agreement with \cref{conj:main}. We conjecture that when one restricts to the class of \emph{connected} plabic graphs, the full \FLY polynomial becomes an invariant of the quiver.
\begin{conjecture}\label{conj:fullFLY}
Let $G,G'$ be two connected simple plabic graphs. Assume that the quivers $\QG$ and $\QX(G')$ are mutation equivalent. Then 
\begin{equation*}%
  \HOMP(\Lplab_G;a,z)=\HOMP(\Lplab_{G'};a,z).
\end{equation*}
\end{conjecture}

\begin{remark}
A stronger cohomological version of~\cref{conj:fullFLY} is that $\HHHC(\Lplab_G) \cong \HHHC(\Lplab_{G'})$ and $\HHH(\Lplab_G) \cong \HHH(\Lplab_{G'})$; see \cref{sec:clust-cohom-vs-link-hom}.
\end{remark}
\begin{remark}
Let $G$ be a reduced plabic graph with strand permutation $\pi=\pig$. The top $a$-degree coefficient $\Ptop_{\Lplab_G}(q)$ of $\HOMP(\Lplab_G;a,z)$ can be extracted from the open positroid variety $\Pio_\pi$ by computing the point count. However, observe that the first two rows in \cref{fig:intro_ex} correspond to isomorphic open positroid varieties. Thus, neither the link $\Lplab_G$ nor the full \FLY polynomial $\HOMP(\Lplab_G;a,z)$ are determined by $\Pio_\pi$. \Cref{conj:fullFLY} implies that when $G$ is connected, $\HOMP(\Lplab_G;a,z)$ is fully determined by $\Pio_\pi$. It would be interesting to find a geometric interpretation of the lower $a$-degree coefficients in this case.
\end{remark}

\begin{remark}
According to M.~Shapiro's conjecture~\cite[Conjecture 6.17]{FPST}, if the quivers $\QG$ and $\QX(G')$ are mutation equivalent then the graphs $G$ and $G'$ are \emph{move-and-switch equivalent}. When two plabic graphs $G,G'$ differ by a switch (see~\cite[Figure~25]{FPST}), the corresponding links $\Lplab_G,\Lplab_{G'}$ need not be isotopic. Nevertheless, by~\cite[Proposition~6.16]{FPST}, \cref{conj:fullFLY} implies $\HOMP(\Lplab_G;a,z)=\HOMP(\Lplab_{G'};a,z)$. The switch operation can be defined more generally for oriented links, and we expect that this operation still preserves the \FLY polynomial.
\end{remark}
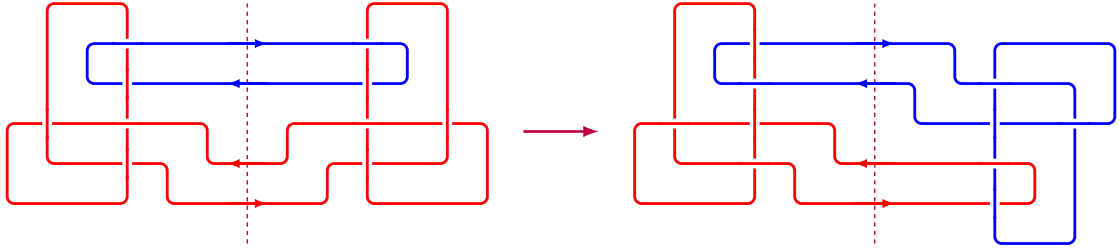
\begin{figure}
\begin{tabular}{ccc}
  \includegraphics[width=0.4\textwidth]{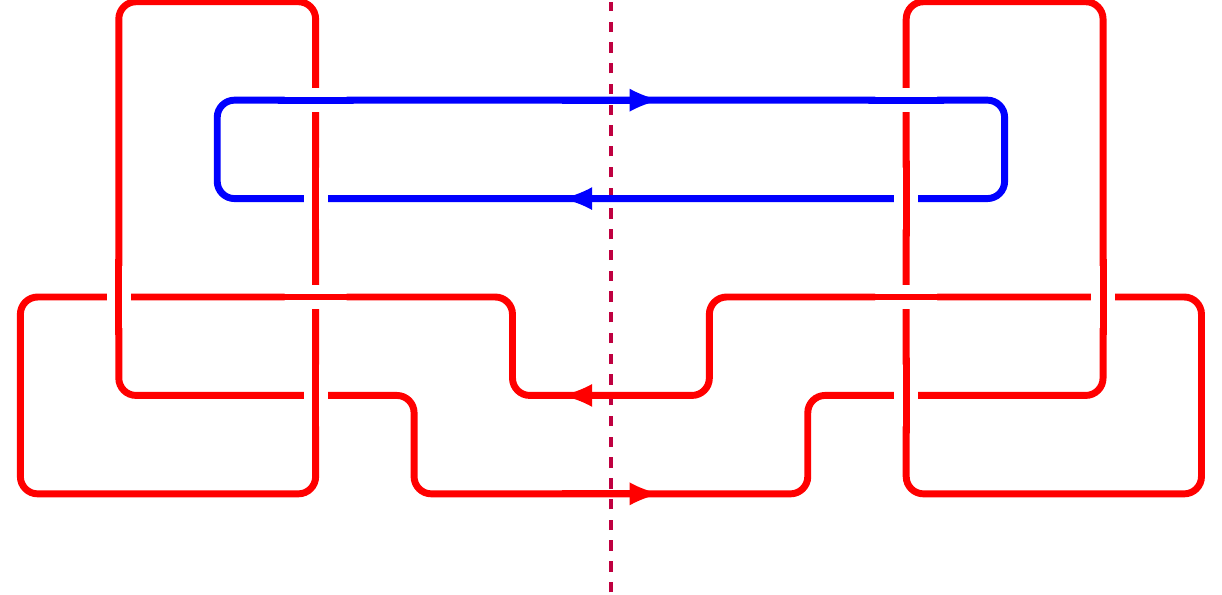}
& 
\begin{tikzpicture}[baseline=(Z.base),scale=1.0]
\coordinate(Z) at (0,-1.5);
\draw[line width=1pt,->,>={latex},purple] (0,0)--(1,0);
\end{tikzpicture}
&
  \includegraphics[width=0.4\textwidth]{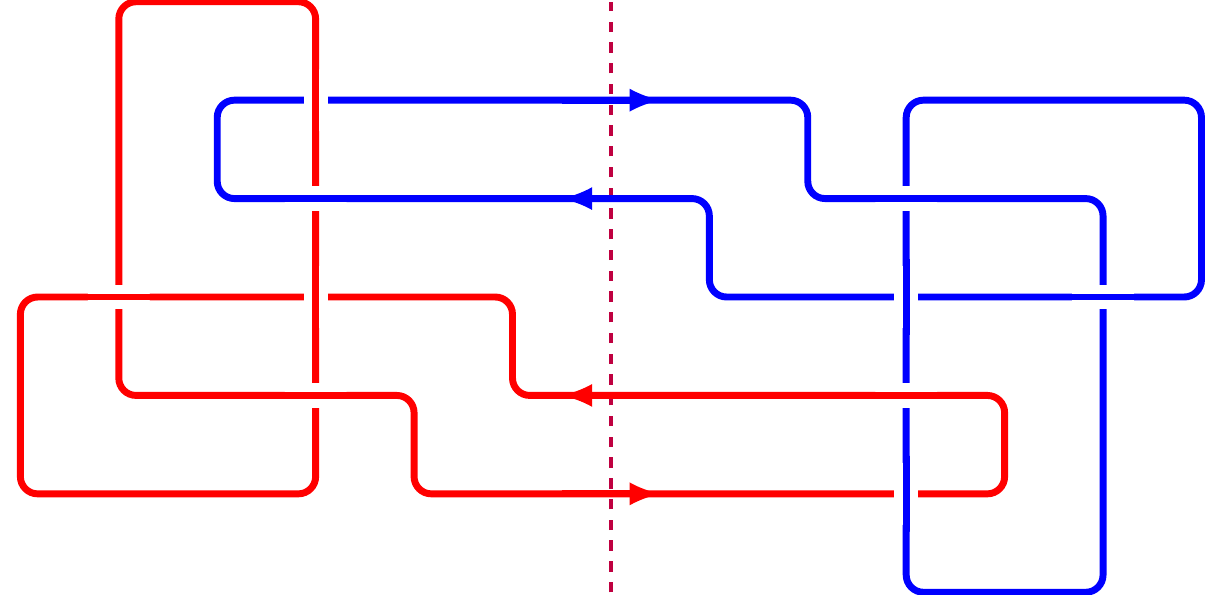}
\end{tabular}
  \caption{\label{fig:switch} The switch operation for links; see \cref{conj:switch}.}
\end{figure}
\begin{conjecture}[Switch preserves \FLY]\label{conj:switch}
Suppose that the planar diagram of a link $L$ intersects the vertical line $x=0$ at four points $(0,-b)$, $(0,-a)$, $(0,a)$, $(0,b)$ for $0<a<b$. Suppose that $L$ is directed to the right at $(0,\pm b)$ and to the left at $(0,\pm a)$. The \emph{switch} of $L$ is a new link $L'$ obtained by taking the $x>0$ part of $L$ and reflecting it along the $x$ axis, flipping all crossings. Then 
\begin{equation*}%
    \HOMP(L;a,z)=\HOMP(L';a,z).
\end{equation*}
\end{conjecture}
\noindent See \cref{fig:switch} for an example of a switch.

\subsection{Small knots}
It is natural to ask which knots and links can potentially be associated to a quiver. Assuming \cref{conj:main}, a necessary condition for a link $L$ to come from a locally acyclic quiver would be that up to a power of $(q-1)$, $\Ptop_L(q)$ is a polynomial in $q$. 

 First, let $K$ be the trefoil knot (\cref{fig:intro_ex}), with $\HOMP(K;q)=\frac{z^2+2}{a^2}-\frac1{a^4}$. Let $\mirr K$ be the mirror image of $K$. Taking the mirror image results in replacing $a$ with $-a^{-1}$; cf.~\eqref{eq:HOMFLY_dfn}. Thus, $\HOMP(\mirr K;q)=-a^{4} + {\left(z^{2} + 2\right)} a^{2}$ and $\Ptop_{\mirr K}(q)=-q^{-2}$, which is not a polynomial. Next, let $E$ be the figure-eight knot. It coincides with its mirror image: $E\cong \mirr E$. We have $\HOMP(E;q)=a^{2} - (z^{2}+1) + \frac{1}{a^{2}}$ and $\Ptop_E(q)=-q^{-1}$, which is again not a polynomial. We therefore do not expect $\mirr K$ and $E$ to be associated to \Louise quivers. (By~\cite[Theorem~1.11]{GL_qtcat}, it follows that neither $\mirr K$ nor $E$ is a Richardson knot.)

In \cref{tab:small_knots}, we give a list of all knots $K$ with up to $9$ crossings such that $\Ptop_K(q)$ is a polynomial in $q$. The knot numbering is taken from~\cite{Rolfsen,KAT_knot}, and the knots are considered up to taking the mirror image. We see that for $K=\href{http://katlas.org/wiki/8_21}{8_{21}}$, the leading coefficient of $\Ptop_K(q)$ is $2$, which violates \cref{conj:top_a_deg}; thus, we do not expect $8_{21}$ to be associated to a quiver.

\begin{table}
\begin{tabular}[t]{cc}
  \begin{tabular}[t]{|c|c|c|}\hline
$K$ & $\chi_K$ & $\Ptop_K(q)$\\\hline\hline
\href{http://katlas.org/wiki/3_1}{$3_{1}$} &  $2$ & $q^{2} + 1$ \\\hline
\href{http://katlas.org/wiki/5_1}{$5_{1}$} &  $3$ & $q^{4} + q^{2} + 1$ \\\hline
\href{http://katlas.org/wiki/5_2}{$5_{2}$} &  $1$ & $q^{2} - q + 1$ \\\hline
\href{http://katlas.org/wiki/7_1}{$7_{1}$} &  $4$ & $q^{6} + q^{4} + q^{2} + 1$ \\\hline
\href{http://katlas.org/wiki/7_2}{$7_{2}$} &  $1$ & $q^{2} - q + 1$ \\\hline
\href{http://katlas.org/wiki/7_3}{$7_{3}$} &  $1$ & $q^{4} - q^{3} + q^{2} - q + 1$ \\\hline
\href{http://katlas.org/wiki/7_4}{$7_{4}$} &  $0$ & $q^{2} - 2 \, q + 1$ \\\hline
\href{http://katlas.org/wiki/7_5}{$7_{5}$} &  $2$ & $q^{4} - q^{3} + 2 \, q^{2} - q + 1$ \\\hline
\href{http://katlas.org/wiki/8_15}{$8_{15}$} &  $1$ & $q^{4} - 2 \, q^{3} + 3 \, q^{2} - 2 \, q + 1$ \\\hline
\href{http://katlas.org/wiki/8_19}{$8_{19}$} &  $5$ & $q^{6} + q^{4} + q^{3} + q^{2} + 1$ \\\hline
\href{http://katlas.org/wiki/8_21}{$8_{21}$} &  $3$ & $2 \, q^{2} - q + 2$ \\\hline
\href{http://katlas.org/wiki/9_1}{$9_{1}$} &  $5$ & $q^{8} + q^{6} + q^{4} + q^{2} + 1$ \\\hline
\href{http://katlas.org/wiki/9_2}{$9_{2}$} &  $1$ & $q^{2} - q + 1$ \\\hline
\href{http://katlas.org/wiki/9_3}{$9_{3}$} &  $1$ & $q^{6} - q^{5} + q^{4} - q^{3} + q^{2} - q + 1$ \\\hline
\href{http://katlas.org/wiki/9_4}{$9_{4}$} &  $1$ & $q^{4} - q^{3} + q^{2} - q + 1$ \\\hline
  \end{tabular}

&

\begin{tabular}[t]{|c|c|c|}\hline
$K$ & $\chi_K$ & $\Ptop_K(q)$\\\hline\hline
\href{http://katlas.org/wiki/9_5}{$9_{5}$} &  $0$ & $q^{2} - 2 \, q + 1$ \\\hline
\href{http://katlas.org/wiki/9_6}{$9_{6}$} &  $3$ & $q^{6} - q^{5} + 2 \, q^{4} - q^{3} + 2 \, q^{2} - q + 1$ \\\hline
\href{http://katlas.org/wiki/9_7}{$9_{7}$} &  $2$ & $q^{4} - q^{3} + 2 \, q^{2} - q + 1$ \\\hline
\href{http://katlas.org/wiki/9_9}{$9_{9}$} &  $2$ & $q^{6} - q^{5} + 2 \, q^{4} - 2 \, q^{3} + 2 \, q^{2} - q + 1$ \\\hline
\href{http://katlas.org/wiki/9_10}{$9_{10}$} &  $0$ & $q^{4} - 2 \, q^{3} + 2 \, q^{2} - 2 \, q + 1$ \\\hline
\href{http://katlas.org/wiki/9_13}{$9_{13}$} &  $0$ & $q^{4} - 2 \, q^{3} + 2 \, q^{2} - 2 \, q + 1$ \\\hline
\href{http://katlas.org/wiki/9_16}{$9_{16}$} &  $4$ & $q^{6} - q^{5} + 3 \, q^{4} - 2 \, q^{3} + 3 \, q^{2} - q + 1$ \\\hline
\href{http://katlas.org/wiki/9_18}{$9_{18}$} &  $1$ & $q^{4} - 2 \, q^{3} + 3 \, q^{2} - 2 \, q + 1$ \\\hline
\href{http://katlas.org/wiki/9_23}{$9_{23}$} &  $1$ & $q^{4} - 2 \, q^{3} + 3 \, q^{2} - 2 \, q + 1$ \\\hline
\href{http://katlas.org/wiki/9_35}{$9_{35}$} &  $0$ & $q^{2} - 2 \, q + 1$ \\\hline
\href{http://katlas.org/wiki/9_38}{$9_{38}$} &  $0$ & $q^{4} - 3 \, q^{3} + 4 \, q^{2} - 3 \, q + 1$ \\\hline
\href{http://katlas.org/wiki/9_45}{$9_{45}$} &  $2$ & $2 \, q^{2} - 2 \, q + 2$ \\\hline
\href{http://katlas.org/wiki/9_46}{$9_{46}$} &  $2$ & $2$ \\\hline
\href{http://katlas.org/wiki/9_49}{$9_{49}$} &  $0$ & $q^{4} - 2 \, q^{3} + 2 \, q^{2} - 2 \, q + 1$ \\\hline
\end{tabular}
\end{tabular}
\caption{\label{tab:small_knots} ``Point counts'' of small knots and the associated Catalan numbers.}
\end{table}

\subsection{Affine plabic fences}
\def\tI{{\tilde I}}
While reduced plabic graphs were classified by Postnikov \cite{Pos}, the problem of classifying simple plabic graphs appears much harder.  We consider an affine version of plabic fences, which contains a large subclass of simple plabic graphs that can be parametrized; see \cref{prop:aff_fence_simple}.

Fix $\n \geq 2$. 
Let $\tI:=I\sqcup\{\n\}=\{1,2,\dots,\n-1,\n\}$. A \emph{double affine braid word} is a word $\bw=(i_1,i_2,\dots,i_m)$ in the alphabet
\begin{equation*}%
  \pm \tI:=\{-1,-2,\dots,-\n\}\sqcup\{1,2,\dots,\n\}.
\end{equation*}

\def\fpf{\mathring{S}_\n}
Recall from~\cite{Pos} that for any $(k,\n)$-bounded affine permutation $f \in \Boundkn$, we have a reduced plabic graph $G(f)$, satisfying the conditions of \cref{sec:intro:quivers-point-count}. For each fixed point $i\in[\n]$ of $\pi:=\fbar\in\Sn$, the boundary vertex $i$ of $G(f)$ is adjacent to an interior leaf, called a \emph{lollipop}. This lollipop is white if $f(i)=i+\n$ and black if $f(i)=i$. In particular, if $\pi$ has no fixed points then $G(f)$ has no interior leaves. Recall from \cref{sec:bound-affine-perm} that in this case, $f$ can be recovered from $\pi$, and we denote $G(\pi):=G(f)$.

 Given a double affine braid word $\bw =(i_1,i_2,\dots,i_m)$, and $f \in \Boundkn$, we define the \emph{affine plabic fence} $G(\w,f)$ as follows.  Begin with the plabic graph $G(f)$.  For each $j=m, m-1, \ldots, 1$, let $h:=|i_j|$ and consider the two boundary vertices $h$ and $h+1$, considered modulo $\n$.  If $i_j>0$, add a black at $h$, white at $h+1$ bridge between these two boundary vertices, and if $i_j<0$, add a white at $h$, black at $h+1$ bridge between these two boundary vertices.  When $\pi:=\fbar$ has no fixed points, we denote $G(\w,\pi):=G(\w,f)$. See \cref{fig:APF} for an example. 

\begin{figure}
  \includegraphics[width=\textwidth]{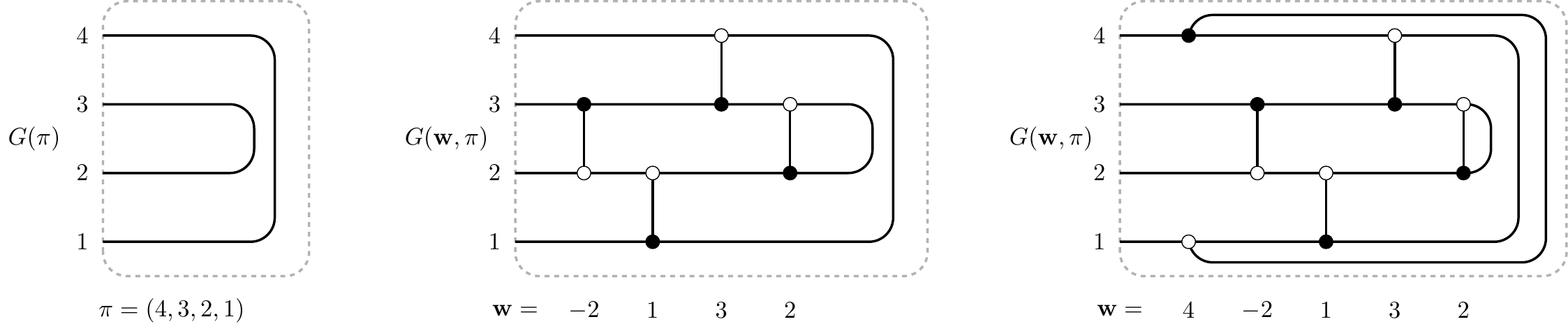}
  \caption{\label{fig:APF} A reduced plabic graph $G(\pi)$ and two affine plabic fences.}
\end{figure}

\begin{proposition}\label{prop:aff_fence_simple}
For any $\bw\in (\pm \tI)^m$ and any permutation $\pi\in\Sn$ without fixed points, the plabic graph $G(\bw,\pi)$ is simple.
\end{proposition}
\begin{proof}
It is clear from the construction that every time we add a bridge to an affine plabic fence, the faces $F$ and $F'$ on the two sides of the bridge are distinct.  Thus, the quiver of $G(\bw,\pi)$ cannot have $1$-cycles.  Let $F, F'$ be two adjacent (interior or boundary) faces of $G(\bw,\pi)$.  If $F$ and $F'$ are both adjacent to the bridge $e$, then $e$ must be the only edge separating $F$ and $F'$ since this is the case if $e$ is the bridge added to $G(f)$.  

In $G(f)$, any two faces have at most one edge in common.  So if $F$ and $F'$ do not have a bridge edge in common in $G(\bw,\pi)$, then their common boundary must belong to a single edge $e'$ of $G(f)$.  The edge $e'$ is divided into a number of segments by the vertices of $G(\bw,\pi)$ that lie on $e'$.  By considering what happens when a single bridge is added, we see that two distinct such segments cannot belong to the boundary of the same two faces of $G(\bw,\pi)$.  It follows that the quiver of $G(\bw,\pi)$ has no directed $2$-cycles.
 \end{proof}
 
Not all simple affine plabic fences are leaf recurrent: for instance, the non-leaf recurrent example in \cref{sec:notleaf} is move equivalent to the example in \figref{fig:APF}(middle).  It would be interesting to find a criterion for an affine plabic fence to be leaf recurrent.

We say that two indices $i,j\in \tI$ are \emph{adjacent} if either $|i-j|=1$ or $\{i,j\}=\{1,\n\}$. Similarly, $i,j\in-\tI$ are \emph{adjacent} if $|i|$ and $|j|$ are adjacent. Consider the following moves on double affine braid words:
\begin{enumerate}[($\tilde{\text{B}}$1)]%
\item\label{abm1} $(\dots,i,j,i,\dots)\leftrightarrow (\dots,j,i,j,\dots)$ \quad if $i,j\in\pmtI$ have the same sign and are adjacent;
\item\label{abm2} $(\dots,i,j,\dots)\leftrightarrow (\dots,j,i,\dots)$ \quad if $i,j\in \pmtI$ have the same sign and are non-adjacent;
\item\label{abm5} $(\dots,i,j,\dots)\leftrightarrow (\dots,j,i,\dots)$  \quad if $i, j\in\pmtI$ have different signs;
\item\label{abm6} $(-i,\dots)\leftrightarrow (i,\dots)$ \quad for $i\in \pmtI$;
\end{enumerate}
  If $\bw,\bw'$ are equivalent using \abmref{abm1}--\abmref{abm6}, then $G(\bw,f)$ and $G(\bw',f)$ are move equivalent. If $\bw,\bw'$ are equivalent using \abmref{abm1}--\abmref{abm5}, then $G(\bw,f)$ and $G(\bw',f)$ are move equivalent without using tail addition/removal, in which case we say that they are \emph{restricted move equivalent}. In particular, the number of boundary vertices is fixed under restricted move equivalence. 
 We now consider the question of the classification of affine plabic fences, up to restricted move equivalence. This is analogous to the classification of equivalence classes of reduced plabic graphs by Postnikov~\cite{Pos}.

\begin{remark}\label{rmk:extra_moves}
For the rest of this subsection, we consider arbitrary bounded affine permutations $f\in\Boundkn$, with no restrictions on fixed points. Consequently, we allow our plabic graphs to have interior vertices of degree $1$, $2$, or $3$. Vertices of degree $2$ can be freely placed on the edges and removed from them. If an interior leaf  is connected to an interior vertex of the same color (as in the right two pictures in \cref{fig:interior_leaves}), the edge connecting them may be contracted. Conversely, we can place a degree two vertex on any edge and attach a leaf of the same color to it. These extra moves are included in the notion of restricted move equivalence.
\end{remark}

We say that $f\in\Boundkn$ \emph{has a left descent at $i\in\Z$} if $i+1\leq f(i)<f(i+1)\leq i+n$. In this case, we let $s_if:\Z\to\Z$ be the $\n$-periodic bijection sending $i\mapsto f(i+1)$, $i+1\mapsto f(i)$, and $j\mapsto f(j)$ for $j\not\equiv i,i+1$ modulo $\n$. Then we have $s_if\in\Boundkn$, and the plabic graph $G(f)$ is obtained from $G(s_if)$ by adding a bridge that is white at $i$ and black at $i+1$. Similarly, $f\in\Boundkn$ \emph{has a right descent at $i\in\Z$} if $i-n\leq f^{-1}(i)\leq f^{-1}(i+1)\leq i$. In this case, we define $fs_i:\Z\to\Z$ to be the $\n$-periodic bijection sending $f^{-1}(i)\mapsto i+1$, $f^{-1}(i+1)\mapsto i$, and $f^{-1}(j)\mapsto j$ for $j\not\equiv i,i+1$ modulo $\n$. We have $fs_i\in\Boundkn$ and the plabic graph $G(f)$ is obtained from $G(fs_i)$ by adding a bridge that is black at $i$ and white at $i+1$. See~\cite[Section~7.4]{LamCDM}. Note that the graphs $G(s_if)$ and $G(fs_i)$ may contain lollipops even when $G(f)$ does not.

We introduce two new moves for affine plabic fences:
\begin{enumerate}[($\tilde{\text{B}}$1)]
  \setcounter{enumi}{4}
\item \label{abm7} $G(\bw,f)\leftrightarrow G((\bw,-i),s_if)$ if $f$ has a left descent at $i$;
\item \label{abm8} $G(\bw,f)\leftrightarrow G((\bw,i),fs_i)$ if $f$ has a right descent at $i$.
\end{enumerate}
Note that these moves respect restricted move equivalence as they can be obtained from the extra plabic graph moves listed in \cref{rmk:extra_moves}.

\begin{conjecture}\label{conj:APF}
Two affine plabic fences $G(\bw,f)$ and $G(\bw',f')$ are restricted move equivalent if and only if they can be related by moves \abmref{abm1}--\abmref{abm5} and \abmref{abm7}--\abmref{abm8}.
\end{conjecture}

\begin{remark}
A strengthening of \cref{conj:APF} can be given when $f=\fkn$ is the ``top cell'' bounded affine permutation (\cref{sec:bound-affine-perm}). Consider the affine braid group with generators $\beta_i^{\pm 1}$ for $i \in I$, lifting the Coxeter generators $s_i$, together with the shift braids that lift the affine permutations $i \mapsto i +k$.  For a word $\bw\in\DRW$, consider an affine braid $\bw_1 \, \fkn \, \bw_2$, where $\bw_1$ (resp., $\bw_2$) is the positive braid lift of the subword of $\bw$ consisting of the letters from $\tI$ (resp., positive braid lift of the reverse of the subword of $\bw$ consisting of letters from $-\tI$). We conjecture that for $\bw,\bw'\in\DRW$, the affine plabic fences $G(\bw,\fkn)$ and $G(\bw',\fkn)$ are restricted move equivalent if and only if the affine braids $\bw_1 \, \fkn \, \bw_2$ and $\bw'_1 \, \fkn \, \bw'_2$ are braid equivalent, or equivalently, the affine braids $\bw_1 (\bw_2+k)$ and $\bw'_1 (\bw'_2+k)$ are braid equivalent, where $(\bw_2+k)$ is the positive braid word obtained by adding $k$ (modulo $\n$) to each letter in $\bw_2$. 
\end{remark}

\subsection{Trees and acyclic quivers}

\begin{definition}
A link $L$ is called \emph{acyclic} (resp. \emph{tree-like}) if there is a simple plabic graph $G$ such that $L =\Lplab_G$, and $\QG$ is mutation acyclic (resp. mutation equivalent to a tree).
\end{definition}

\begin{problem}\
\begin{itemize}
\item
Which acyclic quivers $Q$ can occur as $\QG$ for a simple plabic graph $G$?
\item
Can we characterize acyclic, or tree-like links?
\end{itemize}
\end{problem}

It is clear that any tree $T$ can occur as $\QG$ for some simple plabic graph $G$.  In fact, Lam and Speyer (unpublished) have shown that one can choose such a plabic graph to be reduced.

We expect that acyclic links have particularly simple link invariants, including link homology groups.  Proposition~\ref{prop:acycliccount} gives a closed formula for the top $a$-degree part of the \FLY polynomial of an acyclic link.  It would be interesting to obtain a closed formula for the entire \FLY polynomial.  

In \cite[Theorem 1.3]{LS2}, it is shown that the mixed Hodge numbers of an acyclic really full rank cluster variety satisfy a strong vanishing condition.  It would be interesting to interpret this vanishing via the link homology groups of acyclic links.

\bibliographystyle{alpha_tweaked}
\bibliography{plabic_links}
\end{document}